\documentclass{tac}

\makeatletter

\usepackage[utf8]{inputenc}
\usepackage[T1]{fontenc}
\usepackage[tracking=true]{microtype}
\SetTracking   % Slightly space out small caps, but not too much.
  {encoding=T1, shape=sc}
  {10}
\SetProtrusion % Protrude sub- and superscripts at the right margin
  {encoding=T1,size={7,8}}
  {1={ ,750},2={ ,500},3={ ,500},4={ ,500},5={ ,500},
    6={ ,500},7={ ,600},8={ ,500},9={ ,500},0={ ,500}}
\usepackage{lmodern}   % A slightly improved Computer Modern font

\usepackage{mathtools} % Newer amsmath
\usepackage{enumitem}  % For enumerations in Defs/Thms

\usepackage[square, numbers, sort&compress]{natbib} % Bibliography

\usepackage{aligned-overset} % For oversets in align environments
\usepackage{faktor}          % For quotients of rings %TODO EXAMPLE factor
\usepackage{xfrac}           % \sfrac yields a slanted version of \frac

\usepackage{amssymb} % AMS symbols+ fonts
\usepackage{mathrsfs} % Math script fonts
\usepackage{mathbbol} % BOLDFACE in math environment

\usepackage{pgfplots}
\usepackage{tikz}
\usetikzlibrary{
  backgrounds,
  calc,
  cd,
  decorations.markings,
  external,
  fit,
  scopes,
  intersections,
  pgfplots.fillbetween
}
\usepackage{graphicx} % To include graphics in the document
\usepackage{caption}  % For captions outside of float environments

\usepackage{scalerel} % To set scale objects
\usepackage{etoolbox} % Needed for the compatibility between tikzcd and externalize
\usepackage{xparse}   % Build more complicated macros. Problem right now for example [] vs {}.
\usepackage{xspace}   % Define commands that don't eat spaces

\pgfdeclarelayer{edgelayer}
\pgfdeclarelayer{nodelayer}
\pgfsetlayers{background,edgelayer,nodelayer,main}

\usepackage{hyperref}  % References with links
\hypersetup{
  colorlinks,
  linkcolor={darkblue},
  citecolor={green!65!black},
  urlcolor={blue},
}

\newlist{thmlist}{enumerate}{1}
\setlist[thmlist]{label=(\roman{thmlisti})}

\definecolor{blue}{rgb}{0.38, 0.51, 0.71}
\definecolor{darkblue}{RGB}{17, 42, 60}
\definecolor{pastel blue}{RGB}{94, 149, 174}

\definecolor{pastel purple}{RGB}{102, 41, 163}

\definecolor{red}{RGB}{175, 49, 39}
\definecolor{pastel red}{RGB}{251, 142, 126}

\definecolor{orange}{RGB}{217, 156, 55}
\definecolor{pastel orange}{RGB}{248, 202, 157}

\definecolor{green}{RGB}{146, 227, 95}
\definecolor{pale green}{RGB}{197, 184, 104}
\definecolor{pastel green}{RGB}{197, 215, 192}

\definecolor{yellow}{RGB}{250, 199, 100}
\definecolor{pastel yellow}{RGB}{242, 207, 89}

\definecolor{brokenwhite}{RGB}{218, 192, 166}
\definecolor{brokengrey}{rgb}{0.77, 0.76, 0.82}

\colorlet{qyd-color}{blue}
\definecolor{ayd-color}{rgb}{0.6,0.1,0.2}

\AtBeginEnvironment{tikzcd}{\tikzexternaldisable} % Avoid a conflict with tikz-cd
\AtEndEnvironment{tikzcd}{\tikzexternalenable}

\numberwithin{equation}{section}

\newtheoremrm{remark}[theorem]{Remark}
\newtheoremrm{notation}[theorem]{Notation}
\newtheoremrm{question}[theorem]{Question}
\newtheoremrm{convention}{Convention}
\newtheoremrm{definition}[theorem]{Definition}
\newtheoremrm{example}[theorem]{Example}

\newcommand{\scalemath}[2]{\scalebox{#1}{% Scale math mode equations
    \mbox{\ensuremath{\displaystyle #2}}}%
}

\newcommand{\ie}{i.e.\xspace}

\newcommand{\Benabou}{Bénabou\xspace}
\newcommand{\Bruguieres}{Bruguières\xspace}
\newcommand{\YetterDrinfeld}{Yetter--Drinfeld\xspace}
\newcommand{\YangBaxter}{Yang--Baxter\xspace}
\newcommand{\EilenbergMoore}{Eilenberg--Moore\xspace}
\newcommand{\Malcev}{Mal'cev\xspace}

\newcommand*{\from}{\colon}
\newcommand*{\trightarrow}[2][]{\ensuremath{\xrightarrow[#1]{#2}}} %An arrow pointing to the right which is super and subscribable.
\renewcommand*{\to}{\longrightarrow}
\renewcommand*{\mapsto}{\longmapsto}

\newcommand{\id}{\mathrm{id}} % The identity
\DeclareMathOperator{\ev}{ev} % The evaluation morphism
\DeclareMathOperator{\coev}{coev} %The coevaluation morphism

\newcommand{\dualVee}{\mathchoice
  {\raisebox{0.5 em}{\ensuremath{\scriptscriptstyle{\vee}}}}% \displaystyle
  {\raisebox{0.5 em}{\ensuremath{\scriptscriptstyle{\vee}}}}% \textstyle
  {\raisebox{0.3 em}{\scaleobj{0.8}{\ensuremath{\scriptscriptstyle{\vee}}}}}% \scriptstyle
  {\raisebox{0.3 em}{\scaleobj{0.6}{\ensuremath{\scriptscriptstyle{\vee}}}}}% \scriptscriptstyle
}
\newcommand{\lDual}[1]{{#1}\dualVee}
\newcommand{\lbiDual}[1]{{#1}\dualVee\dualVee}

\newcommand{\rDual}[1]{\dualVee{#1}}

\newcommand*{\defeq}{\mathrel{\vcenter{\baselineskip0.5ex \lineskiplimit0pt
      \hbox{\scriptsize.}\hbox{\scriptsize.}}}%
  =}
\newcommand{\eqdef}{\eqqcolon}

\newcommand{\lact}{\mathchoice
  {\raisebox{0.065 em}{\,\scaleobj{0.75}{\rhd}}\,}% \displaystyle
  {\raisebox{0.065 em}{\,\scaleobj{0.75}{\rhd}}\,}% \textstyle
  {\raisebox{0.04 em}{\scaleobj{0.5}{\rhd}}}% \scriptstyle
  {\raisebox{0.043 em}{\scaleobj{0.33}{\rhd}}}% \scriptscriptstyle
}
\newcommand{\ract}{\mathchoice
  {\raisebox{0.065 em}{\,\scaleobj{0.75}{\lhd}}\,}% \displaystyle
  {\raisebox{0.065 em}{\,\scaleobj{0.75}{\lhd}}\,}% \textstyle
  {\raisebox{0.04 em}{\scaleobj{0.5}{\lhd}}}% \scriptstyle
  {\raisebox{0.043 em}{\scaleobj{0.33}{\lhd}}}% \scriptscriptstyle
}

\newcommand{\bract}{\mathchoice
  {\raisebox{0.065 em}{\scaleobj{0.75}{\blacktriangleleft}}}% \displaystyle
  {\raisebox{0.065 em}{\scaleobj{0.75}{\blacktriangleleft}}}% \textstyle
  {\raisebox{0.04 em}{\scaleobj{0.5}{\blacktriangleleft}}}% \scriptstyle
  {\raisebox{0.043 em}{\scaleobj{0.33}{\blacktriangleleft}}}% \scriptscriptstyle
}
\DeclareMathOperator{\Homspace}{Hom} % The set of homomorphisms
\DeclareMathOperator{\Endspace}{End} % The set of endomorphisms
\DeclareMathOperator{\Autspace}{Aut} % The set of endomorphisms
\DeclareMathOperator{\Isospace}{Iso} % The set of (natural) isomorphisms
\DeclareMathOperator{\Inv}{Inv} % The set of invertible objects
\DeclareMathOperator{\Pic}{Pic} % The Picard group/heap of a rigid category
\DeclareMathOperator{\Piv}{Piv} % The pivotal structures of a rigid category.
\DeclareMathOperator{\QPiv}{QPiv} % The set of isomorphism classes of quasi-pivotal structures. \\
\DeclareMathOperator{\Sym}{Sym}

\DeclareMathOperator{\Gr}{Gr} % Group of group-likes
\DeclareMathOperator{\Char}{Char} % Group of characters

\DeclareMathOperator{\op}{op} % The opposite
\DeclareMathOperator{\U}{U} %Universal enveloping

\newcommand{\cat}[1]{\ensuremath{\mathcal{#1}}} % A category
\newcommand{\Hom}[3]{\ensuremath{\Homspace_{#1}({#2},{#3})}} %The set of homomorphisms
\newcommand{\End}[2]{\ensuremath{\Endspace_{#1}({#2})}} %The set of endomorphisms
\newcommand{\Aut}[2]{\ensuremath{\Autspace_{#1}({#2})}} %The set of automorphisms
\DeclareMathOperator{\Nat}{Nat} %The class of natural transformations
\newcommand{\Id}{\mathrm{Id}} % The identity functor
\DeclareMathOperator{\Ad}{Ad}

\newcommand{\adjoint}{\dashv}
\newcommand{\cocolon}{\nobreak \mskip6mu plus1mu \mathpunct{} \nonscript \mkern-\thinmuskip {:} \mskip2mu \relax} % Like \colon, but for the other side.
\newcommand{\adj}[4]{\ensuremath{#1 \colon #3 \rightleftarrows #4 \cocolon #2}\xspace} % \adj{F}{U}{C}{D} produces F : C <-> D: U
\newcommand{\stdadj}{\adj{F}{U}{\cat C}{\cat D}} %shorthand or the standard free forgetful adjunction

\DeclareMathOperator{\tensorop}{\ensuremath{\otimes}-op} % The tensor opposite

\newcommand{\rMod}[1]{\ensuremath{\cat{M}_{#1}}} % Caligraphy version for right modules
\NewDocumentCommand{\abstractCentFactory}{m o o m o}{
  \IfBooleanTF{#1}
  {\ensuremath{\mathsf{#2}({}_{#3}{#4}_{#5})}}
  {\ensuremath{\mathsf{#2}({}_{#3}{\cat{#4}}_{#5})}}
}
\NewDocumentCommand{\ZCat}{s O{} m O{}}{\abstractCentFactory{#1}[Z][#2]{#3}[#4]} % The center of a category
\NewDocumentCommand{\ACat}{s O{} m O{}}{\abstractCentFactory{#1}[A][#2]{#3}[#4]} % Right twisted center
\NewDocumentCommand{\QCat}{s O{} m O{}}{\abstractCentFactory{#1}[Q][#2]{#3}[#4]} % Left twisted center
\NewDocumentCommand{\SZCat}{s O{} m O{}}{\abstractCentFactory{#1}[SZ][#2]{#3}[#4]} % The center of a category
\DeclareMathOperator{\Com}{Com} % The category of comodule endofunctors

\DeclareMathSymbol{\mathhyphen}{\mathord}{operators}{`\-}% Short hyphen for math mode
\newcommand{\blank}{{-}} %blank symbol in math mode.

\tikzset{%
  baseline=(current bounding box.center)
}
\tikzset{%
  label distance=-0.15cm
}
\tikzset{
  font=\scriptsize
}
\tikzset{
  >=stealth
}
\tikzset{
  dot diameter/.store in=\dot@diameter,
  dot diameter=2pt,
  dot spacing/.store in=\dot@spacing,
  dot spacing=10pt,
  dots/.style={
    line width=\dot@diameter,
    line cap=round,
    dash pattern=on 0pt off \dot@spacing
  }
}

\tikzstyle{none} = [inner sep=3pt,outer sep=0pt]

\tikzstyle{morphism} = [rectangle, fill=white, draw=black, line width=1pt, font=\scriptsize]
\tikzstyle{intersection-pt} = [fill=white, inner sep = 2.5pt]
\tikzstyle{fullblackdot}=[fill=black, draw, shape=circle, scale=0.7]
\tikzstyle{blackdot}=[fill=white, draw, shape=circle, scale=0.3]
\tikzstyle{PP}=[draw={rgb,255:red,102;green,41;blue,163}, shape=circle, fill={rgb,255:red,102;green,41;blue,163}, scale=0.5]

\tikzstyle{comodule-edge}=[-, draw=qyd-color,very thick]
\tikzstyle{ddd}=[-, draw=black, dash dot dot, very thick]
\tikzstyle{unit}=[-, draw=black, very thick, densely dotted]
\tikzstyle{morphism-edge} = [very thick, black, decoration = {markings, mark=at position 0.5 with {\arrow{>}}}]
\tikzstyle{directed} = [morphism-edge, postaction=decorate]
\tikzstyle{inclusion} = [very thick, right hook->,midway,right]

\tikzstyle{Front}=[-, draw=black, fill=white, opacity=0.5]
\tikzstyle{Hidden}=[-, draw=black, fill=white, opacity=0.5]

\tikzstyle{AQEval} = [
  postaction={
    decorate,
    decoration={
      markings,
      mark=between positions 0 and \pgfdecoratedpathlength step 0.5pt with {
        \pgfmathsetmacro\myval{multiply(
          divide(
          \pgfkeysvalueof{/pgf/decoration/mark info/distance from start}, \pgfdecoratedpathlength
          ),
          100
          )}
        \pgfsetfillcolor{qyd-color!\myval!ayd-color}
        \pgfpathrectangle{\pgfpoint{-0.6pt}{-0.6pt}}{\pgfpoint{1.2pt}{1.2pt}}
        \pgfusepath{fill}}
    }}]

\tikzstyle{QAEval} = [
  postaction={
    decorate,
    decoration={
      markings,
      mark=between positions 0 and \pgfdecoratedpathlength step 0.5pt with {
        \pgfmathsetmacro\myval{multiply(
          divide(
          \pgfkeysvalueof{/pgf/decoration/mark info/distance from start}, \pgfdecoratedpathlength
          ),
          100
          )}
        \pgfsetfillcolor{ayd-color!\myval!qyd-color}
        \pgfpathrectangle{\pgfpoint{-0.6pt}{-0.6pt}}{\pgfpoint{1.2pt}{1.2pt}}
        \pgfusepath{fill}}
    }}]

\tikzstyle{AQCoeval} = [
  postaction={
    decorate,
    decoration={
      markings,
      mark=between positions 0 and \pgfdecoratedpathlength step 0.5pt with {
        \pgfmathsetmacro\myval{multiply(
          divide(
          \pgfkeysvalueof{/pgf/decoration/mark info/distance from start}, \pgfdecoratedpathlength
          ),
          100
          )}
        \pgfsetfillcolor{qyd-color!\myval!ayd-color}
        \pgfpathrectangle{\pgfpoint{-0.6pt}{-0.6pt}}{\pgfpoint{1.2pt}{1.2pt}}
        \pgfusepath{fill}}
    }}]

\tikzstyle{QACoeval} = [
  postaction={
    decorate,
    decoration={
      markings,
      mark=between positions 0 and \pgfdecoratedpathlength step 0.5pt with {
        \pgfmathsetmacro\myval{multiply(
          divide(
          \pgfkeysvalueof{/pgf/decoration/mark info/distance from start}, \pgfdecoratedpathlength
          ),
          100
          )}
        \pgfsetfillcolor{ayd-color!\myval!qyd-color}
        \pgfpathrectangle{\pgfpoint{-0.6pt}{-0.6pt}}{\pgfpoint{1.2pt}{1.2pt}}
        \pgfusepath{fill}}
    }}]
\makeatother
\pdfoutput=1

\author{Sebastian Halbig and Tony Zorman}
\address{Technische Universität Dresden, Institut f\"ur
  Geometrie, Zellescher Weg 12--14, 01062 Dresden}
\eaddress{sebastian.halbig@tu-dresden.de \CR tony.zorman@tu-dresden.de}

\date{\today}
\title{Pivotality, twisted centres and the anti-double of a Hopf monad}

\copyrightyear{2022}

\thanks{We would like to thank Ilya Shaprio for many fruitful discussions in the early stages of this project and Kenichi Shimizu for helpful remarks concerning symmetric centres and connections to topological quantum field theories. We further thank the anonymous referee for many useful remarks, suggestions, and improvements.}

\begin{document}
\maketitle

\begin{abstract}
  Finite-dimensional Hopf algebras admit a correspondence between so-called pairs in involution, one-dimensional anti-\YetterDrinfeld modules and algebra isomorphisms between the Drinfeld and anti-Drinfeld double.
  We extend it to general rigid monoidal categories and provide a monadic interpretation under the assumption that certain coends exist.
  Hereto we construct and study the anti-Drinfeld double of a Hopf monad.
  As an application the connection with the pivotality of Drinfeld centres and their underlying categories is discussed.
\end{abstract}

\microtypesetup{protrusion=false}
\tableofcontents
\microtypesetup{protrusion=true}

\section{Introduction}\label{sec: Intro}
The aim of this paper is to study the relationship between the Drinfeld centre of a monoidal category and a `twisted' version of it, which arises in the study of Hopf cyclic cohomology.
Our approach splits into two parts.
First, we deploy general categorical tools in order to identify equivalences of the aforementioned categories with `invertible' objects in a twisted centre.
Second, we take the monadic point of view and explain which of these equivalences translate into isomorphisms of monads generalising the Drinfeld and anti-Drinfeld double.
As a byproduct we show that a rigid category which admits these monads is pivotal if and only if the generalised double and anti-double are isomorphic.

\subsection*{The Hopf algebraic case}

Our goal is best explained by first recalling the interactions between the various objects and categories in the setting of finite-dimensional Hopf algebras.
This is covered in greater detail in \cite{Halbig2019}.

A peculiarity of the Hopf cyclic cohomology, as defined by Connes and Moscovici \cite{Connes1999}, is the lack of `canonical' coefficients.
Originally, see \cite{Connes2000}, modular pairs in involution were considered.
These consist of a group-like and a character implementing the square of the antipode by their respective adjoint actions.
Later, Hajac et.al.~obtained a quite general source for coefficients in what they called the category of anti-\YetterDrinfeld modules, \cite{Hajac2004a}.
Their name is due to the similarity with \YetterDrinfeld modules:
Like their well-known `cousins', they are simultaneously modules and comodules satisfying a compatibility condition between the action and coaction.
In general, they do not form a monoidal category but a module category over the \YetterDrinfeld modules.
This is reflected by the fact that they can be identified with the modules over the anti-Drinfeld double, a comodule algebra over the Drinfeld double.
The special role of pairs in involution is captured by the following theorem due to Hajac and Sommerhäuser:

\begin{theorem}[{\cite[Theorem~3.4]{Halbig2019}}]\label{thm: equiv_in_Hopf_setting}
  For any finite-dimensional Hopf algebra \(H\) the following statements are equivalent:
  \begin{thmlist}
    \item The Hopf algebra \(H\) admits a pair in involution.
    \item There exists a one-dimensional anti-\YetterDrinfeld module over \(H\).
    \item The Drinfeld double and anti-Drinfeld double of \(H\) are isomorphic algebras.
  \end{thmlist}
\end{theorem}

Furthermore, these pairs are of categorical interest as they give rise to pivotal structures on the \YetterDrinfeld modules.
That is, they provide a natural monoidal isomorphism between each object and its double dual.

\subsection*{Twisted centres and pivotality}

We want to reformulate this theorem in a categorical framework with an emphasis on pivotal structures.

First, let us discuss appropriate replacements for the concepts described above.
The role of the Hopf algebra is taken by a rigid monoidal category \(\cat C\).
Roughly speaking, that means a category with a suitably associative and unital product in which every object has a left and right dual.
Due to the monoid-like nature of \(\cat C\), we can study its bimodule categories.
Of special interest is the regular bimodule, whose actions are given by respectively `multiplying' from the left or right.
Its centre  \(\ZCat C\), called the Drinfeld centre of \(\cat C\), provides us with an analogue of the category of \YetterDrinfeld modules, see \cite[Chapter XIII]{Kassel1998}.
Anti-\YetterDrinfeld modules were generalised in
\cite{Hassanzadeh2019} to what one might call the anti-Drinfeld centre
\(\ACat C\) of \(\cat C\).
In the context of topological quantum field theories this type of
category was studied for example in \cite{Fuchs2017} and \cite{Douglas2021}.
As in the Hopf algebraic case, it is a module category over \(\ZCat C\).
An adaptation of pairs in involution are, what we will call, quasi-pivotal structures, studied for example in \cite{Shimizu2016}.
They consist of an invertible object, which replaces the character, and, instead of a group-like element, a certain natural monoidal isomorphism.

The main observation needed to generalise Theorem~\ref{thm: equiv_in_Hopf_setting} is that the anti-Drinfeld centre admits a `dual'.
In Proposition~\ref{prop: module_funs_have_adjoints} this allows us to identify equivalences of \(\ZCat C\) modules between \(\ZCat C\) and \(\ACat C\) with so-called $\cat{C}$-invertible objects in \(\ACat C\).
These are objects in $\ACat C$ whose image under the forgetful functor to $\cat{C}$ is invertible.
Subsequently, we prove that these objects correspond to quasi-pivotal structures on \(\cat C\) and obtain the categorical version of Theorem~\ref{thm: equiv_in_Hopf_setting} as Theorem~\ref{thm: equivalence_in_the_rigid_setting}, which can be summarised as follows.

\begin{theorem}\label{thm: equivalence_in_the_rigid_setting_intro}
  Let \(\cat C\) be a rigid monoidal category. The following are equivalent:
  \begin{thmlist}
    \item The category \(\cat C\) is quasi-pivotal.
    \item There exists a $\cat{C}$-invertible object in \(\ACat C\).
    \item The Drinfeld and anti-Drinfeld centre of \(\cat C\) are equivalent module categories.
  \end{thmlist}
\end{theorem}

The pivotal structures of the Drinfeld centre \(\ZCat C\) of a finite tensor category \(\cat C\) were studied by Shimizu in \cite{Shimizu2016}.
We contribute to these results with the following observations:
the set \(\Pic \ACat C\) of isomorphism classes of $\cat{C}$-invertible objects in \(\ACat C\) forms a heap, see Lemma~\ref{lem: inv_twisted_form_heap}.
That is, it behaves like a group but without a fixed neutral element.
Note that this provides a parallel with the aforementioned fact that Hopf cyclic cohomology has no canonical coefficients.
Equipping the set of pivotal structures \(\Piv \ZCat C\) of \(\ZCat C\) with the same algebraic structure, we  construct a heap morphism
\(\kappa \from \Pic \ACat C \to \Piv \ZCat C\).
In general, we cannot expect \(\kappa\) to be injective.
Instead, we show that invertible objects in the centre \(\ZCat C\) which admit a `trivial' braiding act nicely on \(\Pic \ACat C\) and the induced pivotal structures only depend on the orbit of this action.
The orbits themselves form a quotient heap and the induced morphism
\begin{equation*}
  \iota \from \Pic \ACat C/_{\sim} \to \Piv \ZCat C,
\end{equation*}
is injective, see Theorem~\ref{thm: equiv_establishes_piv}.
In many cases, such as \(\cat C\) being a finite tensor category, it is moreover surjective.
However, by constructing a counterexample, we show in Theorem~\ref{thm: non_induced_pivotal_structure} that this is not true in general.

\subsection*{Reconstruction: Comodule monads}

To reconcile our results with the initial Hopf algebraic formulation, we provide a monadic interpretation under the assumption that certain coends exist.

The starting point for our considerations is a Hopf monad \(H \from \cat V \to \cat V\) on a rigid, possibly pivotal, category \(\cat V\) which we think of as a replacement of finite-dimensional vector spaces.
Its modules form a rigid monoidal category \(\cat V^H\).
Utilising the centralisers of Day and Street, \cite{Day2007}, \Bruguieres and Virelizier described in \cite{Bruguieres2012} the Drinfeld double \(D(H)\) of \(H\).
It is obtained through a two-step process.
First, the central Hopf monad on \(\cat V^H\) is defined.
Then, the double \(D(H)\from \cat V \to \cat V\) arises by applying a variant of Beck's theorem of distributive laws to it.
As in the classical setting, the modules of \(D(H)\) are isomorphic as a braided rigid monoidal category to the Drinfeld centre \(\ZCat*{\cat V^H}\).
By adapting the procedure outlined above for our purposes, we construct the anti-central monad and derive the anti-Drinfeld double \(Q(H) \from \cat V \to \cat V\) of \(H\) from it.
It is a comodule monad over \(D(H)\) in the sense of \cite{Aguiar2012} which implements the `dual' of the anti-centre \(\QCat{\cat V^H}\) as a module category.
Having all ingredients assembled, we show in Proposition~\ref{prop: functors_as_morphisms_of_monads} that certain module equivalences between \(\ZCat*{\cat V^H}\) and \(\QCat{\cat V^H}\) materialise as isomorphisms between their associated monads.
Applying our general categorical results to \(\cat V^H\) and combining it with a monadic version of pairs in involution, we obtain in Theorem~\ref{thm: equivalence_in_the_monad_setting} an almost verbatim translation of Theorem~\ref{thm: equiv_in_Hopf_setting}:

\begin{theorem}\label{thm: equivalence_in_the_monad_setting_intro}
  Let \(H\) be a Hopf monad on a pivotal category \(\cat V\) that admits a Drinfeld and anti-Drinfeld double. The following are equivalent:
  \begin{thmlist}
    \item The Hopf monad \(H\) admits a pair in involution.
    \item There exists a module over \(Q(H)\) whose underlying object is \(1 \in \cat V\).
    \item The Drinfeld and anti-Drinfeld double of \(H\) are isomorphic as monads.
  \end{thmlist}
\end{theorem}

An immediate consequence of the above result is the observation that pivotal structures on \(\cat V^H\) equate to isomorphisms between the central and anti-central monads, see Corollary~\ref{cor: pivotal_from_central_anti_central}.

\subsection*{Outline}

The article is divided into two parts comprising Sections \ref{sec: shades_of_mon_cats}, \ref{sec: heaps} and \ref{sec: pivotality_of_the_df_centre} as well as Sections \ref{sec: monads} and \ref{sec: mon_perspective}.
We give an overview of some categorical tools for our study in Section~\ref{sec: shades_of_mon_cats}.
In Section~\ref{sec: heaps}, we recall the concept of heaps.
Section~\ref{sec: pivotality_of_the_df_centre} starts with a discussion about twisted centres and their Picard heaps, before studying the notion of quasi-pivotality and establishing the   categorical version of the correspondence given in Theorem~\ref{thm: equiv_in_Hopf_setting}.
Section~\ref{sec: monads} provides an overview of the theory of Hopf monads and comodule monads.
In Section~\ref{sec: mon_perspective} the central and anti-central
monad are constructed and from them the Drinfeld and anti-Drinfeld
double. By expressing our abstract categorical findings in the monadic
language we then obtain Theorem~\ref{thm: equivalence_in_the_monad_setting_intro} and comment on how it can be
used to detect pivotal structures.

%%% Local Variables:
%%% mode: latex
%%% TeX-master: "../main"
%%% End:

\section{Monoidal categories, bimodule categories and the centre construction} \label{sec: shades_of_mon_cats}
We assume the readers familiarity with standard concepts of (monoidal) category theory, as given for example in \cite{MacLane1998,Leinster2014,Etingof2015,Riehl2017}.
As a convention, the set of morphisms between two objects \(X, Y \in \cat C\) of a category \(\cat C\) will be written as \(\cat C(X, Y)\).
We will denote the composition of two morphisms \(g \in \cat C(X, Y)\) and \(f \in \cat C(W, X)\) by the concatenation \(gf \defeq g \circ f \in \cat C(W, Y)\).
Adjunctions play an important role in our investigation.
A \emph{right adjoint} of a functor \(F \from \cat C \to \cat D\) is a functor \(U \from \cat D \to \cat C\) together with two natural transformations \(\eta \from {\Id_{\cat C}} \to UF\) and \(\epsilon \from FU \to \Id_{\cat D}\), called the \emph{unit} and \emph{counit} of the adjunction, satisfying the \emph{snake identities}.
These conditions determine \(U\from \cat D \to \cat C\) uniquely up to natural isomorphism.
We write \stdadj or \(F \adjoint U\).

To navigate the proverbial `sea of jargon', \cite{Baez2011}, we provide the reader with a table, inspired by \cite[Figure 2]{Henriques2016}, in order to help us outline the main topics we are about to encounter in this section.

\begin{center}
  \tikzsetnextfilename{flavours_of_monoidal_categories}%
  \input{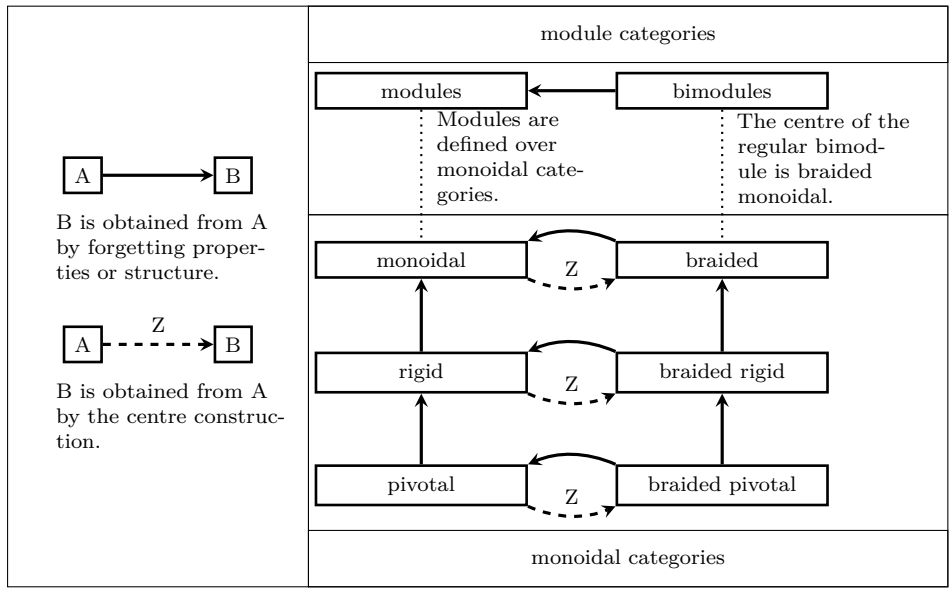}%

  \captionof{figure}{Various types of monoidal and module categories, as well as (some) relations between them.} \label{fig: vars_of_mon_cats}
\end{center}

In Subsection~\ref{subsec: from_mon_cats_to_piv_cats} we work our way down the first column, encountering monoidal, rigid and pivotal categories.
This is based on \citep[Chapter 2]{Etingof2015}.
The concept of braided monoidal categories, responsible for the second column, is discussed in Subsection~\ref{subsec: braided_and_ribbon_cats}.
See \citep[Chapter 8]{Etingof2015} for a reference.
Our approach to module categories, see Subsection~\ref{subsec: bimodule_cats_and_centre_constr}, is derived from \citep[Chapter 7]{Etingof2015}.
We pay special attention to the (Drinfeld) centre construction, responsible for the arrows labelled with a `Z', in Figure~\ref{fig: vars_of_mon_cats}.

\subsection{From monoidal to pivotal categories}\label{subsec: from_mon_cats_to_piv_cats}

Monoidal categories were introduced independently by Mac Lane, \cite{MacLane1963}, and \Benabou, \cite{Benabou1963}, under the name `categories with multiplication'.\footnote{
  Parts of the historical development of the study of monoidal categories is sketched in \cite{Street2012} and, to a lesser extend, in \cite{Baez2011}.
}
The prime examples we draw our inspiration from are finite-dimensional modules over Hopf algebras or, more generally, finite tensor categories, see \cite[Chapters 5 and 6]{Etingof2015}.

\subsubsection{Monoidal categories, functors, and natural transformations}\label{subsubsec: mon_cats}

\begin{definition}\label{def: mon_cat}
  A \emph{strict monoidal category} is a triple \((\cat C, \otimes, 1)\) comprising
  a category \(\cat C\),
  a bifunctor \(\otimes\from \cat C \times \cat C \to \cat C\), called the \emph{tensor product}, and
  an object \(1 \in \cat C\), the \emph{unit},
  satisfying associativity and unitality in the sense that
  \begin{equation}
    (\blank \otimes \blank) \otimes \blank = \blank \otimes ( \blank \otimes \blank) \qquad \text{and} \qquad 1 \otimes \blank = \Id_{\cat C} = \blank \otimes 1.
  \end{equation}
\end{definition}

Many natural examples of monoidal categories, such as the category of vector spaces, are not strict.
That is, the associativity and unitality of the tensor product only hold up to (suitably coherent) natural isomorphisms.
However, we can compensate this by \emph{Mac Lane's strictification theorem}.
It states that any monoidal category is, in a `structure preserving manner', equivalent to a strict one.
A proof is given for example in \cite[Theorem~2.8.5]{Etingof2015}.
Thus, we shall only consider strict monoidal categories from now on, and will omit the prefix `strict'.

The next definition slightly extends the scope of \cite{Etingof2015} but is standard in the literature, see for example \cite[Chapter~3]{Aguiar2010}.

\begin{definition}\label{def: mon_functor}
  An \emph{oplax monoidal functor} between monoidal categories \((\cat C, \otimes, 1)\) and \((\cat C', \otimes', 1')\) is a functor \(F\from \cat C \rightarrow \cat C'\) together with a natural transformation
  \begin{equation*}
    \Delta_{X, Y} \from F(X \otimes Y) \to F(X) \otimes' F(Y), \qquad \qquad \text{ for all } X, Y \in \cat C,
  \end{equation*}
  and a morphism \(\varepsilon \from F(1)\rightarrow 1'\) satisfying \emph{coassociativity} and \emph{counitality}.

  If the \emph{coherence morphisms}, \(\Delta\) and \(\varepsilon\), are isomorphisms or identities, we call \(F\) \emph{(strong) monoidal} or \emph{strict monoidal}, respectively.
\end{definition}

We think of an oplax monoidal functor \((F, \Delta, \varepsilon)\) as a generalisation of a coalgebra.
To emphasise this point of view, we refer to \(\Delta\) and \(\varepsilon\) as the \emph{comultiplication} and \emph{counit} of \(F\). The dual concept is that of a \emph{lax monoidal functor}, which resembles the notion of an algebra.

Assume \(F \from \cat C \to \cat D\) to be strong monoidal and an equivalence of categories.
Its quasi-inverse \(G\from \cat D \to \cat C\) can be turned into a monoidal functor such that the natural isomorphisms \(FG \rightarrow \Id_\cat D\) and \(GF \rightarrow \Id_\cat C\) are compatible with the monoidal structure in a sense we will explain in the next definition.
This justifies calling \(F\) a \emph{monoidal equivalence}.

\begin{definition}\label{def: mon_nat_trafo}
  An \emph{oplax monoidal natural transformation} between oplax monoidal functors \(F, G \from {\cat C} \to {\cat C'}\) is a natural transformation \(\rho \from F\to G\)
  such that for all \(X, Y \in \cat C\)
  \begin{equation}
    \Delta_{X, Y}^{(G)} \rho_{X\otimes Y} = (\rho_X \otimes' \rho_Y) \Delta_{X,Y}^{(F)} \qquad \text{and} \qquad \varepsilon^{(G)} \rho_1 = \varepsilon^{(F)}.
  \end{equation}
  If \(\rho\) is additionally a natural isomorphism, we call it an \emph{oplax monoidal natural isomorphism}.
\end{definition}

In case we want to emphasise that the underlying functors of an oplax monoidal natural transformation \(\rho \from F \to G\) are strong or strict monoidal, we replace the prefix `oplax' with  either `strong' or `strict'.

Adjunctions between monoidal categories are a broad topic with many facets, see \cite[Chapter~3]{Aguiar2010}.
For our purposes, we can restrict ourselves to the following situation.

\begin{definition}\label{def: monoidal_adjunction}
  We call an adjunction \(\stdadj\) between monoidal categories \(\cat C\) and \(\cat D\) \emph{oplax monoidal} if \(F\) and \(U\) are oplax monoidal functors and the unit and counit of the adjunction are oplax monoidal natural transformations.\footnote{%
    In this case, by e.g.~\cite[Lemma~7.10]{Turaev2017}, the functor \(U\) turns out to be strong monoidal.%
  }
  If \(F\) and \(U\) are moreover strong monoidal, we call \(\stdadj\) a \emph{(strong) monoidal adjunction}.
\end{definition}

An efficient means for computations in strict monoidal categories are \emph{string diagrams}.
They consist of strings labelled with objects and vertices between the strings labelled with morphisms.
If two string diagrams can be transformed into each other, the morphisms that they represent are equal.
A more detailed description is given in \cite{joyal91,Selinger2011}.
Our convention is to read diagrams from top to bottom and left to right.
Taking tensor products is depicted by gluing diagrams together horizontally; composition equates to gluing vertically.
Identity morphisms are given by unlabelled vertices. The unit object is represented by the empty edge.

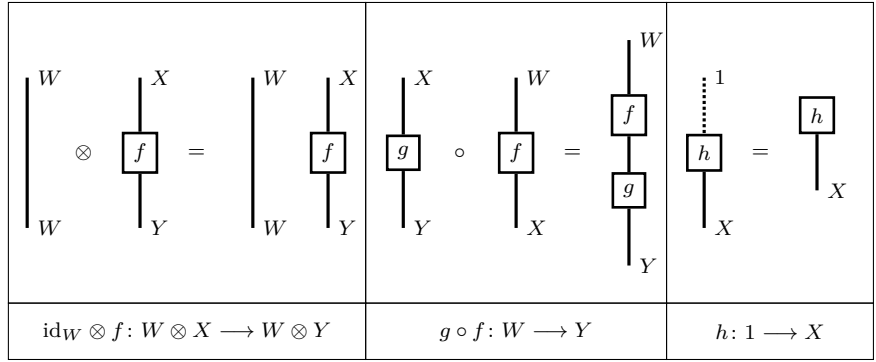
\begin{center}
  \tikzsetnextfilename{explanation_string_diagrams}%
  \input{sections/tikzfigures/explanation_string_diagrams.tikz}%

\end{center}

\subsubsection{Rigidity and pivotality}\label{subsubsec: rigid_and_piv_cats}

Rigidity\footnote{%
    Rigid categories are also called \emph{autonomous},
    or, in the symmetric case, \emph{compact closed}.%
  }
in the context of monoidal categories refers to a concept of duality similar to that of finite-dimensional vector spaces.
Importantly, notions like dual bases and evaluations have their analogues in this setting.
If, moreover, there exists an identification between objects and their double duals that is compatible with the tensor product, the category is called pivotal\footnote{%
  Pivotal categories are also known as \emph{balanced} or \emph{sovereign} categories.%
}.
The more refined notion of spherical categories is not discussed here.
For a treatment in the context of Hopf algebras we refer to the articles \cite{Barrett1999} and \cite{Andruskiewitsch2014}. Examples of duality  inspired by topology are discussed in \cite{Dold1980}.

\begin{definition}\label{def: rigid_cat}
  A \emph{left dual} of an object \(X\in \cat C\) in a monoidal category \(\cat{C}\) is a triple
  \((\lDual{X}, \ev_X^l, \coev_X^l)\) comprising an object \(\lDual{X}\) and two morphisms
  \begin{equation}\label{eq: def_eq_of_left_(co)ev}
    \ev_X^l \from \lDual X \otimes X \to 1 \qquad \text{ and } \qquad
    \coev_X^l \from 1 \to X \otimes \lDual X,
  \end{equation}
  called the left \emph{evaluation} and \emph{coevaluation} of \(X\), such that the \emph{snake identities}
  \begin{subequations}\label{eq: snake_identities}
    \begin{align}
      \id_X &
              = (\id_X \otimes \ev_X^l)
              (\coev_X^l\otimes \id_X)
              \qquad \qquad \text{ and } \label{eq: left_snake_identity}
      \\
      \id_{\lDual X} &
                       = (\ev_X^l\otimes \id_{\lDual X})
                       (\id_{\lDual X}\otimes \coev_X^l) \label{eq: right_snake_identity}
    \end{align}
  \end{subequations}
  hold.
  A \emph{right dual} of \(X\) is a triple
  \((\rDual X, \ev_X^r, \coev_X^r)\) consisting of an object \(\rDual X\) and a right evaluation and coevaluation,
  \(\ev_X^r \from X \otimes \rDual X \to 1\) and \(\coev_X^r \from 1 \to \rDual X \otimes X\),
  subject to analogous identities.

  We call \(\cat C\) a \emph{rigid category} if every object has a left dual and a right dual.
\end{definition}

Left and right duals are unique up to unique isomorphism.
We fix a choice of duals for every object in a rigid category \(\cat C\) and speak of \emph{the} left or right dual in the following.
Graphically, we represent evaluations and coevaluations by semicircles, possibly decorated with arrows if we want to emphasise whether we consider their left or right version.
\begin{center}
  \tikzsetnextfilename{coevs_and_evs}%
  \input{sections/tikzfigures/coevs_and_evs.tikz}%

\end{center}

\begin{definition}\label{def: invertible_objects}
  An object \(X \in \cat C\) in a rigid category \(\cat C\) is called \emph{invertible} if its (left) evaluation and coevaluation are isomorphisms.
\end{definition}

It is an illustrative exercise to show that the right evaluations and coevaluations of an invertible objects must be isomorphisms as well.
Tensor products and duals of invertible objects are invertible too.
Hence, we can consider the full and rigid subcategory \(\Inv(\cat{C}) \subseteq \cat{C}\) of invertible objects of \(\cat C\).

\begin{definition}[{\cite{castelnuovo1905}}]\label{def: invertible_object}
  The \emph{Picard group} \(\Pic \cat C\) of a rigid category \(\cat C\) is the group of isomorphism classes of invertible objects in \(\cat C\).
  Its multiplication is induced by the tensor product of \(\cat C\), \ie \([\alpha] \cdot [\beta] \defeq [\alpha \otimes \beta]\) for \(\alpha, \beta \in \Inv(\cat C)\).
  The unit of \(\Pic \cat C\) is \([1]\) and for any \(\alpha \in \Inv(\cat C)\) we have \([\alpha]^{-1} =[\lDual \alpha]\).
\end{definition}

The next proposition will play a central role in our studies.
To formulate it, we introduce for any \(X \in \cat C\) and \(n \in \mathbb Z\) the shorthand-notation
\begin{equation}\label{eq: interated_duals_of_an_object}
  (X)^{n} \defeq
  \begin{cases}
    \text{The \(n\)-fold left dual of \(X\)}\quad & \text{if } n > 0, \\
    X & \text{if } n = 0, \\
    \text{The \(n\)-fold right dual of \(X\)}\quad & \text{if } n < 0.
  \end{cases}
\end{equation}

\begin{proposition}\label{prop: rigidity_yields_adjoints}
  For every object \(X \in \cat C\) in a rigid category \(\cat C\) we obtain two chains of adjoint endofunctors of \(\cat C\):
  \begin{align}
    &\dots \adjoint
      \blank \otimes (X)^{-1} \;\adjoint\; \blank \otimes X \;\adjoint\; \blank \otimes (X)^1
      \adjoint \dots \text{ and }
      \label{eq: adjoint_functors_left_dual} \\
    &\dots
      \adjoint (X)^1 \otimes \blank \;\adjoint\; X \otimes \blank \;\adjoint\; (X)^{-1} \otimes \blank
      \;\adjoint\; \dots
      \label{eq: adjoint_functors_right_dual}
  \end{align}
  Furthermore, \(\blank \otimes X\) (or \(X \otimes \blank\)) are equivalences of categories if and only if \(X\) is invertible.
\end{proposition}

\begin{proof}
  The existence of the stated chains of adjunctions follows from \cite[Proposition~2.10.8]{Etingof2015}.
  A straightforward calculation shows that tensoring (from the left or the right) with an invertible object establishes an equivalence of categories.

  Conversely, suppose that \(X \in \cat C\) is such that \(F \defeq \blank \otimes X\) is an equivalence of categories.
  The functor \(F\) and its quasi-inverse \(U\) are part of an adjunction with invertible unit \(\eta \from {\Id_{\cat C}} \to UF\) and counit \(\epsilon \from FU \to \Id_{\cat D}\), see for example \cite[Proposition~4.4.5]{Riehl2017}.
  By \cite[Proposition~4.4.1]{Riehl2017}, there exists a natural isomorphism \(\theta \from U \to \blank \otimes \lDual X\) which commutes with the respective counits and units.
  Applied to the monoidal unit \(1 \in \cat C\), we obtain
  \begin{equation*}
    \coev_{X}^l = \theta_X \eta_1 \qquad \text{and} \qquad \ev_X^l (\theta_1 \otimes \id_X)= \epsilon_1 .
  \end{equation*}
  It follows that \(X\) is invertible.
  An analogous argument shows that \(X\otimes \blank \) being an equivalence of categories entails \(X\) being invertible.
\end{proof}

In fact, taking duals extends to the \emph{left dualising functor},
\begin{equation}\label{eq: left_dual_fun}
  \lDual{(\blank)}\from \cat C \to \cat C^{\op, \tensorop},
\end{equation}
mapping objects to their left duals, and a morphism \(f\from X \to Y\) to
\[
  \lDual f \defeq (\ev_Y^l \otimes \id_{\lDual X})(\id_{\lDual Y} \otimes f \otimes \id_{\lDual X})(\id_{ \lDual Y} \otimes \coev_X^l) \from \lDual Y \to \lDual X.
\]
Its coherence morphisms are given by the isomorphisms induced by the uniqueness of duals.
Similarly, we have a \emph{right dualising functor} \(\rDual{(\blank)} \from \cat C \to \cat C^{\op, \tensorop}\).
To simplify computations, we want to `strictify' both of these.

\begin{definition}\label{def: stric_rigid_cat}
  A rigid monoidal category \(\cat C\) is called \emph{strict rigid}\footnote{
    The notion of `strict rigidity' is not prevalent in the literature and does not appear in \cite{Etingof2015}.
    However, hints towards it can be found for example in \cite[Section 5]{Schauenburg2001}.
  }
  if the dualisation functors \(\lDual{(\blank)}, \rDual{(\blank)}\from \mathcal C \to \mathcal C^{op,\otimes\text{-}\mathrm{op}}\) are strict and
  \begin{equation}\label{eq: comp_left_right_dual}
    \lDual{\big(\rDual{(\blank)}\big)} = \Id_{\cat C}
    = \rDual{\big(\lDual{(\blank)} \big)}.
  \end{equation}
\end{definition}

Our next result, a slight variation of \cite[Theorem~2.2]{Ng2007}, shows that every rigid category admits a \emph{rigid strictification}, \ie a monoidally equivalent strict rigid category.
The hinted at compatibility between the respective left and right duality functors is an immediate consequence of the fact that for any strong monoidal functor \(F \from \cat C \to \cat D\) between rigid categories there are natural monoidal isomorphisms
\begin{equation}\label{eq: monoidal_functors_compatible_with_rigid_structure}
  \phi_X \from F \big(\lDual{X}\big) \to \lDual{\big(F(X)\big)}, \qquad
  \psi_X \from F \big(\rDual{X}\big) \to \rDual{\big(F(X)\big)}, \quad \text{ for all } X \in \cat C.
\end{equation}

\begin{theorem}\label{thm: rigid_strictification}
  Every rigid category admits a rigid strictification.
\end{theorem}

\begin{proof}
  Taking a rigid and strict monoidal category \(\cat C\) as our input, we build a monoidally equivalent strict rigid category \(\cat D\).
  The objects of \(\cat D\) are (possibly empty) finite sequences
  \((X_1^{n_1}, \dots, X_i^{n_i})\) of objects \(X_1, \dots, X_i \in \cat C\) adorned with integers \(n_1, \dots, n_i \in \mathbb Z\).
  To define its morphisms, recall the notation of Equation~\eqref{eq: interated_duals_of_an_object} and set:
  \begin{equation*}\scalemath{0.94}{%
      \cat D(
      (X_1^{n_1}, \dots, X_i^{n_i}),
      (Y_1^{m_1}, \dots, Y_j^{m_j})
      )
      \defeq
      \cat C(
      {(X_1)}^{n_1} \otimes \dots \otimes {(X_i)}^{n_i},
      {(Y_1)}^{m_1} \otimes \dots \otimes {(Y_j)}^{m_j}
      ).}
  \end{equation*}
  The category \(\cat D\) is strict monoidal when equipped with the concatenation of sequences as tensor product and the empty sequence as unit.
  By construction, there exists a strict monoidal equivalence of categories \(F \from \cat D \to \cat C\),
  which maps any object
  \((X_1^{n_1}, \dots, X_i^{n_i}) \in \cat D\) to \({(X_1)}^{n_1} \otimes \dots \otimes {(X_i)}^{n_i} \in \cat C\),
  as well as every morphism to itself.\footnote{
    In the definition of \(F\from \cat D \to \cat C\) we regard the unit of \(\cat C\) as the empty tensor product.
  }
  Fixing an object \(X \defeq (X_1^{n_1}, \dots, X_i^{n_i}) \in \cat D\),
  define its left dual to be given by
  \(\lDual X \defeq (X_i^{n_i+1}, \dots, X_1^{n_1+1})\) with
  evaluation and coevaluation morphisms as shown in the following diagram
  \begin{center}
  \tikzsetnextfilename{inductive_evs_and_coevs}%
  \input{sections/tikzfigures/inductive_evs_and_coevs.tikz}%

  \end{center}
  where for all \(1 \leq k \leq i\) we set
  \begin{equation*}
    \phi_k \defeq
    \begin{cases}
      \ev^l_{(X_k)^{n_k}} &\text{ if } n_k \geq 0, \\
      \ev^r_{(X_k)^{n_k+1}} &\text{ if } n_k < 0, \\
    \end{cases}
    \qquad
    \text{ and }
    \qquad
    \psi_k \defeq
    \begin{cases}
      \coev^l_{(X_k)^{n_k}} & \text{ if } n_k \geq 0, \\
      \coev^r_{(X_k)^{n_k+1}} &\text{ if } n_k < 0. \\
    \end{cases}
  \end{equation*}
  We define the right dual of \(X\) to be \(\rDual X \defeq (X_i^{n_i-1}, \dots, X_1^{n_1-1})\) with evaluation and coevaluation similar to the above construction.
  It follows that \(\cat D\) is strict rigid.
\end{proof}

Many applications require that the objects of a rigid category are isomorphic to their double duals in a way which is compatible with the monoidal structure.
One of our aims is to gain a representation theoretic approach to detecting such a property.

\begin{definition}\label{def: pivotal_category}
  A \emph{pivotal category} is a rigid category \({\cat C}\) together with a fixed monoidal natural isomorphism
  \begin{equation}\label{eq:_relation_of_pivot}
    \rho\from \Id_{\cat C} \to \lbiDual{(\blank)},
  \end{equation}
  which is referred to as a \emph{pivotal structure} of \({\cat C}\).
\end{definition}

Rigid categories do not have to admit a pivotal structure and, if they do, it need not be unique.
Examples coming from Hopf algebra theory are given in \cite{Kauffman1993} and \cite{Halbig2018, Halbig2019}.
However, Shimizu showed that every rigid category admits a universal pivotal category, called its \emph{pivotal cover}, see \cite{Shimizu2015}.

\subsection{Braided categories}\label{subsec: braided_and_ribbon_cats}

Braidings are natural transformations relating the tensor product to its opposite.
They where introduced by Joyal and Street in \cite{Joyal1993}, building on the notion of symmetries studied amongst others in \cite{MacLane1963, Eilenberg1966}.

\begin{definition}\label{def: braided_cat}
  A \emph{braiding} on a monoidal category \(\cat C\) is a natural isomorphism
  \begin{equation*}
    \sigma_{X, Y} \from X \otimes Y \to Y \otimes X, \qquad \qquad \text{ for all } X, Y \in \cat C,
  \end{equation*}
  which satisfies \emph{the hexagon axioms}\footnote{
    The name `hexagon axioms' is due to the fact, that in the non-strict setting, the defining equations
    can be organised as a commuting, hexagon-shaped diagram; see \cite{Joyal1993}.
  }.

  The pair \((\cat C, \sigma)\) is referred to as \emph{braided monoidal category}.
\end{definition}

\begin{remark}\label{rem: braiding_compatibility_with unit}
  We will often make use of the fact that the braiding of any object \(X\in \cat C\) with the unit \(1 \in \cat C\) of a braided category \((\cat C, \sigma)\) is trivial.
  This is a consequence of the hexagon identities; computing
  \(\sigma_{X,1} = \sigma_{X, 1\otimes 1} = (\id_1 \otimes \sigma_{X,1})(\sigma_{X,1} \otimes \id_1) = \sigma_{X,1} \sigma_{X,1}\)
  shows \(\sigma_{X,1}\) to be an invertible idempotent, hence \(\sigma_{X,1} = \id_X\).
\end{remark}

Braidings are depicted in the graphical calculus by crossings of strings subject to Reidemeister-esque identities, see \cite{Selinger2011}.
The following figure shows a braiding, its inverse, the hexagon identity, and the naturality of the braiding in its first argument.

\begin{center}
  \tikzsetnextfilename{braidings}%
  \input{sections/tikzfigures/braidings.tikz}%

\end{center}

\subsection{Bimodule categories and the centre construction}\label{subsec: bimodule_cats_and_centre_constr}

Just as monoids can act on sets, monoidal categories can act on categories.
Thinking representation theoretically therefore advocates studying monoidal categories through their modules.
In parallel with our treatment of monoidal categories, we will focus solely on their `strict modules'.
Again, a more general theory is possible by weakening the associativity and unitality of the action.

\subsubsection{Left right and bimodule categories}\label{subsubsec: left_and_right_mod_cat}

\begin{definition}\label{def: module_category}
  A \emph{strict left module (category)} over a monoidal category \({\cat C}\) is a pair \((\cat M, \lact)\) comprising a category \(\cat M\) and an \emph{action} of \({\cat C}\) on \(\cat M\) implemented by a functor \(\lact\from \cat C \times \cat M \to \cat M\) such that
  \begin{equation}
    (\blank \otimes \blank) \lact \blank = \blank \lact ( \blank \lact \blank) \qquad\quad \text{ and } \qquad \quad 1 \lact - = \Id_{\cat M}.
  \end{equation}
\end{definition}

To keep our notation concise, we will simply speak of modules, instead of strict module categories, over a monoidal category.

For a functor between modules to be structure preserving, it has to satisfy a variant of equivariance which is encoded by a natural isomorphism.

\begin{definition}\label{def: functor_mod_cat}
  Let \(\cat M\) and \(\cat N\) be left modules over a monoidal category \(\cat C\).
  A \emph{functor of left modules} is a functor \(F\from \cat M \to \cat N\) together with a natural isomorphism
  \begin{equation*}
    \delta_{X,M} \from F(X \lact M) \to X \lact F(M), \qquad \qquad \text{ for all } X \in \cat C \text{ and } M \in \cat M
  \end{equation*}
  such that for all \(X, Y \in \cat{C}\) and \(M \in \cat{M}\), we have
  \(\delta_{X \otimes Y, M} = (\id_X \lact \delta_{Y, M}) \delta_{X, Y \lact M}\)
  and
  \(\id_{M} = \delta_{1, M}\).
  We call \((F, \delta)\) \emph{strict} if \(\delta\) is given by the identity.
\end{definition}

With respect to the analogy between oplax monoidal functors and coalgebras, module functors play the role of (strong) comodules over the identity functor.
We will encounter the  more general concept of comodule functors in Sections \ref{sec: monads} and \ref{sec: mon_perspective}.

An \emph{equivalence of module categories} is a functor of module categories \(F \from \cat M \to \cat N\) that is an equivalence.
As with monoidal categories, it admits a quasi-inverse functor of module categories \(G \from \cat N \to \cat M\) and the natural isomorphisms \(FG \rightarrow \Id_\cat N\) and \(GF \rightarrow \Id_\cat M\) are compatible with the respective `coactions' in a way explained in the next definition.

\begin{definition}\label{def: mor_of_left_mod_functors}
  Let \(F, G \from \cat M \to \cat N\) be two functors of left modules over a monoidal category \({\cat C}\).
  A \emph{morphism of left module functors} is a natural transformation \(\phi\from F \to G\) satisfying
  \((\id_X \lact \phi_M) \; \delta^{(F)}_{X,M} = \delta^{(G)}_{X,M} \; \phi_{X\lact M}\),
  for all \(X \in \cat C\) and \(M \in \cat M\).
\end{definition}

\emph{Module adjunctions} will be a corner stone of our investigation.
They are defined as adjunctions \(\adj{F}{G}{\cat M}{\cat N}\) of module functors between module categories whose unit and counit are module natural transformations.

A theory of \emph{right modules} can be formulated in a similar fashion.
More precisely, right modules over a monoidal category \(\cat C\) can be identified with left modules over \(\cat C^{\tensorop}\).
If we assume some additional conditions on \(\cat C\), we could define its bimodules as left modules over an `enveloping category' \(\cat C^e\) of \(\cat C\).
For our purposes, however, it will be more beneficial to define them explicitly in terms of categories with compatible left and right actions.

\begin{definition}\label{def: bimodule_category}
  A \emph{(strict) bimodule} \((M,\lact, \ract)\) over a monoidal category \(\cat C\) is a category \(\cat M\) which is simultaneously a left and right module and
  \begin{equation}
    (\blank \lact \blank) \ract \blank = \blank \lact ( \blank \ract \blank).
  \end{equation}
\end{definition}

\begin{example}\label{ex: regular_bimodule}
  The prime example of a bimodule over a monoidal category \({\cat C}\) is the \emph{regular bimodule} \({}_{\Id_{\cat C}} {\cat C}_{\Id_{\cat C}}\)\footnote{%
    This notation will be explained in more detail in Section~\ref{subsec: twisted_centres}.%
  }.
  As a category, it is simply \(\cat C\) and the left and right actions are given by tensoring from the left and right, respectively.
\end{example}

\begin{remark}\label{rem: tensor_product_of_bimodule_categories}
  If \(\cat C\) is for example a tensor category, its bimodules form a monoidal \(2\)-category, see \cite{Greenough2010}.
\end{remark}

Since we will not work with bimodule functors and their natural transformations, we will not state their precise definitions.
Rather, we remark that they equate to (strong) `bicomodules' over the identity functor.

\subsubsection{The Drinfeld centre of a monoidal category}\label{subsubsec: Drinfeld_centre}

The centre construction can be used to obtain a braided category from a monoidal one.
We work in a slightly more general setup than \cite[Chapter 7]{Etingof2015} and define centres for bimodule categories.
See \cite{Gelaki2009,Bruguieres2012, Hassanzadeh2019, Kowalzig2020} for similar approaches.

\begin{definition}\label{def: half_braiding}
  Let \(\cat M\) be a bimodule over a monoidal category \(\cat C\) and \(M\in \cat M\) an object.
  A \emph{half-braiding} on \(M\) is a natural isomorphism
  \begin{equation*}
    \sigma_{M, X} \from M \ract X \rightarrow X\lact M, \qquad \qquad \text{ for all } X \in \cat C,
  \end{equation*}
  satisfying for all \(X,Y \in \cat C\) the hexagon axiom
  \begin{equation}\label{eq: hexagon_axiom_half_braiding}
    \sigma_{M,X\otimes Y} = ({\id_X} \lact \sigma_{M,Y})(\sigma_{M, X} \ract \id_Y).
  \end{equation}
\end{definition}

Let \(\sigma_{M, \blank} \from M \ract - \to \blank \lact M\) be a half-braiding on an object \(M \in \cat M\).
The same arguments as in Remark~\ref{rem: braiding_compatibility_with unit} show that \(\sigma_{M, 1} = \id_M\) for all \(M \in \cat M\).

Thinking of objects plus half-braidings as `central elements', one can try to mimic the centre construction from representation theory.
This leads to the following definition.

\begin{definition}\label{def: centre_bimodue_category}
  The centre of a bimodule \(\cat M\) over a monoidal category \(\cat C\) is the category \(\ZCat M\).
  It has as objects pairs \((M, \sigma_{M, \blank})\) comprising an object \(M \in \cat M\) together with a half-braiding \(\sigma_{M,\blank}\) on \(M\).
  The set of morphisms between two objects \((M, \sigma_{M,\blank})\), \((N, \sigma_{N,\blank}) \in \ZCat M\), consists of those morphisms \(f \in \cat M(M,N)\) which commute with the half-braidings.
  That is,
  \begin{equation}\label{eq: def_rel_morphism_in_centre}
    (\id_X \lact f)\sigma_{M,X} = \sigma_{N,X} (f \ract \id_X), \qquad \qquad \text{ for all } X \in \cat C.
  \end{equation}
\end{definition}

There is a canonical forgetful functor \(U^{(M)} \from \ZCat M \to \cat M\).
Unlike classical representation theory where the centre of a bimodule is a subset, \(U^{(M)}\) need not be injective on objects in general.

\begin{example}\label{ex: Drinfeld_centre}
  The centre \(\ZCat C\) of the regular bimodule of a monoidal category \(\cat C\) is called the \emph{Drinfeld centre} or simply centre of \(\cat C\).
  It is braided monoidal.
  The tensor product is defined by
  \((M, \sigma_{M, \blank}) \otimes (N, \sigma_{N, \blank}) \defeq ( M \otimes N, \sigma_{M\otimes N, \blank})\) with
  \begin{equation*}
    \sigma_{M\otimes N, X} \defeq (\sigma_{M, X} \otimes \id_N)({\id_M} \otimes \sigma_{N, X}), \qquad \qquad \text{ for all } X \in \cat C.
  \end{equation*}
  Its braiding is given by the respective half-braidings.
  The hexagon axioms follow from Equation~\eqref{eq: hexagon_axiom_half_braiding} and the definition of the tensor product of \(\ZCat C\).
\end{example}

The next proposition uses the shorthand notation for iterated duals given in Equation~\eqref{eq: interated_duals_of_an_object}.

\begin{proposition}[{\cite[Lemma~7]{joyal91:yang-baxter}}] \label{prop: rigdity_of_centre}
  Suppose \(\cat C\) to be strict rigid.
  Its Drinfeld centre \(\ZCat C\) inherits the rigid structure of \(\cat C\).
  That is, for all \((X, \sigma_{X,-}) \in \ZCat C\) we have
  \begin{equation*}
    U^{(Z)} \big(\lDual{(X, \sigma_{X,-})}\big) = \lDual X, \qquad \qquad
    U^{(Z)} \big(\rDual{(X, \sigma_{X,-})}\big) = \rDual X.
  \end{equation*}
  Moreover, for every \(n \in \mathbb Z\) and \(X \in \ZCat C\) we have
  \begin{equation}\label{eq: comp_braiding_n_duals}
    \sigma_{(X)^{n}, (Y)^{n}} = (\sigma_{X,Y})^{n}, \qquad \qquad \text{ for all } Y \in \cat C.
  \end{equation}
\end{proposition}

%%% Local Variables:
%%% mode: latex
%%% TeX-master: "../main"
%%% End:

\section{Heaps}\label{sec: heaps}
Heaps can be thought of as groups without a fixed neutral element. Prüfer studied their abelian version under the name \emph{Schar} in \cite{Pruefer1924}.
Since then, the notion has been adapted to the non-abelian case, see \cite{Hollings2017}. Recently, their homological properties were studied in \cite{Elhamdadi2019}; a generalisation towards a `quantum version' of heaps is hinted at in \cite{Skoda2007}. We follow Section 2 of \cite{Brzezinski2020} for our exposition.

\begin{definition}\label{def: heap}
  A \emph{heap} is a set \(G\) together with a ternary operation
  \begin{equation*}
    \langle \blank , \blank , \blank\rangle \from G\times G \times G \to G,
  \end{equation*}
  which we call the \emph{heap operation},
  satisfying a generalised associativity axiom and the \emph{\Malcev identities}, of which we think as unitality axioms:
  \begin{gather}
    \langle g, h, \langle i,j,k\rangle \rangle = \langle \langle g,h,i\rangle , j , k \rangle, \qquad \qquad \text{ for all } g,h, i ,j, k \in G, \label{eq: heap_ass_constraint} \\
    \langle g, g, h \rangle = h = \langle h, g, g \rangle, \qquad \qquad \text{ for all } g, h \in G.
    \label{eq: heap_unit_constraint}
  \end{gather}
\end{definition}

There are two peculiarities we want to point out.
First, our definition does, intentionally, not exclude the empty set from being a heap.
Second, due to a slightly different setup, an additional `middle' associativity axiom is required in \cite{Hollings2017}.
However, as noted in \cite[Lemma~2.3]{Brzezinski2020}, it is implied by the `outer' associativity and the \Malcev identities.

\begin{definition}\label{def: morphism_of_heaps}
  A map \(f \from G \to H\) between heaps is a \emph{morphism of heaps} if
  \begin{equation}
    f \left( \langle g, h, i\rangle \right) = \langle f(g), f(h), f(i) \rangle, \qquad \qquad \text{ for all } g,h, i \in G.
  \end{equation}
\end{definition}

The next lemma can be shown by mimicking the proof of its group theoretical version.

\begin{lemma}\label{lem: isomorphisms_of_heaps}
  A morphism of heaps \(f \from G \to H\) is an isomorphism if and only if it is bijective.
\end{lemma}

By forgetting its unit, any group defines a heap.
Conversely, any non-empty heap can be turned into a group by choosing a fixed element to act as unit, see \cite{Certaine1943}.

\begin{lemma}\label{lem: groups_to_heaps}
  Every group \((G, \cdot, e)\) is a heap via
  \begin{equation*}
    \langle \blank,\blank,\blank \rangle \from G\times G \times G \to G, \qquad \qquad \langle g, h, i \rangle \defeq g \cdot h^{-1} \cdot i.
  \end{equation*}
  A morphism of groups becomes a morphism of the induced heaps.
\end{lemma}

\begin{lemma}\label{lem: heaps_to_groups}
  A non-empty heap \(H\) with a fixed element \(e \in H\) can be considered as a group with unit \(e\) via the multiplication
  \begin{equation*}
    \blank \cdot_e \blank \from H \times H \to H, \qquad g \cdot_e h \defeq \langle g, e, h \rangle.
  \end{equation*}
  The inverse of an element \(g \in H\) with respect to \(\cdot_e\) is given by \(g^{-1} \defeq \langle e, g, e \rangle\).
  A morphism of heaps is a morphism of the induced groups, provided it maps the fixed element of its source to the fixed element of its target.
\end{lemma}

More generally, if \(\cat{G}\) is a groupoid and \(x, y \in \cat{G}\),
then \(\cat{G}(x, y)\) becomes a heap; the heap operation is given by
\(\langle f, g, h \rangle \defeq f g^{-1} h\).
The following example is a special case of this construction,
which will play a prominent role in our investigation.

\begin{example}\label{ex: heap_of_nat_isos}
  Let \(F, G \from \cat C \to \cat C\) be two oplax monoidal endofunctors. The set
  \begin{equation*}
    \Isospace_\otimes (F, G) \defeq \left\{\text{oplax monoidal natural isomorphisms from \(F\) to \(G\)} \right\}
  \end{equation*}
  bears a heap structure, the heap operation being
  \begin{equation}
    \langle \blank , \blank , \blank \rangle \from \Isospace_\otimes (F, G)^3 \to \Isospace_\otimes (F, G), \qquad \qquad
    \langle \phi, \psi, \xi \rangle = \phi \psi^{-1} \xi.
  \end{equation}
\end{example}

%%% Local Variables:
%%% mode: latex
%%% TeX-master: "../main"
%%% End:

\section{Pivotal structures and twisted centres}\label{sec: pivotality_of_the_df_centre}
In this section, we study the relations between pairs in involution, anti-\YetterDrinfeld modules and isomorphisms between the Drinfeld and anti-Drinfeld double from a categorical point of view.
Our approach is representation theoretic in nature.
We consider variants of the regular bimodule of a rigid category \(\cat C\) with either the left or right action twisted by a strict monoidal endofunctor.
Their centres are canonically modules over the Drinfeld centre.
These \emph{twisted centres} inherit a notion of duality which follows in close parallel to that of \(\ZCat C\).
Module functors between the Drinfeld and a twisted centre are determined by their value on the unit object.
A consequence of the above sketched duality is that module equivalences correspond to objects in the twisted centre, which behave as if they were invertible.
We gather these objects into the Picard heap of the twisted centre.
If we twist with the left double dual functor, we obtain a generalised version of the anti-\YetterDrinfeld modules, see~\cite{Hassanzadeh2019}.
Its Picard heap has an alternative interpretation as quasi-pivotal structures; appropriate analogues of pairs in involution.
This observation leads us to the desired relations in categorical terms, given in Theorem~\ref{thm: equivalence_in_the_rigid_setting}.

In~\cite{Shimizu2016}, Shimizu observed that quasi-pivotality of \(\cat C\) induces pivotality of \(\ZCat C\).
We recall his proof from the perspective of twisted centres and investigate how this construction is related to the so-called symmetric centre of \(\cat C\).
This leads to an injective heap morphism from a quotient of the Picard heap of the generalised anti-\YetterDrinfeld modules to the heap of pivotal structures of \(\ZCat C\).
We end the section, by constructing a category such that this morphism is not surjective.

In the following, \(\cat C\) denotes a strict rigid category.

\subsection{Twisted centres and their Picard heaps}\label{subsec: twisted_centres}

The regular action is not the only way in which we can consider \(\cat C\) as a bimodule over itself.
Given two strict monoidal endofunctors \(L, R \from \cat C \to \cat C\), we can `twist' the action by defining for all \(V, W, X, Y \in \cat C\) and \(f \from V\to W, g \from X \to Y\),
\begin{equation}\label{eq: twisted_action}
  \begin{aligned}
    X &\lact Y \defeq L(X) \otimes Y, &\qquad\qquad
                          f &\lact g \defeq L(f) \otimes g, \\
    Y &\ract X \defeq Y \otimes R(X), & \qquad\qquad
                          g &\ract f \defeq g \otimes R(f).
  \end{aligned}
\end{equation}
We write \({}_{L}{\cat C}_{R}\) for the bimodule obtained in this manner and call it the \emph{bimodule obtained by twisting with \(L\) from the left and \(R\) from the right} or, if the functors \(L\) and \(R\) are apparent from the context, simply a \emph{twisted bimodule}.
Accordingly, we refer to \(\ZCat[L]{C}[R]\) as a \emph{twisted centre}.
In case we want to stress that \(L\) or \(R\) are the identity functors, we write \(\cat C_R \defeq {}_{\Id_{\cat C}}{\cat C}_{R}\) and \({}_L{\cat C} \defeq {}_{L}{\cat C}_{\Id_{\cat C}}\) and speak of a \emph{right} and \emph{left} twisted bimodule, respectively.
Following this pattern, \(\ZCat{C}[R]\) and \(\ZCat[L]{C}\) are called \emph{right} and \emph{left} twisted centres.

\begin{remark} \label{rmk:twisted-centre-as-loop}
  Another natural description of twisted centres is obtained from the perspective of bicategories.
  Given a monoidal category \(\cat{C}\)
  let \(\mathbf{B}\cat{C}\) denote its \emph{delooping};
  that is, \(\mathbf{B}\cat{C}\) is a bicategory with a single object \(\bullet\),
  and \(\mathbf{B}\cat{C}(\bullet, \bullet) \simeq \cat{C}\).
  Via delooping, strong monoidal functors are identified with pseudofunctors between one-object bicategories and vice versa.
  However, not every term in the language of monoidal categories is transferred straightforwardly into a higher-categorical analogue.
  Let \(\End{}{\mathbf{B}\cat{C}}\) be the bicategory of endo-pseudofunctors of $\mathbf{B}\cat{C}$, their pseudonatural transformations, and modifications.
  Then \(\End{}{\mathbf{B}\cat{C}}(L, R) \simeq \ZCat[L]{C}[R]\).
  A proof  is given for example in \cite[Proposition~3.6]{femić23:categ-yetter}.
\end{remark}

The forgetful functor from the centre of a twisted bimodule to the underlying monoidal category is faithful.
Therefore, we can use the graphical calculus discussed previously as long as we pay special attention to the half-braidings.
Given that we will often deal with multiple twisted centres at once, we introduce a colouring scheme to help us keep track of the various categories:

\begin{thmlist}
  \item Red for objects in the right twisted centre \(\ZCat{C}[R]\),
  \item blue for objects in the left twisted centre \(\ZCat[L]{C}\) and
  \item black for objects in the Drinfeld centre \(\ZCat C\) or \(\cat C\).
\end{thmlist}
For example, the half-braidings of objects \(A \in \ZCat{C}[R]\) and \(Q \in \ZCat[L]{C}\) are:
\begin{center}
  \tikzsetnextfilename{a_and_q_braidings}%
  \input{sections/tikzfigures/a_and_q_braidings.tikz}%

\end{center}

\begin{remark}\label{rem: other_variants_of_twisted_centres}
  One can easily imagine a more involved setting than what is described above by twisting with an oplax monoidal functor \((L, \Delta, \varepsilon) \from \cat C \to \cat C\) from the left and a lax monoidal functor \((R, \mu, \eta)\) from the right.
  We hypothesise that \({}_{L}{\cat C}_{R}\) would be a type of `oplax-lax' bimodule over \(\cat C\),
  whose actions are associative and unital only up to coherent natural transformations, subject to laws as described in~\cite[Section 2]{Szlachanyi2012}.
  At least conceptually, this unifies our subsequent considerations with the centres studied in~\cite{Bruguieres2012}.
  We will revisit these more general structures in Section~\ref{sec: mon_perspective} and for now only remark that the half-braiding of an object \(X \in \ZCat[L]{C}[R]\) is a natural transformation \(\sigma_{X, \blank} \from X\otimes R(\blank) \to L(\blank) \otimes X\), which has to satisfy:
  \begin{equation}
  \tikzsetnextfilename{generalised_braiding}%
  \input{sections/tikzfigures/generalised_braiding.tikz}%

  \end{equation}
\end{remark}

\begin{convention}\label{conv: twisting_with_only_one_functor}
  In what follows, we are predominantly interested in twisting with the same strict monoidal functor from the left or right.
  For the purpose of brevity, we therefore fix such a functor
  \(L=R \from \cat C \to \cat C\) and consider the categories \({}_{L}{\cat C}\) and \({\cat C}_{R}\).
\end{convention}

Suppose we are given three objects
\begin{equation*}
  (A, \sigma_{A, \blank}) \in \ZCat{C}[R],
  \qquad\qquad
  (Q, \sigma_{Q, \blank}) \in \ZCat[L]{C}
  \qquad\text{and}\qquad
  (X, \sigma_X) \in \ZCat C.
\end{equation*}

The diagrams below show that various tensor products of the underlying objects in \(\cat C\) admit `canonical' half-braidings.
\begin{equation}\label{eq: gluing_of_half_brainings}
  \tikzsetnextfilename{gluing_of_half_brainings}%
  \input{sections/tikzfigures/gluing_of_half_brainings.tikz}%

\end{equation}

The top row suggests a right action of \(\ZCat C\) on left twisted centres and a left action on right twisted centres.

\begin{proposition}\label{prop: Left_or_right_twisted_centre_is_mod_cat}
  The tensor product of \(\cat C\) extends to a left and a right action of the Drinfeld centre \(\ZCat{C}\) on \(\ZCat{C}[R]\) and \(\ZCat[L]{C}\), respectively.
  The half-braidings are as defined in Diagram~\eqref{eq: gluing_of_half_brainings}.
\end{proposition}

\begin{remark}\label{rem: right_vs_left_twisted_centres}
  Right and left twisted centres are two sides of the same coin.
  We write \(\overline{\cat C} \defeq \cat C^{\op, \tensorop}\).
  A direct computation proves the categories \(\ZCat*{\overline{\cat C}}[R]\) and \(\ZCat[R]{C}^{\op}\) to be the same.
  This identification is compatible with the respective actions since \(\blank \otimes^{\op} R(\blank) = R(\blank) \otimes \blank \) and \(\sigma_{X \otimes^{\op} A, \blank} = \sigma_{A \otimes X, \blank}\) for all \(X \in \ZCat{\overline C}\) and \(A \in \ZCat*{\overline{\cat C}}[R]\).
  According to these considerations, from now on we deliberately restrict ourselves to the study of right twisted centres.
\end{remark}

The left dual \(\lDual A\) of any object \((A, \sigma_{A,\blank}) \in \ZCat{C}[R]\) can be turned into an object of \(\ZCat[R]{C}\) if we equip it with the half-braiding

\begin{equation}\label{eq: braiding_of_left_dual_twisted_centre}
  \tikzsetnextfilename{braiding_of_left_dual_twisted_centre}%
  \input{sections/tikzfigures/braiding_of_left_dual_twisted_centre.tikz}%

\end{equation}
This suggest that, in analogy with  Proposition~\ref{prop: rigdity_of_centre}, we may lift the dualising functor of $\cat{C}$ to the  level of twisted centres by interchanging right with left twists.
A more conceptual description  is provided in~\cite[Proposition~3.15]{femić23:categ-yetter}.

\begin{proposition}\label{prop: left_dualising_as_a_module_fun}
  The left dualising functor \(\lDual{(\blank)} \from \cat C \to \cat C^{\op, \tensorop}\) lifts to a functor between right and left twisted centres
  \begin{equation}
    \lDual{(\blank)} \from \ZCat{C}[R] \to \ZCat[R]{C}^{\op}.
  \end{equation}
\end{proposition}

The half-braidings displayed in the right column of Diagram~\eqref{eq: gluing_of_half_brainings} show that every object \(A \in \ZCat{C}[R]\) gives rise to two functors of left modules over \(\ZCat C\),

\begin{equation} \label{eq: objs_as_mod_funs}
  \blank \otimes A \from \ZCat C \to \ZCat{C}[R]
  \qquad \text{ and } \qquad
  \blank \otimes \lDual A \from \ZCat{C}[R] \to \ZCat C.
\end{equation}
Before we prove that the adjunction \(\adj{\blank \otimes A}{\blank \otimes \lDual A}{\cat C}{\cat C}\), discussed in Proposition~\ref{prop: rigidity_yields_adjoints}, lifts to an adjunction of module categories, we fix our notation for the evaluation and coevaluation morphisms in the context of twisted centres.
For any object \((A, \sigma_{A, \blank}) \in \ZCat{C}[R]\), we write

\begin{equation}\label{eq: ev_and_coev_for_twisted_centres}
  \tikzsetnextfilename{ev_and_coev_ayd}%
  \input{sections/tikzfigures/ev_and_coev_ayd.tikz}%

\end{equation}

\begin{proposition}\label{prop: adjunction_twisted_centre}
  Every object \(A \in \ZCat{C}[R]\) induces adjoint \(\ZCat C\)-module functors
  \begin{equation}
    \adj{\blank \otimes A}{\blank \otimes \lDual A}{\ZCat C}{\ZCat{C}[R]}.
  \end{equation}
\end{proposition}
\begin{proof}
  We fix an object \((A, \sigma_{A,\blank}) \in \ZCat{C}[R]\).
  Considered as endofunctors of \(\cat C\), there is an adjunction \(\adj{\blank \otimes A}{\blank \otimes \lDual A}{\ZCat C}{\ZCat{C}[R]}\).
  As stated in the proof of Proposition~\ref{prop: rigidity_yields_adjoints},
  its unit and counit are implemented via the evaluation and coevaluation morphisms
  \begin{gather*}
    \eta_Y \defeq \id_Y \otimes \coev^l_A \from Y \to Y\otimes \lDual A \otimes A, \qquad \qquad \text{ for all } Y \in \ZCat C,\\
    \epsilon_X \defeq \id_X \otimes \ev^l_A \from X \otimes \lDual A \otimes A \to X, \qquad \qquad \text{ for all } X \in \ZCat C[R].
  \end{gather*}
  The next diagram shows that \(\epsilon_X\) is a morphism in \(\ZCat{C}[R]\) for every \(X \in \ZCat C[R]\).
  \begin{equation}\label{eq: ev_induces_morph_of_mod_functors}
  \tikzsetnextfilename{adjunction_module_functors}%
  \input{sections/tikzfigures/adjunction_module_functors.tikz}%

  \end{equation}
  Furthermore, \(\epsilon_{W\lact X} = \id_W \otimes \epsilon_X\) for all \(W \in \ZCat C\).
  A similar argument shows that the unit of the adjunction is a natural transformation of module functors as well.
\end{proof}

The forgetful functor from the (twisted) centre to its underlying category is \emph{conservative}, \ie it `reflects' isomorphisms.
This allows us to characterise equivalences of module categories between \(\ZCat C\) and right twisted centres.

\begin{proposition}\label{prop: module_funs_have_adjoints}
  Any functor of left module categories \(F\from \ZCat C \to \ZCat{C}[R]\) is naturally isomorphic to
  \begin{equation*}
    \blank \otimes A \from \ZCat C \to \ZCat{C}[R],
  \end{equation*}
  with \( A = F(1) \in \ZCat{C}[R]\).
  As a consequence, \(F\) is an equivalence if and only if \(A\) is invertible as an object of \(\cat C\).
\end{proposition}
\begin{proof}
  The first claim is an immediate consequence of the unitality of the action.
  Suppose that \(H \cong \blank \otimes A\) is an equivalence.
  By Proposition~\ref{prop: rigidity_yields_adjoints}, \(A\) must be invertible.
  If conversely \(A\) is invertible, the same result shows that \(\blank \otimes A\) is an equivalence of categories.
\end{proof}
Notice in particular that by the above result two functors  \(F, G\from \ZCat C \to \ZCat{C}[R]\) of left module categories are naturally isomorphic if and only if $F(1) \cong G(1)$.

\begin{definition}\label{def: C-invertible}
  An object $(\alpha, \sigma_{(\alpha, \blank)}) \in \ZCat{C}[R]$ in is called \emph{$\cat C$-invertible} if the image $U(\alpha) \in\cat{C}$ of $\alpha$ under the forgetful functor $U \from\ZCat{C}[R] \to\cat{C}$ is invertible.
\end{definition}

The notion of heaps allows us to define an algebraic structure on the isomorphism classes of objects implementing module equivalences between the Drinfeld centre \(\ZCat C\) and its twisted `relative' \(\ZCat{C}[R]\).
In analogy with the Picard group, we call this the \emph{Picard heap} of a twisted centre.

\begin{lemma}\label{lem: inv_twisted_form_heap}
  The \emph{Picard heap} of the right twisted centre \(\ZCat{C}[R]\) is the set of isomorphism classes
  \begin{equation}
    \Pic \ZCat{C}[R] \defeq
    \big\{
    [\alpha]
    \;\big|\;
    \alpha \in \ZCat{C}[R] \text{ is }\cat{C}\text{-invertible}
    \big\}
  \end{equation}
  together with the heap operation defined for \([\alpha], [B], [C] \in \Pic \ZCat{C}[R]\) by
  \begin{equation}
    \left \langle [\alpha], [\beta] , [\gamma] \right \rangle = [\alpha \otimes \lDual \beta \otimes \gamma] .
  \end{equation}
\end{lemma}
\begin{proof}
  The generalised associativity, see Equation~\eqref{eq: heap_ass_constraint}, follows from the associativity of the tensor product of \(\cat C\) and its compatibility with the `gluing' of half-braidings.
  To show that the \Malcev identities hold, we fix objects \(\alpha, \beta \in \ZCat{C}[R]\), which are invertible in \(\cat C\).
  Proposition~\ref{prop: rigdity_of_centre} and Equation~\eqref{eq: ev_induces_morph_of_mod_functors} imply that
  \begin{gather*}
    \alpha \otimes \lDual \alpha \otimes \beta \trightarrow{{\coev^l_\alpha}^{-1}\otimes {\id_\beta}} \beta
    \qquad
    \text { and }
    \qquad
    \beta \otimes \lDual \alpha \otimes \alpha \trightarrow{{\id_\beta} \otimes {\ev^l_\alpha}} \beta
  \end{gather*}
  are isomorphisms in \(\ZCat{C}[R]\) and therefore
  \(
  \left \langle [\alpha], [\alpha] , [\beta] \right \rangle
  = [\beta]
  = \left \langle [\beta], [\alpha] , [\alpha] \right \rangle.
  \)
\end{proof}

In general, the twisted centre \(\ZCat{C}[R]\) does not inherit a monoidal structure from \(\cat C\).
The above lemma, however, hints towards a slight generalisation where the tensor product is replaced by a trivalent functor, essentially categorifying heaps (without the \Malcev identities).
The well-definedness of this concept was hinted at in~\cite{Skoda2007} under the name of \emph{heapy categories}.

\subsection{Quasi-pivotality}\label{subsec: quasi-pivotality}

A particularly interesting consequence of our previous findings arises in case \(R=\lbiDual{(\blank)}\) is the left double dual functor.
The centre of the regular bimodule twisted on the right by \(\lbiDual{(\blank)}\) can be understood as a generalisation of anti-\YetterDrinfeld modules, see~\cite[Theorem~2.3]{Hassanzadeh2019}.

As before, we fix a strict rigid category \(\cat C\) and consider the twisted bimodules \({\cat C}_{\lbiDual{(\blank)}}\) and \({}_{\lbiDual{(\blank)}} {\cat C}\).

\begin{notation}\label{not: twisted_centres_with_charge}
  We denote by \(\ACat{C} \defeq \ZCat{C}[\lbiDual{(\blank)}]\) and \(\QCat C \defeq \ZCat[\lbiDual{(\blank)}]{C}\) the centre of the regular bimodule twisted by the double dual functor from the right and left, respectively.
  The former will also be called the \emph{anti-Drinfeld centre} of \(\cat C\).
\end{notation}

We have already mentioned the connection between the twisted centre \(\ACat C\) and anti-\YetterDrinfeld modules over Hopf algebras given in~\cite{Hassanzadeh2019}.
The case where \(\cat C\) is the category of modules over a Hopf algebroid was recently investigated by Kowalzig in~\cite{Kowalzig2020}.
The counterpart \(\QCat C\) of the generalised anti-\YetterDrinfeld modules is less common in the literature but plays a crucial role in our investigation, especially in Sections~\ref{sec: monads} and~\ref{sec: mon_perspective}, where we focus on the monadic point of view.

The next definition is a specific case of an unnamed construction studied in~\cite[Section 4]{Shimizu2016}.

\begin{definition}\label{def: quasi_pivotal}
  A \emph{quasi-pivotal structure} on a rigid category \(\cat C\) is a pair \((\beta, \rho_{\beta})\) comprising an invertible object \(\beta \in \cat C\) and a monoidal natural isomorphism
  \begin{equation}
    \rho_\beta \from \Id_{\cat C} \to \beta \otimes \lbiDual{(\blank)} \otimes \lDual \beta.
  \end{equation}
  We refer to \((\cat C, (\beta, \rho_\beta))\) as a \emph{quasi-pivotal} category.
\end{definition}

If \(\cat C\) is the category of finite-dimensional modules over a finite-dimensional Hopf algebra, quasi-pivotal structures have a well-known interpretation---they translate to pairs in involution.
This can be deduced from a slight variation of \cite[Lemma~5.6]{Halbig2019},
the main observation being that the invertible object \(\beta\) of a quasi-pivotal structure \( (\beta, \rho_\beta)\) on \(\cat C\) corresponds to a character, and \(\rho_{\beta}\) determines a group-like element.
The fact that \(\rho_\beta\) is a natural transformation from the identity to a conjugate of the double dual functor is captured by the character and group-like implementing the square of the antipode.
We study a monadic analogue of this statement in Section~\ref{subsec: monadic_perspective_piv_structures}.

\begin{remark}\label{rem: quasi_pivotal_vs_pivotal}
  Every pivotal category is quasi-pivotal; the converse does not hold.
  A counterexample are the finite-dimensional modules over the generalised Taft algebras discussed in~\cite{Halbig2018}.
  Any of these Hopf algebras admit pairs in involution but in general neither the character nor the group-like can be trivial.
  The previous discussion and Lemma \cite[Lemma~5.6]{Halbig2019} show that \(\rMod{H}\) is quasi-pivotal but not pivotal---in contrast to its Drinfeld centre \(\ZCat*{\rMod{H}}\), which admits a pivotal structure  by \cite[Lemma~5.5]{Halbig2019}.
\end{remark}

Let \((\beta, \rho_\beta)\) be a quasi-pivotal structure on \(\cat C\) and \(\phi \from \beta' \to \beta\) an isomorphism.
Clearly, the pair \((\beta', (\phi^{-1}\otimes \id \otimes \lDual \phi)\rho_\beta)\) is another quasi-pivotal structure on \(\cat C\).
This defines an equivalence relation and we write
\begin{equation*}
  \QPiv (\cat C) \defeq \big\{ [(\beta, \rho_\beta)] \mid (\beta, \rho_\beta) \text{ is a quasi-pivotal structure on } \cat C \big\}
\end{equation*}
for the set of equivalence classes of quasi-pivotal structures on \(\cat C\).

\begin{lemma}\label{lem: quasi_piv_vs_inv_ayds}
  Let \(\cat C\) be a strict rigid category.
  The	Picard heap \(\Pic \ACat C\) and the set of equivalence classes of quasi-pivotal structures \(\QPiv (\cat C)\) are in bijection.
\end{lemma}
\begin{proof}
  Let \((\beta, \rho_\beta)\) be a quasi-pivotal structure on \(\cat C\).
  We define the half-braiding
  \begin{equation}
  \tikzsetnextfilename{qpiv_iso}%
  \input{sections/tikzfigures/qpiv_iso.tikz}%

  \end{equation}
  It satisfying the hexagon identity is due to \(\rho_\beta\) being monoidal.
  This establishes a map \(\phi \from \QPiv (\cat C) \to \Pic \ACat C\), \([(\beta, \rho_\beta)] \mapsto [(\beta, \sigma_{\beta, \blank})]\).

  Conversely, let \((\alpha, \sigma_{\alpha,\blank}) \in \ACat C\) be $\cat{C}$-invertible.
  From its half-braiding we obtain a monoidal natural transformation
  \begin{equation}\label{eq: quasi_pivot_from_braiding}
  \tikzsetnextfilename{quasi_pivot_from_braiding}%
  \input{sections/tikzfigures/quasi_pivot_from_braiding.tikz}%

  \end{equation}
  Due to the snake identities, the map \(\psi \from \Pic \ACat C \to \QPiv (\cat C)\), \([(\alpha, \sigma_{\alpha, \blank})] \mapsto [(\alpha, \rho_\alpha)]\) is the inverse of \(\phi\).
\end{proof}

\begin{remark}\label{rem: qpiv_heap_without_canonical_unit}
  Since \(\QPiv (\cat C)\) and \(\Pic \ACat C\) are in bijection, \(\QPiv (\cat C)\) can be endowed with a heap structure.
  However, even if \(\QPiv (\cat C)\) is non-empty, there might not be a canonical element establishing a group structure on it in the sense of Lemma~\ref{lem: heaps_to_groups}.
  This conforms to the fact that there are no canonical coefficients for Hopf cyclic cohomology as mentioned in the introduction.
\end{remark}

Having lifted all Hopf algebraic notions of Theorem~\ref{thm: equiv_in_Hopf_setting},
we can now restate it in its categorical version.
Its proof is an immediate consequence of Proposition~\ref{prop: module_funs_have_adjoints} and Lemma~\ref{lem: quasi_piv_vs_inv_ayds}.

\begin{theorem}\label{thm: equivalence_in_the_rigid_setting}
  Let \(\cat C\) be a strict rigid category. The following sets are in bijection:
  \begin{thmlist}
  \item the equivalence classes of quasi-pivotal structures $\QPiv (\cat C)$ on $\cat{C}$,
  \item the Picard heap \(\Pic \ACat C\) of $\ACat C$, and
  \item the isomorphism classes $\textnormal{ModEquiv}(\ZCat{C}, \ACat{C})/\!\!\sim$   of equivalences of module categories between $\ZCat{C}$ and $\ACat{C}$.
  \end{thmlist}
\end{theorem}

\subsection{Pivotality of the Drinfeld centre}\label{subsec: piv_of_centre}

In Remark~\ref{rem: quasi_pivotal_vs_pivotal} it is noted that pairs in involution give rise to pivotal structures on the \YetterDrinfeld modules.
This relationship follows a categorical principle, which we will examine in this section.
Our approach is similar to Shimizu's investigations in the setting of finite tensor categories, see~\cite{Shimizu2016}.
A major difference being that we focus on the Picard heap of the anti-Drinfeld centre instead of quasi-pivotal structures of the underlying category.

Let us briefly sketch the main benefit of this approach.
Our ensuing constructions lead to a conceptual understanding of the connection between the elements of \(\Pic \ACat C\) and pivotal structures on \(\ZCat C\).
This in turn allows us to determine when two such induced structures coincide by studying actions of the Picard group of the symmetric centre of \(\cat C\) on \(\Piv \ACat C\).
Ultimately, this leads to a heap morphism between the Picard heap of the anti-Drinfeld centre of \(\cat C\) and the pivotal structures on \(\ZCat C\).

Let \(A= (\alpha, \sigma_{\alpha, \blank}) \in \ACat C\) be $\cat{C}$-invertible and write \(\Omega = (\omega, \sigma_{\omega, \blank}) \in \QCat C\) for its left dual.
The coevaluation of \(\alpha\) will play an important role, which is why we gather some of its properties in the next diagram.
Set \(\coev^{-l}_{\alpha} \defeq (\coev^l_{\alpha})^{-1}\).
\begin{equation}\label{eq: unit_braiding_coev}
  \tikzsetnextfilename{coev_and_braiding}%
  \input{sections/tikzfigures/coev_and_braiding.tikz}%

\end{equation}

Appropriate half-braidings allow us to `entwine' \(A\) with any object \(X \in \ZCat C\) in a non-trivial manner, resulting in a morphism from \(X\) to its double dual:
\begin{equation}\label{eq: def_of_nat_iso}
  \tikzsetnextfilename{def_of_nat_iso}%
  \input{sections/tikzfigures/def_of_nat_iso.tikz}%

\end{equation}

As mentioned in Remark~\ref{rmk:twisted-centre-as-loop}, objects in twisted centres correspond to pseudonatural transformations between deloopings of monoidal functors.
The following lemma, which is also discussed in~\cite[Section 4.4]{Shimizu2016} shows that if these objects are `invertible', one can reconstruct monoidal natural isomorphisms from them.
 For the convenience of the reader we provide a proof in terms of string diagrams in Section~\ref{sec:proof:pivotal_structure_from_entwinement}.

\begin{lemma}\label{lem: pivotal_structure_from_entwinement}
  Any $\cat{C}$-invertible object \(A  \in \ACat C\)  defines a pivotal structure on \(\ZCat C\) via
  \begin{equation*}
    X \trightarrow{\rho_{A, X}} \lbiDual{X}, \qquad \qquad \text{ for all } X \in \ZCat C.
  \end{equation*}
\end{lemma}

Given different pivotal structures of \(\ZCat C\), induced by $\cat{C}$-invertible objects in \(\ACat C\), it is a priori challenging to determine whether these structures coincide.
The following lemma is a first step in this direction.
It shows that the induced pivotal structures only depend on the isomorphism classes of $\cat{C}$-invertible objects in \(\ACat C\).

\begin{lemma}\label{lem: pivotal_morphisms_depend_only_on_iso_class}
  Suppose that \(A_1, A_2 \in \ACat C\) are two representatives of the equivalence class \([A_1] = [A_2] \in \Pic \ACat C\).
  Then \(\rho_{A_1} = \rho_{A_2}\).
\end{lemma}
\begin{proof}
  We fix two $\cat{C}$-invertible objects \(A_{1,2}= (\alpha_{1,2}, \sigma_{\alpha_{1,2}, \blank}) \in \ACat C\)  such that there exists an isomorphism  \(\phi \from A_1 \to A_2\) in the anti-Drinfeld centre.
  For any \(X \in \ZCat C\) we have:
  \begin{center}
  \tikzsetnextfilename{pivotality_morphisms_depend_on_iso_class}%
  \input{sections/tikzfigures/pivotality_morphisms_depend_on_iso_class.tikz}%

  \end{center}
  This shows that the induced pivotal structures \(\rho_{A_1}\) and \(\rho_{A_2}\) are the same.
\end{proof}

We will now investigate a class of elements in $\ZCat{C}$ which act on $\Pic(A)$ without altering the induced pivotal structure.

\begin{definition}\label{def: symmetric_object}
  We call an object \(X \in \ZCat C\) \emph{symmetric} if we have
  \begin{equation}\label{eq: symmetric_object}
    \sigma_{X,Y}^{-1} = \sigma_{Y,X}^{\vphantom{-1}}, \qquad \qquad \text{ for all } Y \in \ZCat C.
  \end{equation}
\end{definition}

Following the terminology of~\cite{Mueger2013}, we call the full (symmetric) monoidal subcategory \(\SZCat C\) of \(\ZCat C\) whose objects are symmetric the \emph{symmetric centre} of \(\ZCat C\).

\begin{lemma}\label{lem: symmetric_centre_and_rigidity}
  Suppose \(\cat C\) to be rigid, then \(\SZCat C\) is rigid as well.
\end{lemma}
\begin{proof}
  Suppose \(X\in \ZCat C\) to be symmetric and let \(Y \in \ZCat C\).
  We compute
  \begin{center}
  \tikzsetnextfilename{inverse_of_symmetric_object_is_symmetric}%
  \input{sections/tikzfigures/inverse_of_symmetric_object_is_symmetric.tikz}%

  \end{center}
  This implies \(\sigma_{\lDual X,Y}^{-1} = \sigma_{Y,\lDual X}^{\vphantom{-1}}\).
  Since the left dual of any \(X \in \SZCat C \subseteq \ZCat C\)  can be equipped with the structure of a right dual and \(\SZCat C\) is a full subcategory of \(\ZCat C\), it must be rigid.
\end{proof}

Let us now consider the Picard group \(\Pic \SZCat C\) of the symmetric centre of \(\ZCat C\).
It acts on \(\Pic \ACat{C}\) via tensoring from the left, as shown in Diagram~\eqref{eq: gluing_of_half_brainings}.
We consider two elements \(A, C \in \Pic \ACat{C}\) equivalent if they are contained in the same orbit.
That is
\begin{equation}\label{eq: equiv_action_Pic'_of_Z}
  [A] \sim [C] \iff \text{ there exists a } [B] \in\Pic \SZCat C \text{ such that }
  [B \otimes A] = [C].
\end{equation}

To show that two elements of \(\Pic \ACat C\) induce the same pivotal structure on \(\ZCat C\) if and only if they are contained in the same orbit under the  \(\Pic \SZCat C\)-action, we need two technical observations.
First, an alternate description of symmetric invertible objects.
Second, a more detailed investigation into the inverse of an induced pivotal structure.

\begin{lemma}\label{lem: alternative_char_of_sym_objs}
  An invertible object \((\beta, \sigma_{\beta, \blank}) \in \ZCat C\), is symmetric if and only if it satisfies for all \(X \in \ZCat C\)
  \begin{equation}\label{eq: alternative_char_of_sym_objs}
  \tikzsetnextfilename{twisted_trace_of_beta}%
  \input{sections/tikzfigures/twisted_trace_of_beta.tikz}%

  \end{equation}
\end{lemma}
\begin{proof}
  Let \(B=(\beta, \sigma_{\beta, \blank}) \in \ZCat C\) be invertible and \(X \in \ZCat C\).
  The left-hand side of Equation~\eqref{eq: alternative_char_of_sym_objs} can be rephrased as:
  \begin{equation}\label{eq: rephrase_of_sym_obj_entanglement}
  \tikzsetnextfilename{alternate_char_of_sym_obj}%
  \input{sections/tikzfigures/alternate_char_of_sym_obj.tikz}%

  \end{equation}
  We define the morphism \(f \defeq \id_X \otimes \coev_\beta^l \from X \to X \otimes \beta \otimes \lDual \beta\) and observe that
  Equation~\eqref{eq: alternative_char_of_sym_objs} is identical to
  \begin{equation*}
    f^{-1} ((\sigma_{\beta, X} \sigma_{X, \beta})^{-1} \otimes \id_{\lDual \beta}) f = \id_X.
  \end{equation*}
  This is equivalent to \(\sigma_{\beta, X} \sigma_{X, \beta} \otimes \id_{\lDual \beta} = \id_{X \otimes \beta} \otimes \id_{\lDual \beta}\).
  As the functor \(\blank \otimes \lDual \beta\) is conservative, the claim follows.
\end{proof}

\begin{lemma}\label{lem: inverse_of_ind_piv_struc}
  Let \(A=(\alpha, \sigma_{\alpha,\blank}) \in \ACat C\) be  $\cat{C}$-invertible  and write \(\Omega =(\omega, \sigma_{\omega, \blank }) \in \QCat C\) for its dual.
  For any \(X \in \ZCat C\), the inverse of \(\rho_{A, X}\) is
  \begin{equation}
  \tikzsetnextfilename{inverse_of_pivotality_morphism}%
  \input{sections/tikzfigures/inverse_of_pivotality_morphism.tikz}%

  \end{equation}
\end{lemma}
\begin{proof}
  Let \(X \in \ZCat C\).
  The snake identities and a variant of Equation~\eqref{eq: unit_braiding_coev} imply:
  \begin{equation}\label{eq: diagram_inverse}
  \tikzsetnextfilename{computation_inverse_piotality_morphism}%
  \input{sections/tikzfigures/computation_inverse_piotality_morphism.tikz}%

  \end{equation}
  Thus, writing \(\Omega = (\lDual \alpha, \sigma_{\lDual \alpha, \blank }) \in \ZCat C\), we have \(\rho_{A,X}\rho_{\Omega, X} = \id_X\).
\end{proof}

\begin{lemma}\label{lem: induced_piv_structure_depends_only_on_orbit}
  Two elements \([A], [C] \in \Pic \ACat C\) induce the same pivotal structure on \(\ZCat C\) if and only if there exists a \([B] \in \Pic \SZCat C\) such that \([B \otimes A] = [C]\).
\end{lemma}
\begin{proof}
  Let \([A], [C] \in \Pic \ACat C\).
  Suppose there exists a \([B] \in \Pic \SZCat C\) such that \([B \otimes A] = [C]\).
  For any \(X\in \ZCat C\), we compute:
  \begin{equation}\label{eq: seperating_entwinments_via_hex_ids}
  \tikzsetnextfilename{orbits_and_piv_structure}%
  \input{sections/tikzfigures/orbits_and_piv_structure.tikz}%

  \end{equation}
  If conversely \(\rho_A =\rho_C\), we claim that \(C \otimes \lDual A\) is symmetric.
  By Lemma~\ref{lem: alternative_char_of_sym_objs} we have to show that for every \(X\in \ZCat C\) the `entwinement' \(\rho_{C \otimes \lDual A}\) of \(C \otimes \lDual A\) with \(X\) is the identity and indeed we observe
  \begin{equation*}
    \rho_{C \otimes \lDual A, X}^{\vphantom{-1}} = \rho_{\lDual A, X}^{\vphantom{-1}} \rho_{C, X}^{\vphantom{-1}} = \rho_{A, X}^{-1} \rho_{C, X}^{\vphantom{-1}} = \id_X^{\vphantom{-1}}.
  \end{equation*}
  For the first equality we used the hexagon identities as in Equation~\eqref{eq: seperating_entwinments_via_hex_ids} to separate \(\rho_{C \otimes \lDual A, X}\) into two parts.
  The second one follows from the description of the inverse of \(\rho_{A,X}\) given in Lemma~\ref{lem: inverse_of_ind_piv_struc}.
  Finally, since \(\id_{C}\otimes \ev_{A}^l \from C \otimes \lDual A \otimes A \to C\) is an isomorphism in \(\ACat C\), we have \([(C\otimes \lDual A) \otimes A] = [C]\).
\end{proof}

The isomorphisms classes of $\cat{C}$-invertible objects of \(\ACat C\) are not just a set but form the Picard heap \(\Pic \ACat C\).
Our next lemma shows that its heap operation projects onto the orbits under the
\(\Pic \SZCat C\)-action.

\begin{lemma}\label{lem: quotient_of_Pic_A_is_heap}
  The canonical projection \(\pi \from \Pic \ACat C \to \Pic \ACat C/\Pic \SZCat C\) induces a heap structure on the set of equivalence classes \(\Pic \ACat{C}/\Pic \SZCat C\).
\end{lemma}
\begin{proof}
  The claim follows from a general observation.
  Let \(X \in \ZCat C\) and \(A \in \ACat C\). The half-braiding \(\sigma_{X,A} \from X \otimes A \to A \otimes X\) is an isomorphism in \(\ACat C\):
  \begin{equation}
  \tikzsetnextfilename{braiding_is_iso_in_ACat}%
  \input{sections/tikzfigures/braiding_is_iso_in_ACat.tikz}%

  \end{equation}
  Likewise, \(\sigma_{X, \lDual A} \from X \otimes \lDual A \to \lDual A \otimes X\) is an isomorphism in \(\QCat C\).
  As a consequence, for all \([A], [A'], [A''] \in \Pic \ACat{C}\) and \([B], [B'], [B''] \in\Pic \SZCat C\) we have
  \begin{align*}
    \pi &\left(\left\langle [A], [A'], [A''] \right\rangle \right)
        = \pi \left( [A \otimes\lDual{A'}\otimes A''] \right)
        =\pi \left( [B \otimes \lDual{B'} \otimes B'' \otimes A \otimes\lDual{A'}\otimes A''] \right)\\ & \quad
                                                                          =\pi \left( [B \otimes A \otimes \lDual{(B' \otimes A')} \otimes B'' \otimes A''] \right)
                                                                          =\pi \left(\left\langle [B \otimes A], [B' \otimes A'], [B'' \otimes A''] \right\rangle \right).
  \end{align*}
\end{proof}

Recall that due to Example~\ref{ex: heap_of_nat_isos}, the pivotal structures \(\Piv \ZCat C\) on \(\ZCat C\) admit a heap operation.
This allows us to distil our previous observations into a single result.

\begin{theorem}\label{thm: equiv_establishes_piv}
  The morphism of heaps
  \begin{equation}
    \kappa \from \Pic \ACat C \to \Piv \ZCat C, \qquad [A] \mapsto \rho_A
  \end{equation}
  induces a unique injective morphism  \(\iota \from \Pic \ACat C /\Pic \SZCat C \to \Piv \ZCat C\) such that the following diagram commutes in the category of heaps:
  \begin{equation}\label{eq: induced_heap_morphism_condition_cd}
    \begin{tikzcd}[ampersand replacement=\&]
      \Pic \ACat C
      \arrow[dd, "\pi"', two heads]
      \arrow[rr, "\kappa"]
      \& \&
      \Piv \ZCat C \\
      \& \& \\
      \Pic \ACat C /\Pic \SZCat C
      \arrow[rruu, "\exists! \iota"', dashed, hook]
      \& \&
    \end{tikzcd}
  \end{equation}
\end{theorem}
\begin{proof}
  Lemmas~\ref{lem: pivotal_structure_from_entwinement} and~\ref{lem: pivotal_morphisms_depend_only_on_iso_class} show that \(\kappa\) is well-defined.
  Given three elements \([A], [B], [C] \in \Pic \ACat C\), we compute
  \begin{equation*}
    \kappa(\langle [A], [B], [C] \rangle )
    = \rho_{A \otimes \lDual B \otimes C}
    = \rho_{A} \rho_{\lDual B} \rho_{C}
    = \rho_{A} \rho_{B}^{-1} \rho_{C}
    = \langle \rho_{A}, \rho_{B}^{-1}, \rho_{C}\rangle.
  \end{equation*}
  Here we applied the hexagon identities as in Equation~\eqref{eq: seperating_entwinments_via_hex_ids} for the second step and Lemma~\ref{lem: inverse_of_ind_piv_struc} for the third one.
  We observe, \(\kappa\) is a morphism of heaps.
  Lemma~\ref{lem: induced_piv_structure_depends_only_on_orbit} states that for any two elements \([A], [B]\in \Pic \ACat C\) we have \(\kappa([A]) =\kappa([B])\) if and only if \(\pi([A]) = \pi([B])\).
  It follows from Lemma~\ref{lem: quotient_of_Pic_A_is_heap} that the unique injective map
  \[
    \iota \from \Pic \ACat C /\Pic \SZCat C \to \Piv \ZCat C
  \]
  that lets Diagram~\eqref{eq: induced_heap_morphism_condition_cd} commute is a morphism of heaps.
\end{proof}

\begin{remark}\label{rem: projection_modular_tensor_category}
  The centre \(\ZCat C\) of a \emph{finite tensor category} \(\cat C\) over an algebraically closed field is factorisable due to  \cite[Proposition~4.4]{Etingof2004}.
  By \cite[Theorem~1.1]{Shimizu2019}, the Picard group  \(\Pic \SZCat C\) is trivial.
  In this setting, the induced pivotal structures depend only on the Picard heap \(\Pic \ACat C\) and not on a quotient thereof.

  On the other side of the spectrum, one might consider the discrete category \(\cat G\) of an abelian group \(G\); its set of objects is \(G\) and all morphisms are identities.
  The category \(\cat G\) is rigid monoidal with the tensor product given by the multiplication of \(G\) and the left and right duals given by the respective inverses.
  A direct computation shows that \(\SZCat{G}= \ZCat{G}\cong \cat G\).
  Since \(\cat G\) is skeletal\footnote{%
    A category \(\cat C\) is \emph{skeletal} if \(X \cong Y\) implies that \(X=Y\),
    for all objects \(X, Y \in \cat C\).%
  }
  and every object is invertible, \(\Pic \SZCat G \cong G\).
  As the double dual and identity functor on \(\cat G\) coincide, the same argument implies \(\Pic \ACat G \cong G\) and thus \(\Pic \ACat G / \Pic \SZCat G \cong \{1\}\).
\end{remark}

It was proven by Shimizu in~\cite[Theorem~4.1]{Shimizu2016} that under certain circumstances all pivotal structures on the centre of \(\cat C\) are induced by the quasi-pivotal structures of \(\cat C\).
In our terminology, his result can be formulated as:
\begin{proposition}
  The map \(\iota \from \Pic \ACat C /\Pic \SZCat C \to \Piv \ZCat C\) is bijective if \(\cat C\) is a finite tensor category.
\end{proposition}

However, in the introduction of \cite{Shimizu2016} the author states that it is not to be expected that this holds true in general.
In the remainder of this section, we will construct an explicit counterexample.
The key observation needed to find a fitting category \(\cat C\) is the following:
Suppose there is an object \(X \in \cat C\) which can be endowed with two different half-braidings \(\sigma_{X, \blank}\) and \(\chi_{X, \blank}\).
Assume furthermore that  there is a pivotal structure \(\zeta \from \Id_{\ZCat C} \to \lbiDual{(\blank)}\) on \(\ZCat C\) such that \(\zeta_{(X,\sigma_{X,\blank})} \neq \zeta_{(X, \chi_{X, \blank})}\).
If the unit of \(\cat C\) is the only invertible object, there is no (quasi-)pivotal structure inducing \(\zeta\) and therefore
\(\iota  \from \Pic \ACat C /\Pic \SZCat C \to \Piv \ZCat C\) cannot be surjective.

We will now  define such a category \(\cat C\) in terms of generators and relations.
The details of this type of construction are explained in \cite[Chapter~XII]{Kassel1998}.
As a first step, consider a `free' monoidal category \(\cat C^\text{free}\).
Its objects are monomials in the variable \(X\).
Their tensor product is given by \(X^n\otimes X^m =X^{n+m}\).
The morphisms of \(\cat C^{\text{free}}\) are formal compositions and tensor products of `atomic' building blocks, subject to suitable associativity and unitality relations.
These `atoms' are identities on objects plus the set \(\mathbb M\) of \emph{generating morphisms} depicted below.
\begin{equation} \label{eq: generating_morphisms}
  \tikzsetnextfilename{generating_morphisms}%
  \input{sections/tikzfigures/generating_morphisms.tikz}%

\end{equation}

By \cite[Lemma~XII.1.2]{Kassel1998}, every morphism \(f \from X^n \to X^m\) in \(\cat C^\text{free}\) is either the identity or can be written as
\begin{equation*}
  f = (\id_{X^{j_l}}\otimes f_l \otimes \id_{X^{i_l}})
  \dots
  (\id_{X^{j_2}}\otimes f_2 \otimes \id_{X^{i_2}})
  (\id_{X^{j_1}}\otimes f_1 \otimes \id_{X^{i_1}}),
\end{equation*}
where \(i_1,j_1, \dots, i_l,j_l \in \mathbb N\) and \(f_1, \dots, f_l \in \mathbb M\).
Such a presentation is not unique but the number \(l \in \mathbb N\) of generating morphisms needed to write \(f\) in such a manner is.
We call it the \emph{degree} of \(f\) and write \(\deg(f)=l\).

To pass to the category \(\cat C\), we take a quotient of \(\cat C^{\text{free}}\) by the relations depicted below.
This will turn \(\cat C\) into a pivotal, strict rigid category and allows us to extend \(\sigma\) to a symmetric braiding.
To increase readability, we omit labeling the strings with \(X\).
\begin{gather}
  \tikzsetnextfilename{counter_rels_idsquare_leftsnake}%
  \input{sections/tikzfigures/counter_rels_idsquare_leftsnake.tikz}%
 \label{eq: squares_are_ids_left_snake} \displaybreak[0]\\
  \tikzsetnextfilename{counter_rels_self_duals}%
  \input{sections/tikzfigures/counter_rels_self_duals.tikz}%
 \label{eq: everything_is_self_dual} \\
  \tikzsetnextfilename{counter_rels_braidrho}%
  \input{sections/tikzfigures/counter_rels_braidrho.tikz}%
 \label{eq: rho_is_nat}
\end{gather}

Due to \cite[Proposition~XII.1.4]{Kassel1998}, we observe that there is a unique functor \(P \from \cat C^{\text{free}}\to \cat C\) which maps objects to themselves and generating morphisms to their respective equivalence classes.

\begin{definition}\label{def: degree of morphism}
  Consider a morphism \(f \in \Hom{\cat C}{X^n}{X^m}\).
  A \emph{presentation} of \(f\) is a morphism \(g \in \Hom{\cat C^{\text{free}}}{X^n}{Y^n}\) such that \(f = P(g)\).
  If the degree of \(g\) is minimal amongst the presentations of \(f\), we call it a \emph{minimal presentation}.
\end{definition}

Before we classify half-braidings of objects in \(\cat C\) by studying their minimal presentations, we first need to gather some information about the structure of \(\cat C\).

\begin{theorem}\label{thm: counterexample_underlying_cat_properties}
  The category \(\cat C\) is strict rigid and the double dual functor is the identity.
  Furthermore, \(\id_X,\rho_X\from X \to X\) can be extended to pivotal structures and \(\sigma_{X,X}\from X^2 \to X^2\) to a symmetric braiding.
\end{theorem}
\begin{proof}
  The evaluation and coevaluation morphisms plus their snake identities make \(X \in \cat C\), and by extension every object of \(\cat C\), its own left, respectively right, dual.
  Using the Relations~\eqref{eq: everything_is_self_dual} together with the snake identities, we compute
  \begin{gather*}
    \lDual{\rho_X} = \rho_X = \rDual \rho_X, \qquad \qquad \qquad \qquad \lDual{ \sigma_{X,X}} =\sigma_{X,X} =\rDual \sigma_{X,X},\\
    \lDual{\ev_X} =\coev_X = \rDual \ev_X \qquad \qquad \text{and} \qquad \qquad \lDual{\coev_X} = \ev_X = \rDual \coev_X.
  \end{gather*}
  Thus, \(\cat C\) is a strict rigid category whose double dual functor is equal to the identity.

  Our candidate for a pivotal structure on \(\cat C\), different from the trivial one, is
  \begin{equation*}
    \rho \from \Id_\cat C \to \Id_\cat C \qquad \text{defined by}\qquad
    \rho_{X^n}\defeq \rho_X \otimes \dots \otimes \rho_X \from X^n \to X^n ,\qquad n \in \mathbb N.
  \end{equation*}
  This family of isomorphisms is compatible with the monoidal structure of \(\cat C\) by construction and we only have to investigate its naturality.
  It suffices to verify this property on the generators. Relations~\eqref{eq: rho_is_nat} imply that \(\rho_{X^2}\) commutes with \(\sigma_{X,X}\).
  For the evaluation of \(X\in \cat C\) we use the dual of Equation~\eqref{eq: everything_is_self_dual} to compute
  \begin{equation*}
    \ev_X\rho_{X^2}= \ev_X (\rho_X \otimes \rho_X)= \ev_X (\lDual \rho_X \otimes \rho_X) = \ev_X (\id_X \otimes \rho^2_X) = \rho_1\ev_X.
  \end{equation*}
  Applying the left dualising functor, we get \(\coev_X \rho_1 = \rho_{X^2} \coev_X\) and thus \(\rho \from \Id_\cat C \to \Id_\cat C\) defines a pivotal structure.

  Lastly, we establish that \(\sigma_{X,X}\) implements a symmetry \(\sigma \from \otimes \to \otimes^{\op}\) on \(\cat C\).
  Set
  \begin{equation*}
    \sigma_{X,X^m}\defeq (\id_X\otimes\sigma_{X, X^{m-1}})(\sigma_{X,X}\otimes \id_{X^{m-1}}), \qquad\qquad m\in \mathbb N
  \end{equation*}
  and extend this to arbitrary objects:
  \begin{equation*}
    \sigma_{X^n,X^m}\defeq (\sigma_{X^{n-1,m}}\otimes \id_X)(\id_{X^{n-1}}\otimes \sigma_{X,X^m}), \qquad \qquad n,m \in \mathbb N.
  \end{equation*}
  As this family of isomorphisms is constructed according to the hexagon axioms, we only have to prove its naturality.
  Again, it suffices to consider the generating morphisms.
  By Equation~\eqref{eq: rho_is_nat}, \(\sigma\) is natural with respect to \(\rho_X\), \(\sigma_{X,X}\) and \(\coev_X\).
  The self-duality of \(\sigma_{X,X}\) and \(\lDual \coev_X =\ev_X\) imply the desired commutation between \(\sigma\) and \(\ev_X\).
  Thus \(\sigma\) is a braiding on \(\cat C\), which is symmetric by Equation~\eqref{eq: squares_are_ids_left_snake}.
\end{proof}

We think of a generic morphism of \(\cat C\) to be of the form:
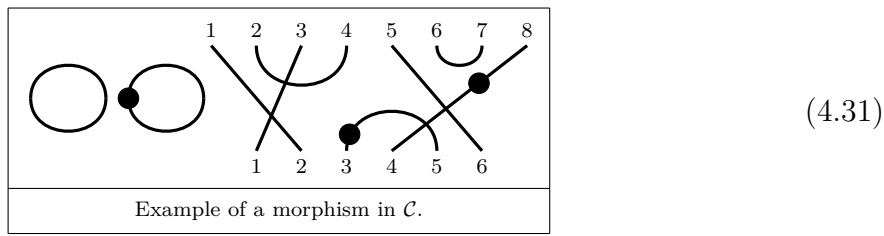
\begin{equation} \label{eq: generic_morphism_of_cat_C}
  \tikzsetnextfilename{generic_morphism_C}%
  \input{sections/tikzfigures/generic_morphism_C.tikz}%

\end{equation}
This diagram suggests a distinction between different kinds of morphisms: there are \emph{connectors}, which link an input to an output vertex, \emph{closed loops} and \emph{half-circles} of \emph{evaluation}- and \emph{coevaluation-type}.
Connectors induce a permutation on a subset of \(\mathbb N\).
For example, the permutation arising from Diagram~\eqref{eq: generic_morphism_of_cat_C} can be identified with \((1 \; 2)(3 \; 4)\).

Conversely, suppose \(s= t_{i_1} \dots t_{i_l}\in\Sym(n)\) to be a permutation written as a product of elementary transpositions and set \(f_s \defeq f_{t_{i_1}}\dots f_{t_{i_l}} \from X^n \to X^n\), where
\begin{equation*}
  f_{t_i} \defeq \id_{X^{i-1}}\otimes \sigma_{X,X}\otimes \id_{X^{n-(i+1)}}\from X^n \to X^n, \qquad \text{for} \qquad  1\leq i \leq n-1.
\end{equation*}
Since the braiding \(\sigma\) is symmetric \(f_s\) does not depend on the presentation of \(s\).
However, should the presentation of \(s\) be minimal, then so is the corresponding presentation of \(f_s\).
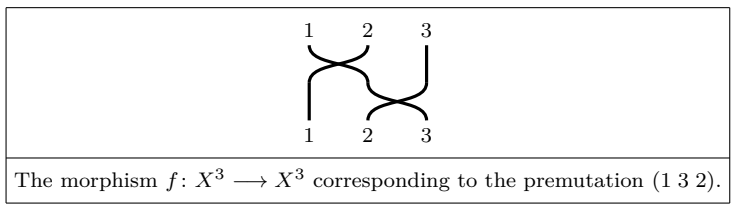
\begin{equation*}
  \tikzsetnextfilename{permutation_example}%
  \input{sections/tikzfigures/permutation_example.tikz}%

\end{equation*}

To derive a normal form of the automorphisms of \(\cat C\) and turn our previously explained thoughts into precise mathematical statements, we need to study the `topological features' of the  morphisms in \(\cat C\).

\begin{remark}\label{rem: category_of_tangles}
  We recall the \emph{category \(\cat T\) of tangles}, a close relative to the string diagrams arising from \(\cat C\), based on \cite[Chapter XII.2]{Kassel1998}.
  Its objects are finite sequences in \(\{+, -\}\) and its morphisms are isotopy classes of oriented tangles.
  A detailed discussion of tangles is given in \cite[Definition X.5.1]{Kassel1998}.
  For us, it suffices to think of an oriented tangle \(L\) of type \((n,m)\) as a finite disjoint union of embeddings of either the unit circle \(S^1\) or the interval \([0,1]\) into \(\mathbb R^2 \times [0,1]\) such that
  \begin{equation}\label{eq: tangle_boundary_condition}
    \partial L = L \, \cap \left( \mathbb R^2 \times \{0,1\} \right) = \left([n] \times \left\{(0,0)\right\}\right) \cup \left([l]\times \left\{(0,1)\right\}\right),
  \end{equation}
  where \([n] = \{1, \dots, n\}\) and \([l] = \{1, \dots, l\}\).
  The orientation on each of the connected components of \(L\) is induced by the counter-clockwise orientation of \(S^1\) and the (ascending) orientation of \([0,1]\).
  The tensor product of tangles is given by pasting them next to each other.
  Their composition is implemented, by appropriate gluing and rescaling.
\end{remark}

To distinguish isotopy classes of tangles, one can study their images under the projection \(\mathbb R^2 \times [0,1] \to \mathbb R \times [0,1]\).
This leads to a combinatorial description of \(\cat T\), see for example \cite[Theorem~XII.2.2]{Kassel1998}.

\begin{proposition}\label{prop: presentation_tangle_category}
  The strict monoidal category \(\cat T\) is  generated by the morphisms:
  \begin{equation*}
  \tikzsetnextfilename{generating_morphisms_tangle_category}%
  \input{sections/tikzfigures/generating_morphisms_tangle_category.tikz}%

  \end{equation*}
  These are subject to the following relations:
  \begin{gather}
  \tikzsetnextfilename{tangle_category_snake_identities}%
  \input{sections/tikzfigures/tangle_category_snake_identities.tikz}%
 \displaybreak[0]\\
  \tikzsetnextfilename{tangle_category_invertability}%
  \input{sections/tikzfigures/tangle_category_invertability.tikz}%
 \displaybreak[0]\\
  \tikzsetnextfilename{tangle_category_dual_braiding}%
  \input{sections/tikzfigures/tangle_category_dual_braiding.tikz}%
 \displaybreak[0]\\
  \tikzsetnextfilename{tangle_category_braiding_dual}%
  \input{sections/tikzfigures/tangle_category_braiding_dual.tikz}%
 \displaybreak[0]\\
  \tikzsetnextfilename{tangle_category_twist_trivial}%
  \input{sections/tikzfigures/tangle_category_twist_trivial.tikz}%

  \end{gather}
\end{proposition}

The connection between tangles and the category \(\cat C\) is attained through applying \cite[Proposition~XII.1.4]{Kassel1998}.

\begin{lemma}\label{lem: functor_tangles_to_C}
  There exists a strict monoidal functor \(S \from \cat T \to \cat C\) which is uniquely determined by
  \(S(+) = X = S(-)\) and
  \begin{gather*}\
    S(\ev_{\pm})= \ev_X,  \qquad  S(\coev_{\pm})= \coev_X , \qquad
    S(\tau^{\pm}_{+,+}) = \sigma_{X,X}.
  \end{gather*}
\end{lemma}

To investigate the `topological features' of  \(\cat C\), we want to lift its morphisms  to \(\cat T\).
Hereto we want to `trivialise' the generator \(\rho_{X,X}\from X \to X\).
Set \(\cat C/\langle \rho_X \rangle\) to be the category obtained from \(\cat C\) by identifying \(\rho_X\) with \(\id_X\). The `projection' functor \(\mathrm{Pr} \from \cat C \to\cat C/\langle \rho_X \rangle\) allows us to define an equivalence relation on the morphisms of \(\cat C\):
\begin{equation}\label{eq: equivalence_relation_morphs_of_C}
  f \sim g \qquad \iff \qquad \mathrm{Pr}(f) = \mathrm{Pr}(g).
\end{equation}

For example the following two  endomorphisms \(\bigcirc, \bullet \!\! \bigcirc \from 1 \to 1\) of the monoidal unit of \(\cat C\) would be equivalent with respect to this relation:
\begin{equation}\label{eq: the_two_types_of_endomorphisms_of_the_unit}
  \tikzsetnextfilename{clsed_loop}%
  \input{sections/tikzfigures/clsed_loop.tikz}%

\end{equation}

\begin{proposition}\label{prop: unique_decomposition_of_automorphisms}
  Every automorphism \(f\in \cat C(X^n, X^n)\) can be uniquely written as
  \begin{equation}\label{eq: decomp_of_automorphisms}
    f = f_s f_\phi,
  \end{equation}
  where \(f_s \from X^n \to X^n\) is the automorphism induced by a permutation \(s\in \Sym(n)\) and
  \begin{equation}\label{eq: morph_powers_rho}
    f_\phi = \rho_X^{\phi_1} \otimes \dots \otimes \rho_X^{\phi_n}, \qquad \text{ with } \phi_1, \dots, \phi_n \in \mathbb Z_2.
  \end{equation}
  Furthermore, if a minimal presentation \(s= t_{i_1}\dots t_{i_l}\) is fixed, the resulting presentation of \(f\) is minimal as well.
\end{proposition}
\begin{proof}
  For any \(f \in \Aut{\cat C}{X^n}\) there exists another automorphism \(g \in \Aut{\cat C}{X^n}\) such that \(\mathrm{Pr}(f) = \mathrm{Pr}(g)\) and \(g\) has a presentation in which no copies of \(\rho\) occur.
  By proceeding analogous to \cite[Lemma~X.3.3]{Kassel1998}, we construct a tangle \(L_g\) out of \(g\) such that \(S(L_g)=g\) and it is isotopic to a tangle \(L_g'\), whose images of its connected components under the projection \(\mathbb R^2\times [0,1] \to \mathbb R\times [0,1]\) are either closed loops, half-circles of evaluation- or coevaluation-type or straight lines.
  Write \(L_n^{\text{triv}}\) for a tangle which projects to \(n\) parallel straight lines
  \begin{equation*}
    \{(k,t) \mid t \in [0,1] \text{ and } k \in \{1, \dots, n\} \}.
  \end{equation*}
  Since \(g\) was invertible by assumption, we can lift its inverse \(g^{-1}\from X^n \to X^n\) to a tangle \(L_{g^{-1}}\) with \([L_g][L_{g^{-1}}] = [L_n^{\text{triv}}] = [L_{g^{-1}}][L_g] \).
  This equation readily implies that \(L_g'\) could not have contained any loops or half-circles.
  In other words \(g=f_s\), where \(f_s\) is the morphism obtained from the permutation \(s\in \Sym(n)\), induced by the projection of \(L_g'\) onto \(\mathbb R \times [0,1]\).
  Due to the naturality of \(\sigma_{X,X}\), the equivalence between \(f\) and \(g\) implies \(f= f_s f_\phi\), with \(f_\phi\) being a tensor product of identities and copies of \(\rho_X\).
  Consequentially, a minimal representation of \(s\) induces a minimal representation of \(f\).
\end{proof}

The first step in showing that \(\iota  \from \Pic \ACat C /\Pic \SZCat C \to \Piv \ZCat C\) cannot be surjective is to prove that the Picard heap \(\Pic \ACat C\) contains at most two elements.

\begin{corollary}\label{cor: piv_on_C}
  The only (quasi-)pivotal structures on \(\cat C\) are \(\id \from \Id_\cat C \to \Id_\cat C\) and \(\rho\from \Id_\cat C \to \Id_\cat C\).
\end{corollary}
\begin{proof}
  The only invertible object of \(\cat C\) is its monoidal unit, which implies that any quasi-pivotal structure on \(\cat C\) is already pivotal.
  The claim follows since these are determined by their value on \(X\) and, by Proposition~\ref{prop: unique_decomposition_of_automorphisms}, \(\Aut{\cat C}{X}=\{\id_X, \rho_X\}\).
\end{proof}

Let us now focus on the various ways in which we can equip an object \(Y\in \cat C\) with a half-braiding.
Our classification of automorphisms in \(\cat C\) allows us to easily verify that  on \(X\in \cat C\) there are four different half-braidings.
These are determined by
\begin{equation}\label{eq: half-braidings_on_X}
  \tikzsetnextfilename{types_of_braidings_1}%
  \input{sections/tikzfigures/types_of_braidings_1.tikz}%
 \quad\text{and}\quad
  \tikzsetnextfilename{types_of_braidings_2}%
  \input{sections/tikzfigures/types_of_braidings_2.tikz}%

\end{equation}

The fact that these braidings are distinguished by the appearances of \(\rho\) on the respective strings, motivates our next definition.

\begin{definition}\label{def: characteristic_sequence}
  Let \(f= f_s f_\phi \from X^n \to X^n\) be an automorphism in \(\cat C\).
  Its \emph{characteristic sequence} is
  \(\phi \defeq(\phi_1, \dots, \phi_n) \in \left( \mathbb Z_2 \right)^n\) with
  \begin{equation}\label{eq: characteristic_sequence}
    f_\phi = \rho_X^{\phi_1}\otimes \dots \otimes \rho_X^{\phi_n}.
  \end{equation}
\end{definition}

Indeed, it is the interplay between instances of \(\rho\) and the underlying permutation that determine whether an automorphism \(\chi_{Y,X}\from Y \otimes X \to X \otimes Y\) can be lifted to a half-braiding.

\begin{lemma}\label{lem: half-braidings_on_C}
  Any automorphism \(\chi_{Y, X} \from Y \otimes X \to X \otimes Y\) extends to a half-braiding on \(Y\) if and only if
  there exists an \(f \in \Aut{\cat C}{Y}\) with characteristic sequence \((\phi_1, \dots, \phi_n)\) and underlying permutation \(s \in \Sym(n)\) such that for all \(1 \leq i \leq n\)
  \begin{equation}\label{eq: involution_and_invariance_of_char_sequence}
    s^2(i) = i, \qquad \qquad \phi_{s(i)} = \phi_i,
  \end{equation}
  and \(\chi_{Y, X}= \sigma_{Y,X}(f\otimes \rho_X^j)\) for an integer \(j \in \mathbb Z_2\).
\end{lemma}
\begin{proof}
  Assume \(\chi_{Y, X} \from Y \otimes X \to X \otimes Y\)  to induce a half-braiding on \(Y= X^n\).
  Due to Proposition~\ref{prop: unique_decomposition_of_automorphisms}, we can write \(\chi_{Y, X}= \sigma_{Y,X}(f\otimes \rho_X^j)\), where \(f\from Y \to Y\) is an automorphism of \(Y\) and \(j \in \mathbb Z_2\).
  Let \(\phi= (\phi_1, \dots, \phi_n)\) be the characteristic sequence of \(f\) and \(s\in \Sym(n)\) its underlying permutation.
  Write \(f_s\from Y \to Y\) for the morphism induced by \(s\) and set
  \begin{equation*}
    f_\phi = \rho_X^{\phi_1} \otimes \dots \otimes \rho_X^{\phi_n}, \qquad \qquad f_{s^{-1}(\phi)} = \rho_X^{\phi_{s^{-1}(1)}} \otimes \dots \otimes \rho_X^{\phi_{s^{-1}(n)}}.
  \end{equation*}
  We write \(W \defeq X^{n-1}\) and, using that \(f= f_s f_\phi\) plus the naturality of \(\chi_{Y, \blank}\) and Equation~\eqref{eq: everything_is_self_dual}, compute:
  \begin{equation}\label{eq: conditions_on_braiding}
  \tikzsetnextfilename{computation_conditions_for_braiding}%
  \input{sections/tikzfigures/computation_conditions_for_braiding.tikz}%

  \end{equation}
  This is equivalent to \(s\) being an involution and \(\phi\) being invariant under \(s\).

  Conversely, let \(\chi_{Y, X}= \sigma_{Y,X}(f\otimes \rho_X^j) \from Y \otimes X \to X \otimes Y\), where \(f\) is an automorphism satisfying the assumptions of the lemma.
  We extend it to a family of automorphisms \(\chi_{Y, \blank} \from Y \otimes \blank \to \blank \otimes Y\)  according to the hexagon axioms and verify its naturality on the generators of \(\cat C\).
  For \(\rho_X\) and \(\sigma_{X,X}\) this is immediate consequence of their respective naturality conditions.
  To prove the commutation relations between \(\chi_{Y, \blank}\), \(\coev_X\) and \(\ev_X\), we argue as in Equation~\eqref{eq: conditions_on_braiding}
\end{proof}

The previous lemma severely restricts the number of possibilities in which an automorphism of \(\cat C\) can lift to the centre \(\ZCat C\).

\begin{corollary}\label{cor: criterion_lifting_automorphisms}
  Consider an object \(X^n\in \cat C\) equipped with two half-braidings
  \begin{equation*}
    \chi_{X^n, X} = \sigma_{X^n,X}(f_sf_\phi\otimes \rho_X^j), \qquad \qquad
    \theta_{X^n, X} =\sigma_{X^n,X}(f_tf_\psi\otimes \rho_X^k).
  \end{equation*}
  If \(g= g_r g_\lambda \in \Aut{\cat C}{X^n}\) lifts to a morphism \(g \from (X^n, \chi_{X^n, \blank})\to (X^n, \theta_{X^n, \blank})\) of objects in the centre \(\ZCat C\) of \(\cat C\), then
  \begin{equation}\label{eq: restriction_for_lifts}
    \phi_{i}\lambda_{sr(i)} = \psi_{r(i)}\lambda_{r(i)}, \qquad \qquad \text{ for all }1 \leq i \leq n.
  \end{equation}
\end{corollary}
\begin{proof}
  For the automorphism \(g= f_r f_\lambda \in \Aut{\cat C}{X^n}\) to lift to the centre it must satisfy
  \begin{equation*}
    \sigma_{X^n,X}(f_sf_\phi g\otimes \rho_X^j)
    = \chi_{X^n,X}(g\otimes \id_X)
    = (\id_X\otimes g)\theta_{X^n, X}
    = \sigma_{X^n,X}(g f_tf_\psi\otimes \rho_X^k).
  \end{equation*}
  This implies \(f_sf_\phi g = g f_tf_\psi\) and therefore \(\phi_{s(i)}\lambda_{sr(i)} = \lambda_{r(i)} \psi_{rt(i)}\) for all \(1\leq i \leq n\).
  Since \(\mathbb Z_2\) is abelian and \(\phi_{s(i)}= \phi_i\) as well as \(\psi_{t(i)}= \psi_{i}\), the claim follows.
\end{proof}

In view of Lemma~\ref{lem: half-braidings_on_C}, we state a slightly refined version of Definition~\ref{def: characteristic_sequence}.

\begin{definition}\label{def: characteristic_sequence_of_half_braiding.}
  Consider an object \(Y =(X^n, \chi_{X^n, X})\in \ZCat C\) whose half-braiding is defined by
  \(\chi_{X^n, X}= \sigma_{X^n,X}(f\otimes \rho_X^j)\) for an integer \(j \in \mathbb Z_2\).
  We call the characteristic sequence \(\phi\) of \(f\) the \emph{signature} of \(Y\).
\end{definition}

We now construct a pivotal structure on the centre of \(\cat C\) which differs from the lifts of \(\id\) and \(\rho\) from \(\cat C\) to \(\ZCat C\).

\begin{theorem}\label{thm: additional_pivotal_structure}
  The Drinfeld centre \(\ZCat C\) of \(\cat C\) admits a pivotal structure \(\zeta\) with
  \begin{subequations}
    \begin{gather}
      \zeta_{(X, \sigma^{\circ, \circ}_{X,\blank})}= \id_X, \qquad \qquad
      \zeta_{(X, \sigma^{\circ, \bullet}_{X,\blank})}= \id_X, \\
      \zeta_{(X, \sigma^{\bullet, \circ}_{X,\blank})}= \rho_X, \qquad \qquad
      \zeta_{(X, \sigma^{\bullet, \bullet}_{X,\blank})}= \rho_X.
    \end{gather}
  \end{subequations}
\end{theorem}
\begin{proof}
  For any object \(Y\in \ZCat C\) we define
  \begin{equation*}
    \zeta_Y = \rho_X^{\phi_1}\otimes\dots\otimes \rho_X^{\phi_n}, \qquad \text{ where \(\phi=(\phi_1, \dots, \phi_n)\) is the signature of \(Y\)}.
  \end{equation*}
  Since the signature \(\varphi\) of a tensor product \(Y\otimes W\) of objects \(Y, W\in \ZCat C\) is given by concatenating the signatures \(\phi\) of \(Y\) and \(\psi\) of \(W\), this defines a family of isomorphisms \(\zeta \from \Id_{\ZCat C} \to \Id_{\ZCat C}\), which is compatible with the monoidal structure.
  It therefore only remains to prove the naturality of \(\zeta\).
  This can be verified by considering all possible lifts of identities and generators of \(\cat C\) to its Drinfeld centre.
  For \(\id_X, \rho_X \from X \to X\) and \(\sigma_{X,X}\from X^2 \to X^2\), this follows by Corollary~\ref{cor: criterion_lifting_automorphisms}.
  To study the coevaluation of \(X\), we fix a half-braiding \(\chi_{X^2, \blank}\from X^2 \otimes \blank \to \blank \otimes X^2\) on \(X^2\).
  Due to Lemma~\ref{lem: half-braidings_on_C}, it is determined by
  \begin{equation*}
    \chi_{X^2, X}= \sigma_{X^2,X}(\sigma_{X,X}^i(\rho_X^j \otimes \rho_X^k)\otimes \rho_X^l), \qquad \text{ where } i,j,k,l \in \mathbb Z_2.
  \end{equation*}
  Now suppose, \(\coev_X \from 1 \to X^2\) lifts to a morphism in \(\ZCat C\), where \(X^2\) is equipped with this half-braiding.
  Relation~\eqref{eq: everything_is_self_dual} together with the self-duality of \(\sigma_{X,X}\) imply \(\sigma_{X,X} \coev_X = \coev_X\) and \(\ev_X \sigma_{X,X} = \ev_X\), which allows us to compute:
  \begin{equation*}
  \tikzsetnextfilename{relation_coev_lift_centre}%
  \input{sections/tikzfigures/relation_coev_lift_centre.tikz}%

  \end{equation*}
  Therefore \(j=k\) and \(\zeta_{(X^2, \chi_{X^2, \blank})} = \id_X^2\) or \(\zeta_{(X^2, \chi_{X^2, \blank})} = \rho_X^2\), from which the desired naturality condition follows.
  A similar argument for the evaluation of \(X\) concludes the proof.
\end{proof}

By Corollary~\ref{cor: piv_on_C}, the Picard heap of \(\ACat C\) can have at most two elements.
However, the above theorem constructs a third pivotal structure on \(\ZCat C\).
This implies our desired result:

\begin{theorem}\label{thm: non_induced_pivotal_structure}
  The pivotal structure \(\zeta\) of \(\ZCat C\) is not induced by the Picard heap of \(\ACat C\).
  In particular, the map \(\iota \from \Pic \ACat C / \Pic \SZCat C \to \Piv \ZCat C\)
  is not surjective.
\end{theorem}

Let us conclude this section by stating that we deem the question interesting under which conditions on a rigid category \(\cat C\), the map \(\iota \from \Pic \ACat C / \Pic \SZCat C \to \Piv \ZCat C\) is surjective.

%%% Local Variables:
%%% mode: latex
%%% TeX-master: "../main"
%%% End:

\section[Bimodule and comodule monads]{Bimonads and comodule monads as coordinate systems for (twisted) centres}\label{sec: monads}
Bimonads and Hopf monads are a vast generalisation of bialgebras and Hopf algebras, respectively.
They naturally arise in the study of (rigid) monoidal categories and topological quantum field theories, see amongst others \cite{Kerler2001, Moerdijk2002, Bruguieres2007, Bruguieres2011, Turaev2017}.
While there are several, sometimes non-equivalent, notions of `Hopf monad', see \cite{Boardman1995,Mesablishvili2011}, we follow the approach of \cite{Bruguieres2007}.

A monadic interpretation of module categories was given by Aguiar and Chase under the name `comodule monad', see \cite{Aguiar2012}.
In this section, we recall some aspects of their theory needed to obtain a monadic version of the results in Section~\ref{sec: pivotality_of_the_df_centre}.

\subsection{Bimonads and monoidal categories}

Due to the lack of a braiding on the endofunctors \(\Endspace(\cat C)\) over \(\cat C\), the naïve notion of bialgebras does not generalise to the monadic setting and needs to be adjusted.
One possible way of overcoming this problem was introduced and studied by Moerdijk under the name `Hopf monads'\footnote{As remarked in \cite{Moerdijk2002}, the concept of Hopf monads is strictly dual to that of monoidal comonads, which are studied for example in \cite{Boardman1995}.} in \cite{Moerdijk2002}; the idea being that the coherence morphisms of an oplax monoidal functor \((T, \Delta, \varepsilon) \from \cat C \to \cat D\), see Definition~\ref{def: mon_functor}, serve as its `comultiplication' and `counit'.
Following the conventions of \cite{Bruguieres2007} we refer to such structures as bimonads.

Modules and their morphisms over a monad \(T\) on \(\cat C\) form the category \(\cat C^T\) of \emph{\(T\)-modules}\footnote{
  In the literature, modules over \(T\) are also referred to as \emph{\(T\)-algebras} and \(\cat C^T\) is called the \emph{\EilenbergMoore category} of \(T\).
  The intention behind our conventions is to have a closer similarity to (Hopf) algebraic notions.
}.
The \emph{free} and \emph{forgetful functor} of \(T\) are
\begin{equation*}
  F^T \from \cat C \to \cat C^T,\quad F^T(M) = ( T(M), \mu_M^{(T)}) \quad \text{ and } \quad
  U^T \from \cat C^T \to \cat C,\quad U^T(M, \vartheta_{M})= M.
\end{equation*}
They constitute the \emph{\EilenbergMoore adjunction} \(\adj{F^T}{U^T}{\cat C}{\cat C^T}\).

Let \(T\) be the monad of the adjunction \stdadj.
In the spirit of our previous remark, we might ask how much the functors \(F\) and \(U\) `differ' from the free and forgetful functors \(F^T \from \cat C \to \cat C^T\) and \( U^T \from \cat C^T \to \cat C\)  of \(T\), respectively.
Roughly summarised we are interested in the following:
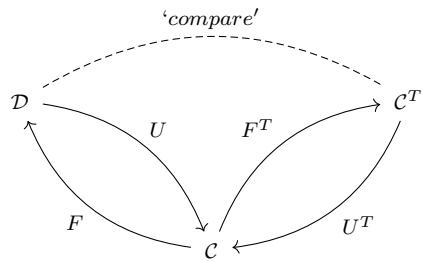
\begin{center}
  \begin{tikzcd}
    \cat D \arrow[rrdd, "U", bend left] &  &                                                                        &  & \cat C^T \arrow[llll, "`compare'"', no head, dashed, bend right] \arrow[lldd, "U^T", bend left] \\
    &  &                                                                        &  &                                                                                          \\
    &  & \cat C \arrow[lluu, "F", bend left] \arrow[rruu, "F^T", bend left] &  &
  \end{tikzcd}
\end{center}

\begin{definition}\label{def: comparison_functor}
  Let \(T \defeq UF\) be the monad of the adjunction \stdadj.
  The \emph{comparison functor} is the unique functor \(\Sigma \from \cat D \to \cat C^T\) satisfying
  \begin{equation}\label{eq: conditions_comparison_functor}
    \Sigma F = F^T \qquad \qquad \text{ and } \qquad \qquad
    U^T \Sigma = U.
  \end{equation}
  On objects it is given by
  \begin{equation}\label{eq: comparison_functor_on_objects}
    \Sigma(X) = (U(X), U(\epsilon_{X})), \qquad \qquad \text{ for all } X \in \cat D.
  \end{equation}
\end{definition}

We call an adjunction \emph{monadic} if its comparison functor is an equivalence.

\begin{definition}\label{def: bimonad}
  A \emph{bimonad} on a monoidal category \(\cat C\) is an oplax monoidal endofunctor \((B, \Delta, \varepsilon) \from \cat C \to \cat C\) together with oplax monoidal natural transformations \(\mu \from B^2 \to B\) and \(\eta: \Id_{\cat C} \to B\) implementing a monad structure on \(B\).

  A \emph{morphism of bimonads} is a natural transformation \(f \from B \to H\) between bimonads which is oplax monoidal as well as a morphism of monads.
\end{definition}

\begin{remark}\label{rem: bimonads_lead_to_monoidal category}
  Despite this terminology not being standard, it can be justified by representation theoretic considerations.
  Under \emph{Tannaka--Krein reconstruction}, see \cite[Chapter 5]{Etingof2015}, the comultiplication and counit of a bialgebra correspond to a tensor product and unit on its category of modules.
  Similarly, given a bimonad \((B, \mu ,\eta, \Delta, \varepsilon)\from \cat C \to \cat C\) and two  modules \((M, \vartheta_{M}), (N, \vartheta_{N}) \in \cat C^B\)  we set
  \begin{equation}\label{eq: tensor_product_of_modules_over_bimonad}
    \big(M, \vartheta_{M}\big) \otimes \big(N, \vartheta_N\big) \defeq \big(M \otimes N, (\vartheta_{M}\otimes \vartheta_N) \Delta_{M, N} \big).
  \end{equation}
  Moreover, we define \(\vartheta_{1} \from B(1) \to 1\).
  The coassociativity and counitality of the comultiplication of \(B\) imply that the above construction implements a monoidal structure on \(\cat C^B\), parallel to that on the modules over a bialgebra.
\end{remark}

Going further, we can incorporate rigidity into this picture.
In view of \cite[Theorem~3.8]{Bruguieres2007}, we state:

\begin{definition}\label{def: Hopf_monad}
  A bimonad \(H\from \cat C \to \cat C\) on a rigid category \(\cat C\) is called a \emph{Hopf monad} if its category of modules \(\cat C^H\) is rigid.
\end{definition}

\begin{remark}\label{rem: left_and_right_antipode}
  The rigidity of the modules \(\cat C^{H}\) of a Hopf monad \(H\from \cat C \to \cat C\) is reflected by the existence of two natural transformations
  \begin{equation}\label{eq: left_and_right_antipode}
    s^{l}_{X}\from H(\lDual{H(X)})\to \lDual H, \qquad
    s^{r}_{X}\from H(\rDual{H(X)})\to \rDual H, \qquad \text{ for all }X \in \cat C,
  \end{equation}
  called the \emph{left} and \emph{right antipode} of \(H\).
  In Example~2.4 of \cite{Bruguieres2012} it is explained how these generalise the antipode of a Hopf algebra.
\end{remark}

The intricate interplay between monads and adjunctions transcends to monoidal categories and bimonads.
Suppose \stdadj to be an oplax monoidal adjunction between \(\cat C\) and \(\cat D\).
The monad of the adjunction \(UF \from \cat C \to \cat C\) is a bimonad whose comultiplication is defined for every \(X,Y \in \cat C\) as the composition
\begin{equation}\label{eq: comult_for_bimonad_from_adjunction}
  UF(X \otimes Y)
  \trightarrow{U\big(\Delta^{(F)}_{X,Y}\big)}
  U(F(X)\otimes F(Y))
  \trightarrow{\Delta^{(U)}_{F(X),F(Y)}}
  UF(X)\otimes UF(Y).
\end{equation}
Its counit is
\begin{equation}\label{eq: counit_for_bimonad_from_adjunction}
  UF(1)
  \trightarrow{U \big( \varepsilon^{(F)}\big)}
  U(1)
  \trightarrow{\varepsilon^{(U)}}
  1.
\end{equation}

The next result is a slightly simplified version of \cite[Lemma~7.10]{Turaev2017}.

\begin{lemma}\label{lem: monoidal_adjunctions_vs_bimonads}
  Let \stdadj be a pair of adjoint functors between two monoidal categories.
  The adjunction \(F \adjoint U\) is monoidal if and only if \(U\) is a strong monoidal functor.
  That is, the coherence morphisms of \(U\) are invertible.
\end{lemma}

Let \(B \from \cat C \to \cat C\) be the bimonad arising from the monoidal adjunction \stdadj.
Since the forgetful functor \(U^B \from \cat C^B \to \cat C\) is strict monoidal, the adjunction \(F^B \adjoint U^B\) is monoidal by the above lemma.
This raises the question whether the comparison functor, mediating between the two adjunctions, is compatible with this additional structure.
Due to \cite{kelly74}, see also \cite[Theorem~2.6]{Bruguieres2007}, we have the following result.

\begin{lemma}\label{lem: comparison_functor_bimonad_monoidal}
  Let \stdadj be a monoidal adjunction and write \(B \from \cat C \to \cat C\) for its induced bimonad.
  The comparison functor \(\Sigma \from \cat D \to \cat C^B\) is strong monoidal and
  \(U^B \Sigma = U\) as well as \(\Sigma F = F^B\) as strong, respectively, oplax monoidal functors.
\end{lemma}

The question to which extend the monoidal structure on \(\cat C^B\) is unique
was answered by Moerdijk \cite[Theorem~7.1]{Moerdijk2002} and McCrudden \cite[Corollary~3.13]{McCrudden2002}.

\begin{proposition}\label{prop: Moerdijk_reconstruction}
  Let \((B, \mu, \eta)\) be a monad on a monoidal category \(\cat C\).
  There exists a one-to-one correspondence between bimonad structures on \(B\) and monoidal structures on \(\cat C^B\) such that the forgetful functor \(U^B\) is strict monoidal.
\end{proposition}

\subsection{Comodule monads}\label{subsec: comodule_monads}

Monads with a `coaction' over a bimonad were defined and studied by Aguiar and Chase in \cite{Aguiar2012}.
This concept is needed to obtain an adequate monadic interpretation of twisted centres.
We briefly summarise the aspects of the aforementioned article that are needed for our investigation\footnote{
  We slightly deviate from \cite{Aguiar2012} in that we study right comodule monads as opposed to their left versions.
}.
To keep our notation concise, we fix two monoidal categories \(\cat C\) and \(\cat D\) and over each a right module category \(\cat M\) and \(\cat N\).
\begin{definition} \label{def: comodule_functor}
  Suppose \((F, \Delta, \varepsilon) \from \cat C \to \cat D\) to be an oplax monoidal functor.
  A \emph{(right) comodule functor over \(F\)} is a pair \((G, \delta)\) consisting of a  functor \(G \from \cat M \to \cat N\) together with a natural transformation
  \begin{equation}\label{eq: coaction_of_comodule_functor}
    \delta_{M,X} \from G(M \ract X) \to G(M) \ract F(X),\qquad  \text{ for all } X\in \cat C \text{ and } M \in \cat M,
  \end{equation}
  called the \emph{coaction} of \(G\), which is \emph{coassociative} and \emph{counital}; see \cite[Definition~3.5]{Aguiar2012}.

  A comodule functor is called \emph{strong} if its coaction is an isomorphism.
\end{definition}

A recurring example of strong comodule functors in our investigation is given by forgetful functors.
By construction \(U^{(Z)} \from \ZCat C \to \cat C\) is strict monoidal.
Over it, the forgetful functor \(U^{(L)} \from \ZCat[L]{C} \to \cat C\) from a left twisted centre to its base category is strict comodule.

\begin{definition}\label{def: morphism_of_comodules}
  Let \(G, K \from \cat M \to \cat N\) be comodule functors over  \(B, F \from \cat C \to \cat D\).
  A \emph{comodule natural transformation} from \(G\) to \(K\) is a pair of natural transformations \(\phi \from G\to K\) and \(\psi \from B \to F\) such that
  \begin{equation}\label{eq: comodule_natural_transformation}
    (\phi_M \ract \psi_X)\delta^{(G)}_{M,X}
    = \delta^{(K)}_{M, X} \phi_{M \ract X},
    \qquad \text{for all }X \in \cat C \text{ and }M \in \cat M.
  \end{equation}
  We call \((\phi, \psi)\) a \emph{morphism of comodule functors} if \(B=F\) and \(\psi=\id_B\).
\end{definition}

Suppose the pair \(\phi \from G\to K\) and \(\psi \from B \to F\) to constitute a comodule natural transformation.
We can view \(\phi \from G \to K\) as a morphism of comodule functors over \(F\) if we equip \(G\) with a new coaction.
It is given for all \(X \in \cat C\) and \(M \in \cat M\) by
\begin{equation*}
  G(M \ract X) \trightarrow{\delta^{(G)}_{M,X}} G(M) \ract B(X) \trightarrow{\id_{G(M)}\ract\psi_X} G(M) \ract F(X).
\end{equation*}
It follows that by altering the involved coactions suitably, comodule natural transformations and morphisms of comodule functors can be identified with each other.

\begin{remark}\label{rem: comodules_over_bimonad_mon_cat}
  Let \((B, \mu, \eta, \Delta, \varepsilon) \from \cat C \to \cat C\) be a bimonad and \(\cat M\) a module category over \(\cat C\).
  The unit \(\eta \from \Id_{\cat C} \to B\) implements a coaction on \(\Id_{\cat M} \from \cat M \to \cat M\) via
  \begin{equation}
    \id_M \ract \eta_X \from \Id_{\cat M}(M \ract X) \to \Id_{\cat M}(M) \ract B(X), \qquad \text{ for all } X \in \cat C, M \in \cat M.
  \end{equation}
  Using the multiplication \(\mu \from B^2 \to B\), we can equip the composition \(GK\) of two comodule functors \(G, K \from \cat M \to \cat M\) with a comodule structure:
  \begin{equation}
    \delta^{(GK)}\defeq (\id \ract\mu)\delta^{(G)} G(\delta^{(K)})\from GK(\blank \ract \blank) \to GK(\blank) \ract B(\blank).
  \end{equation}
  Due to the associativity and unitality of the multiplication of \(B\), the category \(\Com(B, \cat M)\) of comodule endofunctors on \(\cat M\) over \(B\) is monoidal.
\end{remark}

\begin{definition}\label{def: comodule_monad}
  Consider a bimonad \(B \from \cat C \to \cat C\) and a module category \(\cat M\) over \(\cat C\).
  A \emph{comodule monad} over \(B\) on \(\cat M\) is a comodule endofunctor \((K, \delta) \from \cat M \to \cat M\) together with morphisms of comodule functors  \(\mu \from K^2 \to K\)  and \(\eta \from \Id_{\cat M} \to K\) such that \((K,\mu,\eta)\) is a monad.

  A \emph{morphism of comodule monads} is a natural transformation of comodule functors \(f \from K \to L\) that is also a morphism of monads.
\end{definition}

\begin{remark}\label{rem: comodule_monads_induce_module_categories}
  Let \(B \from \cat C \to \cat C\) be a bimonad and  \((K, \delta) \from \cat M \to \cat M\) a comodule monad over it.
  The coaction of \(K\) allows us to define an action \(\ract \from \cat M^{K}\times \cat C^B \to \cat M^K\).
  For any two modules \((M, \vartheta_{M}) \in \cat M^K\) and \((X, \vartheta_{X}) \in \cat C^B\), it is given by
  \begin{equation}\label{eq: action_from_coaction}
    \big(M, \vartheta_{M}\big) \ract \big(X, \vartheta_{X}\big) \defeq \big(M \ract X, (\vartheta_{M}\ract \vartheta_{X}) \delta_{M,X} \big).
  \end{equation}
  The axioms of the coaction of \(B\) on \(K\) translate precisely to the compatibility of the action of \(\cat C^B\) on \(\cat M^{K}\) with the tensor product and unit of \(\cat C^B\).
\end{remark}

\subsection{Cross products and distributive laws}\label{subsec: cross_products_and_distributive_laws}

Suppose \(\cat C\) to be the modules of a Hopf monad \(H \from \cat V \to \cat V\).
The Hopf monadic description of the Drinfeld centre \(\ZCat C\) of \(\cat C\) due to \Bruguieres and Virelizier, given in \cite{Bruguieres2012}, is achieved as a two-step process.
First, by finding a suitable monad on \(\cat C\) and then `extending' it to a monad on \(\cat V\).
We will review this `extension' process based on Sections 3 and 4 of \cite{Bruguieres2012}.

\begin{definition}\label{def: crossed_product_monad}
  Let \(H\from \cat V \to \cat V\) be a monad and \(T \from \cat V^H \to \cat V^H\) a functor.
  The \emph{cross product} \(T\rtimes H\) of \(T\) by \(H\) is the endofunctor
  \(U^H T F^H \from \cat V \to \cat V\).

  If \(T\) is a monad, \(T \rtimes H\) inherits this structure:
  multiplication and unit are given by
  \begin{equation}\label{eq: mult_and_unit_of_cross_product}
  \mu^{T \rtimes H} \defeq U^H(\mu^{(T)}_{F^H}) U^HT(\epsilon^{(F \adjoint U)}_{TF^H}) \qquad\text{and}\qquad
  \eta^{T \rtimes H} \defeq U^H(\eta^{(T)}_{F^H})\eta^{(F \adjoint U)}.
  \end{equation}
\end{definition}

The cross product \(B \rtimes H \from \cat C \to \cat C\) of two bimonads \(H \from \cat V \to \cat V\) and \(B \from \cat V^H \to \cat V^H\) is a bimonad again, with comultiplication
\[
  \Delta^{(U^H)}_{TF^H(\blank), TF^H(\blank)} U^H (\Delta^{(T)}_{F^H(\blank), F^H(\blank)}) U^H T (\Delta^{(F^H)}_{\blank, \blank}),
\]
and counit
\[
  \varepsilon^{(U^H)} U^H (\varepsilon^{(T)}) U^H T (\varepsilon^{(F^H)}).
\]

Similar considerations imply the following:

\begin{lemma}\label{lem: cross_product_of_bi/comonads_is_bi/comonad}
  Let \(H \from \cat V \to \cat V\) and \(B \from \cat V^H \to \cat V^H\) be bimonads which respectively coact on the comodule monads \(K \from \cat M \to \cat M\) and \(C \from \cat M^K \to \cat M^K\).
  The cross product \(C \rtimes K \from \cat M \to \cat M\) is a comodule monad over \(B \rtimes H\) via the coaction
  \begin{equation}\label{eq: coaction_on_cross_product}
    \delta^{(U^K)}_{C F^K(\blank), B F^H(\blank)} U^K (\delta^{(C)}_{F^K(\blank), F^H(\blank)}) U^K C (\delta^{(F^K)}_{\blank, \blank}).
  \end{equation}
\end{lemma}

Assume we have a monad \(B \from \cat V^H \to \cat V^H\) `on top' of another monad \(H \from \cat V \to \cat V\).
The question under which conditions the modules \(\cat V^{B\rtimes H}\) of \(B\rtimes H\) are isomorphic to \((\cat V^H)^B\) is closely related to Beck's theory of \emph{distributive laws}, developed in \cite{Beck1969}.

Street developed the theory of monads and distributive laws intrinsic to `well-behaved' 2-categories in \cite{Street1972}.
If we apply his findings to the 2-category \(\otimes\text{-}\mathsf{Cat}\) of monoidal categories, oplax monoidal functors and oplax monoidal natural transformations, we obtain a description of bimonads and \emph{oplax monoidal distributive laws}, see also \cite{McCrudden2002}.
That is, oplax monoidal natural transformations \(\Lambda \from HB \to BH\) between bimonads \(H, B \from \cat V \to \cat V\) that are moreover distributive laws.
Accordingly, suppose \(\Lambda \from HB \to BH\) to be an oplax monoidal distributive law.
The comultiplication and counit of the underlying functor \(BH \from \cat V \to \cat V\) turn \(B \circ_{\Lambda} H\) into a bimonad.

Comodule monads, on the other hand, can be intrinsically described in the 2-category \((\!\ract\text{-}\mathsf{Cat}, \otimes\text{-}\mathsf{Cat})\) which has
\begin{thmlist}
  \item as objects pairs \((\cat M, \cat V)\) comprising a right module category \(\cat M\) over a monoidal category \(\cat V\),
  \item as \(1\)-morphisms pairs \((G, F)\) of a comodule functor \(G\) over an oplax monoidal functor \(F\) and
  \item as \(2\)-morphisms pairs \((\phi, \psi)\) which constitute a  comodule natural transformation.
\end{thmlist}

The subsequent results arise immediately from \cite{Street1972}.

\begin{definition}\label{def: comodule_distributive_law}
  Let \(K, C \from \cat M \to \cat M\) be two comodule monads over the bimonads \(H,B \from \cat V \to \cat V\), respectively.
  A \emph{comodule distributive law} is a pair of distributive laws \(\Omega \from KC \to CK\) and \(\Lambda \from HB \to BH\)  such that \((\Lambda, \Omega)\) is a comodule natural transformation.
\end{definition}

\begin{proposition}\label{prop: distributive_laws_are_lifts_of_monads}
  Consider two comodule monads \(K, C \from \cat M \to \cat M\) over the bimonads \(H,B \from \cat V \to \cat V\).
  There exists a bijective correspondence between:
  \begin{thmlist}
    \item comodule distributive laws \((KC \trightarrow{\Omega} CK, HB \trightarrow{\Lambda}BH)\) and
    \item lifts of \(B\) to a bimonad \(\tilde B \from \cat V^H \to \cat V^H\) together with lifts of \(C\) to a comodule monad \(\tilde C \from \cat M^K \to \cat M^K\) over \(\tilde B\) such that \(B U^H = U^H\tilde B\) as oplax monoidal functors and \(C U^K = U^K \tilde C\) as comodule functors.
  \end{thmlist}
\end{proposition}

Let \((KC \trightarrow{\;\Omega\;} CK, HB \trightarrow{\;\Lambda\;} BH)\) be a comodule distributive law.
The coactions of \(K\) and \(C\) turn \(C \circ_{\Omega} K\) into a comodule monad over \(B \circ_{\Lambda} H\).

\begin{lemma}\label{lem: distributive_laws_and_comparison_functor}
  Suppose \(\Omega \from KC \to CK\) and \(\Lambda \from HB \to BH\) to form a comodule distributive law, then
  \begin{thmlist}
    \item \((\cat V^H)^{\tilde B^{\Lambda}}\) is isomorphic as a monoidal category to
    \(\cat V^{B \circ_{\Lambda} H}\) and
    \item \((\cat M^K)^{\tilde C^{\Omega}}\) is isomorphic as a module category over \(\cat V^{B \circ_{\Lambda} H}\) to \(\cat M^{C \circ_{\Omega} K}\).
  \end{thmlist}
\end{lemma}

If \(B, H \from \cat V \to \cat V\) are Hopf monads,
\cite{Bruguieres2012} shows that if \(\Lambda \from HB \to BH\) is a monoidal distributive law,
\(B \circ_{\Lambda} H\from \cat V \to \cat V\) and the lift \(\tilde B^{\Lambda} \from \cat V^H \to \cat V^H\) are Hopf monads, as well.

%%% Local Variables:
%%% mode: latex
%%% TeX-master: "../main"
%%% End:

\section{A monadic perspective on twisted centres}\label{sec: mon_perspective}
The anti-\YetterDrinfeld modules of a finite-dimensional Hopf algebra are a module category over the \YetterDrinfeld modules.
Subsequently, they are implemented by a comodule algebra over the Drinfeld double, see \cite{Hajac2004a}.
As explained in Section~\ref{sec: pivotality_of_the_df_centre}, we find ourselves in a similar situation.
Our replacement of the anti-\YetterDrinfeld modules, the anti-Drinfeld centre,
is a module category over the Drinfeld centre.

We replace finite-dimensional vector spaces by a rigid, possibly pivotal, category \(\cat V\) and the underlying Hopf algebra with a Hopf monad  \(H \from \cat V \to \cat V\).
In this section we study a Hopf monad \(D(H) \from \cat V \to \cat V\) and over it a comodule monad \(Q(H)\from \cat V \to \cat V\), which realise the centre and its twisted cousin as their respective modules.
\Bruguieres and Virelizier gave a transparent description of \(D(H)\) in \cite{Bruguieres2012} by extending results of Day and Street, see \cite{Day2007}.
The key concept in its construction is the so-called centraliser of the identity functor of \(\cat V^H\).
It is used to define a Hopf monad \(\mathfrak D(\cat V^H)\) on \(\cat V^H\) with \(\ZCat{\cat V^H}\) as its \EilenbergMoore category.
From this, one obtains—as an application of Beck's theory of distributive laws—the Drinfeld double \(D(H) \from \cat V \to \cat V\).
We apply the same techniques to define the anti-double \(Q(H)\) of \(H\), whose modules are isomorphic to the `dual' of the anti-Drinfeld centre \(\QCat{\cat V^H}\).
This approach is best summarised by the following diagram:
\begin{center}
  \begin{tikzcd}[column sep={1cm,between origins}, row sep={0.875cm,between origins}]
    & &
    \ZCat*{\cat V^H}
    \arrow[rrddd, "U^{(Z)}" description, bend left]
    \arrow[lldddd, leftrightarrow, "\Sigma^{(\mathfrak D)}" description, bend left=15]
    & & & &
    \QCat{\cat V^H}
    \arrow[llddd, "U^{(Q)}" description, bend left]
    \arrow[llll, "action" description, no head, dotted]
    \arrow[rrdddd, leftrightarrow, "\Sigma^{(\mathfrak Q)}" description, bend right=15]
    & & \\
    & & & & & & & & \\
    & & & & & & & & \\
    & & & &
    {\cat V^H}
    \arrow[rruuu, "F^{(Q)}" description, bend left]
    \arrow[lluuu, "F^{(Z)}" description, bend left]
    \arrow[lllld, "F^{\mathfrak D}" description, bend left=20]
    \arrow[rrrrd, "F^{\mathfrak Q}" description, bend left=20]
    \arrow[ddd, "U^H" description, bend left]
    & & & & \\
    {{\cat V^H}}^{\mathfrak D}
    \arrow[rrrru, "U^{\mathfrak D}" description, bend left=20]
    \arrow[rrdddd, leftrightarrow, "\Sigma^{(D(H))}" description, bend left=15]
    & & & & & & & &
    {{\cat V^H}}^{\mathfrak Q}
    \arrow[llllu, "U^{\mathfrak Q}" description, bend left=20]
    \arrow[lldddd, leftrightarrow, "\Sigma^{(Q(H))}" description, bend right=15]
    \\
    & & & & & & & & \\
    & & & &
    \cat V
    \arrow[uuu, "F^H" description, bend left]
    \arrow[rrdd, "F^{Q(H)}" description, bend left]
    \arrow[lldd, "F^{D(H)}" description, bend left]
    & & & & \\
    & & & & & & & & \\
    & &
    \cat V^{D(H)}
    \arrow[rruu, "U^{D(H)}" description, bend left]
    & & & &
    \cat V^{Q(H)}
    \arrow[lluu, "U^{Q(H)}" description, bend left]
    \arrow[llll, "action" description, no head, dotted]
    & &
  \end{tikzcd}
  
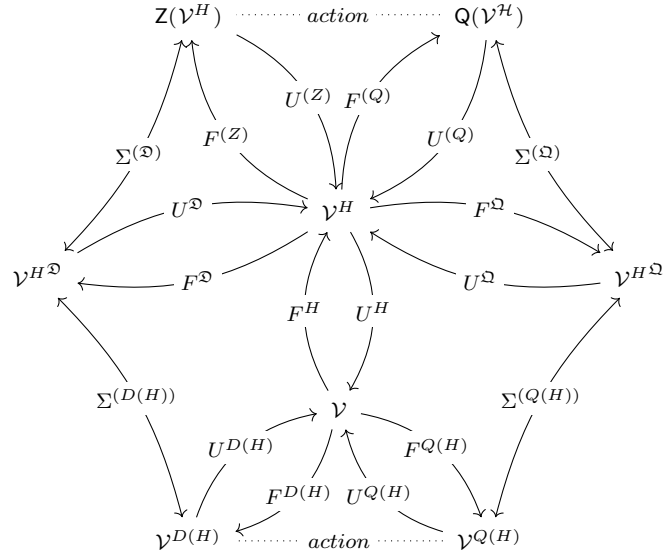
\captionof{figure}{A cobweb of adjunctions, monads and various versions of the Drinfeld and anti-Drinfeld centre.}
  \label{fig: twisted_centres_and_monads}
\end{center}
The translation of module functors between \(\ZCat*{\cat V^H}\) and \(\QCat{\cat V^H}\) into morphisms of comodule monads between \(Q(H)\) and \(D(H)\) yields our desired monadic version of Theorem~\ref{thm: equiv_in_Hopf_setting}, which we prove in Theorem~\ref{thm: equivalence_in_the_monad_setting}.
We end our endeavour into the theory of comodule monads with Corollary~\ref{cor: pivotal_from_central_anti_central}.
In it, we explain how pivotal structures on \(\cat V^H\) arise from module morphisms between the so-called central Hopf monad \(\mathfrak D\) and the anti-central comodule monad \(\mathfrak Q\).

\subsection{Centralisable functors and the central bimonad} \label{section: centralisable_functors_and_the_central_bimonad}

The construction of the double of a Hopf monad \(H \from \cat V \to \cat V\) given in \cite{Bruguieres2012} relies heavily on an `accessible' left dual of the forgetful functor \(U^{(Z)} \from \ZCat*{\cat V^H} \to \cat V^H\).
It is obtained as an application of \emph{coend calculus}; see for example \cite{loregian2021}.

\begin{definition}\label{def: centraliser}
  Suppose \(\cat C\) to be a rigid category and \(T \from \cat C \to \cat C\) to be an endofunctor.
  We call \(T\) \emph{centralisable} if there exists a universal extranatural transformation
  \begin{equation*}
    \zeta_{Y,X} \from \lDual{T(Y)} \otimes X \otimes Y \to Z_T(X),
    \qquad \text{ for } X, Y \in \cat C.
  \end{equation*}
\end{definition}

A centralisable functor \(T \from \cat C \to \cat C\) admits a \emph{universal coaction}
\begin{equation}\label{eq: universal_coaction}
  \chi_{X,Y} \defeq (\id_{T(Y)} \otimes \zeta_{Y,X})(\coev^l_{T(Y)} \otimes\, \id_{X \otimes Y}),
  \qquad \text{ for } X, Y \in \cat C,
\end{equation}
which is natural in both variables.
We call the pair \((Z_T, \chi)\) a \emph{centraliser} of \(T\).

Graphically, we represent the universal coaction as
\begin{equation}\label{eq: universal_coaction_string_diag}
  \tikzsetnextfilename{universal_coaction}%
  \input{sections/tikzfigures/universal_coaction.tikz}%

\end{equation}
It being natural equates to
\begin{equation}\label{eq: universal_coaction_natural}
  \tikzsetnextfilename{universal_coaction_natural}%
  \input{sections/tikzfigures/universal_coaction_natural.tikz}%

\end{equation}

The \emph{extended factorisation property} of universal coactions provides us with a potent tool for constructing bi- and comodule monads.
Its proof is given for example in \cite[Lemma~5.4]{Bruguieres2012}.
\begin{lemma} \label{lem: extended_factorisation_property}
  Let \((Z_T, \chi)\) be the centraliser of a functor \(T \from \cat C \to \cat C\) and suppose that \(L, R \from \cat D \to \cat C\) are two functors.
  For any \(n \in \mathbb N\) and any natural transformation
  \begin{equation*}
    \phi_{X, Y_1, \dots, Y_n} \from L(X) \otimes  Y_1 \otimes \dots \otimes Y_n \to T(Y_1) \otimes \dots \otimes T(Y_n) \otimes R(X),
  \end{equation*}
  where \(X \in \cat D\) and \(Y_1, \dots, Y_n \in \cat C\), there exists a unique natural transformation
  \begin{equation*}
    \nu_V \from Z_T^n L(X) \to R(X) , \qquad \text{ for } V \in \cat D,
  \end{equation*}
  which satisfies
  \begin{equation}
  \tikzsetnextfilename{property_of_univ_coaction}%
  \input{sections/tikzfigures/property_of_univ_coaction.tikz}%

  \end{equation}
\end{lemma}

Suppose \((T, \Delta^{(T)}, \varepsilon^{(T)}) \from \cat C \to \cat C\) to be an oplax monoidal functor with centraliser \((Z_T, \chi)\).
For all \(X \in \cat C\), the counit of \(T\) combined with the universal coaction of \(Z_T\) gives rise to a natural transformation
\begin{equation}\label{eq: unit_of_centraliser}
  \tikzsetnextfilename{unit_T_central_monad}%
  \input{sections/tikzfigures/unit_T_central_monad.tikz}%

\end{equation}

We derive another natural transformation \(\mu^{(Z_T)} \from Z_T^2 \to Z_T\) from the comultiplication of \(T\).
Due to Lemma~\ref{lem: extended_factorisation_property} it is uniquely defined by
\begin{equation}\label{eq: mult_of_centraliser}
  \tikzsetnextfilename{mult_T_central_monad}%
  \input{sections/tikzfigures/mult_T_central_monad.tikz}%

\end{equation}
\begin{lemma}\label{lem: oplax_centralisable_leads to_monad}
  The centraliser \((Z_T, \chi)\)  of an oplax monoidal endofunctor \(T \from \cat C \to \cat C\) is  a monad with multiplication and unit as given in Equations \eqref{eq: mult_of_centraliser} and \eqref{eq: unit_of_centraliser}.
\end{lemma}

The above lemma is proven as the first part of \cite[Theorem~5.6]{Bruguieres2012}.
In it, the authors further consider \(T \from \cat C \to \cat C\) to be equipped with a Hopf monad structure and show that in this case \(Z_T\) is a Hopf monad as well.
The extended factorisation property given in Lemma~\ref{lem: extended_factorisation_property} allows us to reconstruct a comultiplication on \(Z_T\) from a twofold application of the universal coaction and the multiplication of \(T\):
\begin{equation}\label{eq: comult_of_centraliser}
  \tikzsetnextfilename{comult_central_bimonad}%
  \input{sections/tikzfigures/comult_central_bimonad.tikz}%

\end{equation}
Likewise, the unit of \(T\) induces a counit on \(Z_T\) via
\begin{equation}\label{eq: counit_of_centraliser}
  \tikzsetnextfilename{counit_central_bimonad}%
  \input{sections/tikzfigures/counit_central_bimonad.tikz}%

\end{equation}
A direct computation now verifies that the centraliser \(Z_T\) is a bimonad as well.
For the construction of left and right antipodes, see \cite[Theorem~5.6]{Bruguieres2012}.

\begin{remark}\label{rem: twisted_centre_of_Hopf_monad_is_rigid}
  We think of \(\ZCat[H]{C}\) as the centre of an oplax bimodule category as stated in Remark~\ref{rem: other_variants_of_twisted_centres}, see also \cite[Section 5.5]{Bruguieres2007}.
  Objects in \(\ZCat[H]{C}\) are pairs \((X, \sigma_{X, \blank})\), where \(X \in \cat C\) and
  \begin{equation*}
    \sigma_{X,Y} \from X \otimes Y \to H(Y)\otimes X, \qquad \qquad \text{ for all } Y \in \cat C
  \end{equation*}
  is a natural transformation satisfying for all \(X,Y,W \in \cat C\)
  \begin{gather}
    (\Delta^{(H)}_{Y,W} \otimes \id_X) \sigma_{X, Y \otimes W} = (\id_{H(Y)} \otimes \sigma_{X,W})(\sigma_{X,Y} \otimes \id_{W}) \\
    (\varepsilon^{(H)} \otimes \id_{X}) \sigma_{X,1} = \id_{X}.
  \end{gather}
  Analogous to the centres studied before, the morphisms in \(\ZCat[H]{C}\) are those morphisms of \(\cat C\) which commute with the respective half-braidings.
  As shown in \cite[Proposition~5.9]{Bruguieres2012}, the structure morphisms of a Hopf monad \(H \from \cat C \to \cat C\) can be used to define a rigid structure on \(\ZCat[H]{C}\).
  For example, the tensor product of two modules \((X, \sigma_{X,\blank}), (Y, \sigma_{Y,\blank}) \in \ZCat[H]{C}\) is \(X\otimes Y \in \cat C\) together with the half-braiding
  \begin{equation}
  \tikzsetnextfilename{tensor_product_half_braidings_Hopf_monad}%
  \input{sections/tikzfigures/tensor_product_half_braidings_Hopf_monad.tikz}%

  \end{equation}
\end{remark}

Since centralisers of Hopf monads are Hopf monads themselves, it stands to reason that their modules implement the twisted centres discussed in the previous remark as a rigid category.
This is proven in \cite[Theorem~5.12 and Corollary~5.14]{Bruguieres2012}.

\begin{proposition}\label{prop: central_bimonad_Drinfeld_centre}
  Suppose \(H \from \cat C \to \cat C\) to be a centralisable Hopf monad.
  The modules \(\cat C^{Z_H}\) of its centraliser \((Z_H, \chi)\) are isomorphic as a rigid category to \(\ZCat[H]{C}\).
\end{proposition}

Applying the above proposition to the identity functor \(\Id \from \cat C \to \cat C\), we obtain a Hopf monadic description of the Drinfeld centre \(\ZCat C\) of a rigid category \(\cat C\).
The terminology of our next definition is due to Shimizu, see \cite{Shimizu2017}.

\begin{definition}\label{def: central_bimonad}
  Let \(\Id \from \cat C \to \cat C\) be centralisable with centraliser \((Z, \chi)\).
  We call  \(\mathfrak D(\cat C) \defeq (Z, \mu^{(Z)}, \eta^{(Z)}, \Delta^{(Z)}, \varepsilon^{(Z)}) \from \cat C \to \cat C\) the \emph{central Hopf monad} of \(\cat C\) and denote the category of its modules by \(\cat C^{\mathfrak D}\).
\end{definition}

An important step in proving Proposition~\ref{prop: central_bimonad_Drinfeld_centre} is determining an inverse to the comparison functor \(\Sigma^{(Z_T)} \from \ZCat[T]{C} \to \cat C^{Z_T}\).
This construction will also play a substantial role in our monadic description of the anti-Drinfeld centre, hence why we recall it in its full generality.
Let \(T \from \cat C \to \cat C\) be a centralisable oplax monoidal endofunctor with \((Z_T, \chi)\) as its centraliser.
To every module \((M, \vartheta_M)\) over \(Z_T\) we associate a half-braiding \(\sigma_{M, \blank} \from M \otimes \blank \to T(\blank) \otimes M\).
For any \(X \in \cat C\) it is given by the composition
\begin{equation}\label{eq: free_half_braiding}
  \tikzsetnextfilename{braiding_from_universal_coaction}%
  \input{sections/tikzfigures/braiding_from_universal_coaction.tikz}%

\end{equation}
This yields a functor \(E^{(Z_T)} \from \cat C^{Z_T} \to \ZCat[T]{C}\) which is the identity on morphisms and on objects is given by
\begin{equation}\label{eq: inverse_comparison_functor}
  E^{(Z_T)}(M, \vartheta_{M}) = (M, \sigma_{M,\blank}), \qquad \text{ for all }(M, \vartheta_{M}) \in \cat C^{Z_T}.
\end{equation}

Conversely, we assign to every object \((M, \sigma_{M, \blank}) \in \ZCat[T]{C}\) a module over \(Z_T\) whose action \(\vartheta_M\) is uniquely defined by
\begin{equation}
  \tikzsetnextfilename{action_from_braiding}%
  \input{sections/tikzfigures/action_from_braiding.tikz}%

\end{equation}
As it turns out, this yields the comparison functor \(\Sigma^{(Z_T)} \from \ZCat[T]{C} \to \cat C^{Z_T}\).

\begin{remark}\label{rem: adjoint_from_inverse_of_the_comparison_functor}
  Suppose \(T \from \cat C \to \cat C\) to be a centralisable oplax monoidal endofunctor with  \((Z_T, \chi)\) as its centraliser.
  Denote the free functor of the \EilenbergMoore adjunction of  \(Z_T\) by \(F^{Z_T} \from \cat C \to \cat C^{Z_T}\).
  The composition
  \begin{equation}
    \cat C \trightarrow{F^{Z_T}} \cat C^{Z_T} \trightarrow{E^{(Z_T)}}  \ZCat[T]{C}
  \end{equation}
  defines a left adjoint of the forgetful functor \(U^{(T)} \from \ZCat[T]{C} \to \cat C\).
\end{remark}

We recall \cite[Theorem~5.12]{Bruguieres2012}, which proves the adjunction \(F^{(T)} \adjoint U^{(T)}\)  to be monadic.

\begin{proposition}\label{prop: inverse_comparison_functor}
  Assume \((Z_T, \chi)\) to be a centraliser of the oplax monoidal endofunctor  \(T \from \cat C \to \cat C\).
  The functor \(E^{(Z_T)} \from \cat C^{Z_T} \to \ZCat[T]{C}\) is an isomorphism of categories whose inverse is the comparison functor \(\Sigma^{(Z_T)} \from \ZCat[T]{C} \to \cat C^{Z_T}\).
\end{proposition}

\subsection{Centralisers and comodule monads}

We will now apply the methods of \Bruguieres and Virelizier to twisted centres for the purpose of obtaining a comodule monad that implements the anti-Drinfeld centre.
Hereto, we need a generalised version of the concept of modules over a monad.
Our approach is based on \cite{Mesablishvili2011}.

\begin{definition}\label{def: oplax_right_action}
  Suppose \((B, \mu, \eta) \from \cat{C} \to \cat{C}\) to be a bimonad and \(F \from \cat{C} \to \cat{D}\) an oplax monoidal functor.
  An \emph{oplax monoidal right action} of \(B\) on \(F\) is an oplax natural transformation
  \(\alpha \from FB \to F\), such that the following diagrams commute:
  \[
    \begin{tikzcd}
      {FB^2} & FB & F & FB \\
      FB & F && F
      \arrow["F\mu", from=1-1, to=1-2]
      \arrow["\alpha", from=1-2, to=2-2]
      \arrow["{\alpha B}"', from=1-1, to=2-1]
      \arrow["\alpha"', from=2-1, to=2-2]
      \arrow["F\eta", from=1-3, to=1-4]
      \arrow["{\id_F}"', from=1-3, to=2-4]
      \arrow["\alpha", from=1-4, to=2-4]
    \end{tikzcd}
  \]
\end{definition}

Similarly, we could define oplax monoidal left actions.
A prime example of the latter is given by the forgetful functor \(U^B \from \cat{C}^B \to \cat{C}\) of a bimonad \(B \from \cat{C} \to \cat{C}\) together with the action
\[
  \vartheta \defeq U^B(\epsilon) \from B U^B = U^B F^B U^B \to U^B.
\]

To keep our notation concise, in the following we fix an oplax monoidal functor \(L \from \cat{C} \to \cat{C}\) with an oplax right action \(\alpha \from LB \to L\) by a bimonad \(B\from \cat{C} \to \cat{C}\) and assume that \(L\) and \(B\) are centralisable.
Their centralisers will be denoted by \((Q, \xi)\) and \((Z, \chi)\), respectively.

We think of \(\ZCat[B]{C}\) as a more general version of the Drinfeld centre
which is supposed to act on  \(\ZCat[L]{C}\) from the right.
To emphasise this, and in line with the colouring scheme of Section~\ref{sec: pivotality_of_the_df_centre}, we use black for objects in \(\cat{C}\) or its generalised Drinfeld centre and blue for objects in \(\ZCat[L]{C}\).

Consider two objects \((M, \sigma_{M,\blank}) \in \ZCat[L]{C}\) and \((X, \sigma_{X,\blank}) \in \ZCat[B]{C}\).
The action of \(B\) on \(L\), combined with the half-braidings of \(M\) and \(X\), yields a natural transformation
\begin{equation}\label{eq: right_action_of_gen_centre}
  \tikzsetnextfilename{right_action_of_H_centre_on_T_centre}%
  \input{sections/tikzfigures/right_action_of_H_centre_on_T_centre.tikz}%

\end{equation}

\begin{lemma}\label{lem: right_action_through_centralisers}
  The centre \(\ZCat[B]{C}\) acts on \(\ZCat[L]{C}\)
  from the right by tensoring the underlying objects and gluing together the half-braidings as in Equation~\eqref{eq: right_action_of_gen_centre}.
  With respect to this action, the forgetful functor \(U^{(L)} \from \ZCat[L]{C} \to \cat C \) is a strict comodule functor over \(U^{(B)} \from  \ZCat[B]{C} \to \cat C\).
\end{lemma}

A proof is given in Section~\ref{sec:proof:right_action_through_centralisers}.

We extend our colouring scheme to universal coactions and write
\begin{equation}\label{eq: colouring_scheme_universal_coactions}
  \tikzsetnextfilename{universal_coaction_centre_and_twisted_centre}%
  \input{sections/tikzfigures/universal_coaction_centre_and_twisted_centre.tikz}%

\end{equation}

The identification of \(\cat C^Z\) and \(\cat C^Q\) with the generalised Drinfeld centre and its twisted cousin suggest that \(Q\) is a comodule monad over \(Z\).
In analogy with Equation~\eqref{eq: comult_of_centraliser}, we define a candidate for the coaction of \(Q\) by
\begin{equation}\label{eq: coaction_anti_central_monad}
  \tikzsetnextfilename{coaction_anti_central_comodule_monad}%
  \input{sections/tikzfigures/coaction_anti_central_comodule_monad.tikz}%

\end{equation}

A proof of the next result may be found in Section~\ref{sec:proof:construction_anti_central_comonad}.

\begin{proposition}\label{prop: construction_anti_central_comonad}
  Let \(\alpha \from LB \to L\) be an oplax monoidal right action of a bimonad \(B \from \cat C \to \cat C\) on an oplax monoidal functor \(L \from \cat C \to \cat C\).
  Suppose furthermore that the centralisers \((Q, \xi)\) of \(L\)  and  \((Z, \chi)\) of \(B\) exist.
  The coaction  of Equation~\eqref{eq: coaction_anti_central_monad} turns \(Q\) into a comodule monad over \(Z\) such that \(\cat C^Q\) is isomorphic as a right module category over \(\cat C^Z\) to \(\ZCat[L]{C}\).
\end{proposition}

Let us apply our findings to the identity and double dual functor of a rigid category \(\cat C\).
Suppose \((Q, \xi)\) and  \((Z, \chi)\) to be the centralisers of \(\lbiDual{(\blank)}\) and \(\Id_{\cat C}\from \cat C \to \cat C\), respectively.
There is a trivial right action of the identity of \(\cat C\) on its double dual functor,
\begin{equation*}
  \id_X \from\lbiDual{(\Id_{\cat C} (X))} \to \lbiDual{X}, \qquad \qquad \text{for all } X \in \cat C.
\end{equation*}
It turns \(Q\) into a comodule monad over \(Z\) and its modules \(\cat C^Q\) are isomorphic to \(\QCat C\) as a \(\cat C^Z\)-module category.
Due to Remark~\ref{rem: right_vs_left_twisted_centres}, we can identify \(\QCat C\) with \({\ACat{\overline{C}}}^{\op} \), justifying our next definition.

\begin{definition}\label{def: anti_central_comodule_monad}
  Assume \(\lbiDual{(\blank)}, \Id_{\cat C}\from \cat C \to \cat C\) to admit centralisers \((Q, \xi)\) and \((Z, \chi)\).
  We call \(\mathfrak{Q}(\cat C) \defeq (Q, \mu^{(Q)}, \eta^{(Q)}, \delta^{(Q)})\) the \emph{anti-central} comodule monad of \(\cat C\).
\end{definition}

\subsection{The Drinfeld and anti-Drinfeld double of a Hopf monad}\label{subsec: Drinfeld_anti-Drinfeld_double}

We are now able to untangle the relationship between the various adjunctions and categories displayed in Figure~\ref{fig: twisted_centres_and_monads}.
To that end, we fix a Hopf monad \(H \from \cat V \to \cat V\) on a rigid category \(\cat V\) together with an oplax monoidal functor \(L \from \cat V^H \to \cat V^H\), a bimonad \(B \from \cat V^H\to \cat V^H\) and an oplax monoidal right action \(\alpha \from LB \to B\).
Furthermore, we assume that the cross products \(B \rtimes H\) and \(L \rtimes H\) have centralisers \((Z_H, \nu)\) and \((Q_H, \tau)\).

We start by the following observation,
extending the action of \(B\) on \(L\) to an action of the respective cross products.
It follows by a straightforward calculation.

\begin{lemma}\label{lem: cross_products_and_oplax_monoidal_actions}
  The action \(\alpha \from LB \to B\) induces an oplax monoidal action
  \begin{equation}
  \tikzsetnextfilename{extension_of_action_to_cross_products}%
  \input{sections/tikzfigures/extension_of_action_to_cross_products.tikz}%

  \end{equation}
\end{lemma}

The following variant of \cite[Theorem~7.4]{Bruguieres2012} lies at the heart of our investigation.

\begin{theorem}\label{thm: centralisablity_of_lifts}
  Both \(B, L\from \cat V^H\to \cat V^H\) admit centralisers \((Z, \chi)\) and \((Q, \xi)\)
  such that \(Z\) is a lift of \(Z_H\) as a bimonad and \(Q\) is a lift of \(Q_H\) as a comodule monad.
\end{theorem}

\begin{proof}
  By \cite[Theorem~7.4(a)]{Bruguieres2012}, we know that there are centralisers \((Q, \xi)\) and \((Z, \chi)\) of \(L\) and \(B\)
  that satisfy for all \((X, \vartheta_X), (Y, \vartheta_Y) \in \cat V^H\)
  \begin{align*}
    U^H Q (X, \vartheta_X) &= Q_H(X),   &
                                  U^H(\xi_{(X, \vartheta_X), (Y, \vartheta_Y)}) &= (U^HL(\vartheta_Y) \otimes  \id_{Q_H(X)})\tau_{X,Y}, \\
    U^H Z (X, \vartheta_X) &= Z_H(X),  &
                                 U^H(\chi_{(X, \vartheta_X), (Y, \vartheta_Y)}) &= (U^HB(\vartheta_Y) \otimes \id_{Z_H(X)})\nu_{X,Y}.
  \end{align*}
  The second and third part of the above mentioned theorem state that \(Q\) is a lift of the monad \(Q_H\) and \(Z\) is a lift of the bimonad \(Z_H\).
  It remains for us to show that the coactions of \(Q\) and \(Q_H\) are compatible with the forgetful functor \(U^H \from \cat V^H \to \cat V\).
  We fix objects \((X, \vartheta_X), (Y, \vartheta_Y) \in \cat V^H\) and \(W \in \cat V\) and compute:
  \begin{equation*}
  \tikzsetnextfilename{compatibility_coactions_cross_product}%
  \input{sections/tikzfigures/compatibility_coactions_cross_product.tikz}%

  \end{equation*}
  The uniqueness property of universal coactions as given in Lemma~\ref{lem: extended_factorisation_property} then implies that \(U^H(\delta^{(Q)}_{(X, \vartheta_X), (Y, \vartheta_Y)}) = \delta^{(Q_H)}_{X,Y}\).
  Since \(U^{H}\from \cat V^H \to \cat V\) is a strict comodule functor, the claim follows.
\end{proof}

The previous theorem together with Lemma~\ref{lem: cross_product_of_bi/comonads_is_bi/comonad} imply that we obtain a comodule monad \(D(L,H)\defeq Q \rtimes H\) over \(D(B, H) \defeq Z \rtimes H\).
The correspondence between lifts and monads given in Proposition~\ref{prop: distributive_laws_are_lifts_of_monads} yields a unique comodule distributive law
\[
  (H Q_H  \trightarrow{\;\Omega\;} Q_H H,\; H Z_H \trightarrow{\;\Lambda\;} Z_H H)
\]
such that
\begin{equation}
  D(L, H) = Q_H \circ_{\Omega} H \qquad \text{and} \qquad
  D(B, H) = Z_H \circ_{\Lambda} H.
\end{equation}

\begin{definition}\label{def: double-and-twisted-double}
  We call \(D(B,H)\) and \(D(L,H)\) the \emph{double} and \emph{twisted double} of the pairs \((B,H)\) and \((L,H)\).
\end{definition}

The relationship between doubles and generalised Drinfeld centres is explained in \cite[Proposition~7.5 and Theorem~7.6]{Bruguieres2012}.
Our next result uses the same techniques to prove how twisted doubles parameterise twisted centres.

\begin{theorem}\label{thm: twisted_doubles_vs_twisted_centres}
  The twisted double \(D(L, H)\) is a comodule monad over \(D(B,H)\) and \(\cat V^{D(L, H)}\) is isomorphic as a module category over \(\cat V^{D(B, H)}\) to
  \(\ZCat*[L]{\cat V^H}\).
\end{theorem}
\begin{proof}
  Since \(Q\) is a lift of \(Q_H\) as a comodule monad, the twisted double \(D(L, H)\) is a comodule monad over \(D(B,H)\).
  By Lemma~\ref{lem: distributive_laws_and_comparison_functor}, this implies the existence of an isomorphism of \(\cat V^{D(B, H)}\)-module categories \(K^{(\Omega)} \from \cat V^{D(L, H)} \to \left(\cat V^H\right)^Q\).
  Due to the proof of Proposition~\ref{prop: construction_anti_central_comonad} the comparison functor \(\Sigma^{(Q)}\from  \ZCat*[L]{\cat V^H} \to  \left(\cat V^H\right)^Q\) implements an isomorphism of module categories and the statement follows by considering
  \begin{equation}
    \cat V^{D(L, H)} \trightarrow{K^{(\Omega)}}
    \left(\cat V^H\right)^Q \trightarrow{E^{(Q)}}
    \ZCat*[L]{\cat V^H}.
  \end{equation}

\end{proof}

\begin{definition}\label{def: double_of_bimonad}
  Suppose \(B = \Id_{\cat C^K}\) and \(L=\lbiDual{(\blank)} \from \cat V^H \to \cat V^H\).
  We refer to \(D(H) \defeq D(B,H)\) and \(Q(H) \defeq D(L,H)\) as the \emph{Drinfeld} and \emph{anti-Drinfeld double} of \(H\).
\end{definition}

Our previous definition can be understood as an extension of the notion of the anti-Drinfeld double given by \cite{Hajac2004a} to the monadic framework.

\subsection{Pairs in involution for Hopf monads}\label{subsec: monadic_perspective_piv_structures}

For the final step in our investigation, let us consider a Hopf monad \(H \from \cat V \to \cat V\) which admits a double and anti-double.
Tracing the various identifications of the centre and anti-centre of a monoidal category given in Figure~\ref{fig: twisted_centres_and_monads}, we observe that module functors from \(\ZCat*{\cat V^H}\) to \(\QCat*{\cat V^H}\) equate bidirectionally to module functors between \(\cat V^{D(H)}\) and \(\cat V^{Q(H)}\).
In the spirit of viewing \(D(H)\) and \(Q(H)\) as `coordinate systems' of their respective modules, we want to translate such functors into comodule monad morphisms.
Our main focus here is on pivotal structures of \(\cat V^{D(H)}\).

We begin by developing the notion of pairs in involution for a Hopf monad.
Classically, pairs in involution consist of a group-like and character of a Hopf algebra, which implement the square of its antipode by their adjoint actions.

\begin{definition}\label{def: character_and_group-like}
  Let \(H \from \cat V \to \cat V\) be a Hopf monad.
  A \emph{character} of \(H\) is a module \(\beta \defeq (1, \vartheta_{\beta}) \in \cat{V}^H\), whose underlying object is the monoidal unit of \(\cat{V}\).

  A \emph{group-like} element of \(H\) is an opmonoidal natural transformation \(g \from \Id_{\cat V} \to H\).
  More explicitly, for all \(X, Y \in \cat V\)
  \begin{equation} \label{eq:2}
    \Delta^{(H)}_{X,Y}g_{X\otimes Y} = g_X \otimes g_Y \qquad \text{ and } \qquad  \varepsilon^{(H)} g_1= \id_1.
  \end{equation}

  We write \(\Char(H)\) for the characters of \(H\) and \(\Gr(H)\) for its group-likes.
\end{definition}

Note that the characters \(\Char(H)\) of a Hopf-monad \(H \from \cat V \to \cat V\) form a monoid and, by Lemma~\cite[Lemma~3.21]{Bruguieres2007}, the set \(\Gr(H)\) of group-like elements bears a group structure.

Furthermore, the group-likes of a Hopf monad \(H\) act on it by conjugation.
We recall this construction based on \cite[Section~1.4]{Bruguieres2007}.
Given a natural transformation \(g \from \Id_{\cat V} \to H\), we define the \emph{left} and \emph{right regular action} of \(g\) on \(H\) to be the natural transformations defined for every \(X \in \cat V\) by
\begin{align}\label{eq: left_and_right_regular_action}
  L_{g, X} &\defeq H(X) \trightarrow{g_{H(X)}} H^2(X) \trightarrow{\mu^{(H)}_X} H(X), \\
  R_{g, X} &\defeq H(X) \trightarrow{H(g_X)} H^2(X) \trightarrow{\mu^{(H)}_X} H(X).
\end{align}

Before we state our next definition, we set for all \(X,Y, W \in \cat V\)
\begin{equation}\label{eq: iterated_coproducts}
  \Delta^{(H)}_{X,Y,W} \defeq(\Delta^{(H)}_{X,Y} \otimes \id_{H(W)}) \Delta^{(H)}_{X\otimes Y, W} = (\id_{H(X)}\otimes \Delta^{(H)}_{Y, W}) \Delta^{(H)}_{X, Y \otimes W}.
\end{equation}

\begin{definition}\label{def: adjoint_action}
  Every group-like \(g \in \Gr(H)\) and character \(\beta \in \Char(H)\) of a Hopf monad \(H \from \cat V \to \cat V\) give rise to natural transformations
  \begin{align}\label{eq: adjoint_action}
    \Ad_{g, X} &\defeq L_{g,X}R_{g^{-1}, X} \from H(X) \to H(X), &\text{ for all } X \in \cat V, \\
    \Ad_{\beta, X} &\defeq (\vartheta_{\beta}\otimes \id_{H(X)} \otimes \vartheta_{\lDual \beta}) \Delta^{(H)}_{1, X,1} \from H(X) \to H(X), &\text{ for all } X \in \cat V.
  \end{align}
  called the \emph{adjoint actions} of \(g\) and \(\beta\) on \(H\), respectively.
\end{definition}

To define pairs in involution, we need the `square of the antipode'.
This notion was developed in \cite[Section~7.3]{Bruguieres2007}.

\begin{definition} \label{def: s-squared}
  Suppose \(\phi \from \Id_{\cat V} \to \lbiDual{(\blank)}\) to be a pivotal structure on \(\cat V\) and let \(H \from \cat V \to \cat V\) be a Hopf monad.
  The \emph{square of the antipode} of \(H\) is a natural transformation \(S^{2}\from H \to H\)
  that is defined for every \(X \in \cat V\) by
  \begin{equation}
    S^2_X \defeq \phi^{-1}_{H(X)} s^l_{\lDual{H(X)}} H(\lDual{s^l_X}) H(\phi_X),
  \end{equation}
  where \(s^l\) is the left antipode of \(H\),
  see Equation~\eqref{eq: left_and_right_antipode} and \cite[Section~3.3]{Bruguieres2007}.
\end{definition}

Analogous to the Hopf algebraic case, we state the following:

\begin{definition}\label{def: pair_in_involution_Hopf_monad}
  Let \(H \from \cat{V} \to \cat{V}\) be a Hopf monad,
  and \(\phi \from \Id_{\cat V} \to \lbiDual{(\blank)}\) a pivotal structure.
  A \emph{pair in involution} of \(H\) and \(\phi\) comprises a group-like \(g \in \Gr(H)\) and a character \(\beta \in \Char(H)\), such that for all \(X \in \cat{V}\)
  \begin{equation}\label{eq: antipode_condition}
    Ad_{g, X}=  Ad_{\beta,X}S^{2}_{X}.
  \end{equation}
  We write \((g, \beta) \in \mathrm{PI}_H^{\phi}\).
\end{definition}

To prove that pairs in involution correspond to certain pivotal structures on the Drinfeld centre of \(\cat V^H\), we need two technical results.
The first one is classical; for a proof see for example \cite[Lemmas~1.2 and~1.3]{Bruguieres2007}.

\begin{lemma} \label{lem: sharp-lift}
  Let \(H\) be a monad with associated forgetful functor \(U^H \from  \cat V^H \to \cat V\).
  Suppose that \(F, G \from \cat{V} \to \cat{W}\) are functors, for some category \(\cat{W}\).
  There is a canonical bijection
  \begin{equation}\label{eq: sharp-lift}
    (\blank)^{\sharp}\from \Nat(F, GH) \to \Nat(FU^H, GU^H), \qquad f \mapsto f^{\sharp},
  \end{equation}
  where \(f^{\sharp}_{(M, \vartheta_M)} = G(\vartheta_M) f_M\).
\end{lemma}

The next lemma is a variant of \cite[Lemma~7.5]{Bruguieres2007}.

\begin{lemma} \label{lem: s2-is-conjugate}
  Let \(\phi \from \Id_{\cat V} \to \lbiDual{(\blank)}\) be a pivotal structure on \(\cat V\) and \(H \from \cat V \to \cat V\) a Hopf monad.
  For any group-like \(g \in \Gr(H)\) and character \(\beta \in \Char(H)\) the following are equivalent:
  \begin{thmlist}
    \item The arrows \(g\) and \(\beta\) form a pair in involution of \(H\) and \(\phi\).
    \item The natural morphism \(\phi g^{\sharp}\in \Nat(U^H, U^H)\) lifts to \(\Nat(\Id_{\cat V^H}, \beta \otimes \lbiDual{(\blank)}\otimes \lDual \beta )\).
  \end{thmlist}
\end{lemma}
\begin{proof}
  Consider a module \((M, \vartheta_M) \in \cat V^H\).
  By \cite[Theorem~3.8(a)]{Bruguieres2007} and the definition of \(S^2\), the action on \(\lbiDual M\) is given by
  \begin{equation*}
    \vartheta_{\lbiDual M}
    = \lbiDual{\vartheta_M} s^l_{\lbiDual{H(M)}} H(\lDual{s^l_M})
    = \phi_M \vartheta_M S^2_M H(\phi^{-1}_M)
  \end{equation*}
  and therefore we have
  \begin{equation*}
    \vartheta_{\beta \otimes \lbiDual M \otimes \lDual \beta}
    = (\vartheta_{\beta} \otimes \vartheta_{\lbiDual{M}} \otimes \vartheta_{\lDual \beta}) \Delta^{(H)}_{1,M,1}
    = (\vartheta_{\beta} \otimes \phi_M \vartheta_M S^2_M H(\phi^{-1}_M) \otimes \vartheta_{\lDual \beta}) \Delta^{(H)}_{1,M,1}.
  \end{equation*}
  By definition \(\phi g^{\sharp}\) lifts to a natural transformation from \(\Id_{\cat V^H}\) to \(\beta \otimes \lbiDual{(\blank)} \otimes \lDual \beta\),  if and only if for any \(H\)-module \((M, \vartheta_M)\), we have
  \begin{equation} \label{eq: sharpbeta-lifts}
    (\phi g^{\sharp})_M \vartheta_M = \vartheta_{\beta \otimes \lbiDual{M} \otimes \lDual \beta}H( (\phi g^{\sharp})_M).
  \end{equation}
  Let us now successively simplify both sides of the equation.
  Using the naturality of \(g \from \Id_{\cat V} \to H\), the fact that \(\vartheta_M\) is an action and the definition of \(g^{\sharp}\) as given in Lemma~\ref{lem: sharp-lift}, we can rewrite the left hand side of the equation as
  \begin{equation*}
    (\phi g^{\sharp})_M \vartheta_M
    = \phi_M \vartheta_M g_M \vartheta_M
    = \phi_M \vartheta_M H(\vartheta_M) g_{H(M)}
    = \phi_M \vartheta_M \mu^{(H)}_M g_{H(M)}.
  \end{equation*}
  Similarly, we simplify the right-hand side to
  \begin{align*}
    &\vartheta_{\beta \otimes \lbiDual{M} \otimes \lDual \beta}H( (\phi g^{\sharp})_M)
    = (\vartheta_{\beta} \otimes \phi_M \vartheta_M S^2_M H(\phi^{-1}_M) \otimes \vartheta_{\lDual \beta}) \Delta^{(H)}_{1,M,1} H( (\phi g^{\sharp})_M) \\
    & = (\vartheta_{\beta} \otimes \phi_M \vartheta_M  S^2_M H(\phi^{-1}_M) H( (\phi g^{\sharp})_M) \otimes \vartheta_{\lDual \beta}) \Delta^{(H)}_{1,M,1} \\
    & = (\vartheta_{\beta} \otimes \phi_M \vartheta_M  S^2_M H(\vartheta_M g_M ) \otimes \vartheta_{\lDual \beta}) \Delta^{(H)}_{1,M,1}
     = (\vartheta_{\beta} \otimes \phi_M \vartheta_M   H(\vartheta_M g_M ) S^2_M \otimes \vartheta_{\lDual \beta}) \Delta^{(H)}_{1,M,1} \\
    & = \phi_M \vartheta_M   H(\vartheta_M g_M ) (\vartheta_{\beta} \otimes  \id_{H(M)}\otimes \vartheta_{\lDual \beta}) \Delta^{(H)}_{1,M,1}S^2_M
     = \phi_M \vartheta_{M}\mu^{(H)}_M H(g_M) \Ad_{\beta,M} S^2_M .
  \end{align*}
  Using the fact that \(\phi\) is an isomorphism, Equation~\eqref{eq: sharpbeta-lifts} can thus be restated as
  \begin{align*}
    & & \vartheta_M \mu^{(H)}_M g_{H(M)} &= \vartheta_M\mu^{(H)}_M H(g_M)  \Ad_{\beta,M} S^2_M \\
    \iff & & \vartheta_M L_{g, M} &= \vartheta_M R_{g, M} \Ad_{\beta, M} S^2_M
  \end{align*}
  Since \((\blank)^{\sharp}\) is a bijection,
  the above equation is equivalent to
  \(L_{g, M} = R_{g, M} \Ad_{\beta, M} S^2_M\)
  by Lemma~\ref{lem: sharp-lift}.
  We conclude the proof by multiplying both sides with \(R_{g^{-1}, M}\).
\end{proof}

The previous lemma leads to an identification of pairs in involution of \(H\) and \(\phi\)
with certain quasi-pivotal structures on \(\cat V^H\).

\begin{proposition}\label{prop: piis_are_qpiv}
  Suppose \(H \from \cat V \to \cat V\) to be a Hopf monad
  and \(\phi \from \Id_{\cat V} \to \lbiDual{(\blank)}\) a pivotal structure on \(\cat V\).
  Then $(H, \phi)$ admits a pair in involution if and only if there exists a quasi-pivotal structure on $\cat V^{H}$ whose underlying invertible object is a character $\beta \in \Char(H)$.
\end{proposition}
\begin{proof}
  We proceed analogous to \cite[Proposition~7.6]{Bruguieres2007}.
  Suppose \((g, \beta) \in \mathrm{PI}_{H}^{\phi}\).
  By the previous lemma, \(\phi g^{\sharp}\) lifts to a natural isomorphism
  \begin{equation*}
    \rho_{\beta,X}\from X \to \beta \otimes \lbiDual X \otimes \lDual \beta, \qquad \text{ for all } X \in \cat V^H.
  \end{equation*}
  Since \(\phi\) is monoidal by definition and \(g^{\sharp}\) is monoidal by virtue of \(g\) being a group-like, see \cite[Lemma~3.20]{Bruguieres2007}, we obtain a quasi-pivotal structure
  \(\rho_{\beta} \from \Id_{\cat V^H} \to \beta \otimes \lbiDual{(\blank)}\otimes \lDual \beta\).

  On the other hand, let \((\beta,\rho_{\beta})\) be a quasi-pivotal structure, where \(\beta \in \Char(H)\)  is a character.
  Since the forgetful functor \(U^H\) is strong monoidal and thus
  \begin{equation*}
    U^H(\beta\otimes \lbiDual{(\blank)}\otimes \lDual \beta) = U^H(\lbiDual{(\blank)})= \lbiDual{(U^H(\blank))},
  \end{equation*}
  there exists a monoidal natural transformation
  \begin{equation*}
    \phi^{-1}_{U^H(X)}U^H(\rho_{\beta,X}) \from U^H(X) \to U^H(X), \qquad \qquad \text{for all } X \in \cat V^H.
  \end{equation*}
  Again, we apply \cite[Lemma~3.20]{Bruguieres2007} and obtain a unique group-like \(g\in \Gr(H)\) such that \(g^{\sharp} = \phi^{-1}_{U^H(X)}U^H(\rho_{\beta,X})\).
  As \(\phi g^{\sharp} = U^H(\rho_{\beta})\) lifts to the quasi-pivotal structure \((\beta,\rho_{\beta})\) on \(\cat V^H\), Lemma~\ref{lem: s2-is-conjugate} implies that \((g, \beta) \in \mathrm{PI}_{H}^{\phi}\).
\end{proof}

Let us now study a variant of \cite[Lemma~2.9]{Bruguieres2007}.

\begin{proposition}\label{prop: functors_as_morphisms_of_monads}
  Assume \(K, C \from \cat M \to \cat M\)  to be two comodule monads over a bimonad \(B \from \cat C \to \cat C\).
  There is a bijective correspondence between morphisms of comodule monads \(f \from K \to C\) and strict module functors \(F \from \cat M^C \to \cat M^K\) such that \(U^K F = U^C\).
\end{proposition}
\begin{proof}
  As shown for example in \cite[Lemma~1.7]{Bruguieres2007}, any functor  \(F \from \cat M^C \to \cat M^K\) with  \(U^K F = U^C\) is `induced' by a unique  morphism of monads \(f \from K \to C\).
  That is, \(F\) is the identity on morphisms and on objects it is defined by
  \begin{equation*}
    F(M, \vartheta_M) = (M, \vartheta_M f_M), \qquad \qquad \text{ for all } (M, \vartheta_M) \in \cat M^C.
  \end{equation*}
  It remains to show that \(f\) is a morphism of comodules if and only if \(F\) is a strict module functor in the sense of Definition~\ref{def: functor_mod_cat}.
  Let \((M, \vartheta_M)\in \cat M^C\) and \((X, \vartheta_X) \in \cat C^B\).
  We compute
  \begin{align*}
    F((M, \vartheta_M) \ract (X, \vartheta_X))
    &= (M \ract X, (\vartheta_M \ract \vartheta_X)\delta^{(C)}_{M,X} f_{M\ract X}), \\
    F(M, \vartheta_M) \ract  (X, \vartheta_X)
    &= (M \ract X, (\vartheta_M \ract \vartheta_X) (f_M \ract \id_{B(X)}) \delta^{(K)}_{M, X}).
  \end{align*}
  According to \cite[Lemma~1.4]{Bruguieres2007}, these modules coincide if and only if
  \begin{equation*}
    \delta^{(C)}_{M,X} f_{M \ract X} = (f_M \ract \id_{B(X)}) \delta^{(K)}_{M, X},
  \end{equation*}
  which is exactly the condition for \(f\) to be a comodule morphism.
\end{proof}

The above result readily implies the desired monadic version of Theorem~\ref{thm: equiv_in_Hopf_setting}.

\begin{theorem}\label{thm: equivalence_in_the_monad_setting}
  Let \(\cat{V}\) be a rigid monoidal category,
  and suppose that \(H \from \cat V \to \cat V\) is a Hopf monad
  that admits a double \(D(H)\) and anti-double \(Q(H)\).
  The following statements are equivalent:
  \begin{thmlist}
    \item The monoidal unit \(1 \in \cat V\) lifts to a module over \(Q(H)\).
    \item The Drinfeld double and anti-Drinfeld double of \(H\) are isomorphic as comodule monads.
    \item There is an isomorphism of monads \(g \from Q(H) \to D(H)\).
  \end{thmlist}
  Additionally, if \(\cat V\) is pivotal with pivotal structure \(\phi\),
  any of the above statements hold if and only if \(H\) and \(\phi\) admit a pair in involution.
\end{theorem}
\begin{proof}
  \((i) \!\! \implies \!\! (ii)\): suppose \(\omega \in \QCat*{\cat V^H}\) with \(U^{Q(H)}(\omega) = 1\).
  As shown in Equation~\eqref{eq: objs_as_mod_funs}, it induces a functor of module categories
  \begin{equation*}
    \omega \otimes \blank  \from \cat V^{D(H)} \to \cat V^{Q(H)}.
  \end{equation*}
  Since \(U^{Q(H)}(\omega)= 1\in \cat V\), we can apply Proposition~\ref{prop: functors_as_morphisms_of_monads} and obtain that \(Q(H)\) and \(D(H)\) are isomorphic as comodule monads.

  It immediately follows that \((ii)\) implies \((iii)\); we proceed with \((iii) \!\! \implies \!\! (i)\): consider an isomorphism of monads \(g \from Q(H) \to D(H)\).
  It gives rise to a functor \(G \from \cat V^{D(H)} \to \cat V^{Q(H)}\) that, on objects, is defined by
  \begin{equation*}
    G(M, \vartheta_M) = (M, \vartheta_M g_M), \qquad \qquad \text{ for all } (M, \vartheta_M) \in \cat V^{D(H)}.
  \end{equation*}
  We compose \(G\) with the inverse of the comparison functor \(E^{(Q(H))} \from \cat V^{Q(H)} \to \QCat*{\cat{V}^H}\), defined in Equation~\eqref{eq: inverse_comparison_functor}, and see that there exists an object
  \begin{equation*}
    1^{(Q)} \defeq E^{(Q(H))}G(1) \in \QCat*{\cat{V}^H}
  \end{equation*}
  whose underlying object is the unit of \(\cat V\).

  Now let \((\cat{V}, \phi)\) be pivotal.
  By Lemma~\ref{lem: quasi_piv_vs_inv_ayds}, lifts of \(1\in \cat V\) to the dual of the anti-center \(\QCat*{\cat V^H}\) are in correspondence with quasi-pivotal structures
  \((\beta, \rho_{\beta})\), where \(\beta \in \Char(H)\).
  By Proposition~\ref{prop: piis_are_qpiv} such a quasi-pivotal structure exists if and only if
  there exists a pair in involution for \(H\) and \(\phi\).
\end{proof}

As a corollary, we can determine whether a category is pivotal in terms of monad isomorphisms between the central and anti-central monad.
For a category \(\cat{C}\),
recall Definition~\ref{def: central_bimonad} of its central Hopf monad,
and Definition~\ref{def: anti_central_comodule_monad} of its anti-central comodule monad.

\begin{corollary}\label{cor: pivotal_from_central_anti_central}
  Let \(\cat{C}\) be a rigid monoidal category.
  If \(\cat{C}\) admits a central Hopf monad \(\mathfrak{D}(\cat{C})\)
  and an anti-central comodule monad \(\mathfrak{Q}(\cat{C})\),
  then it is pivotal if and only if \(\mathfrak{D}(\cat{C}) \cong \mathfrak{Q}(\cat{C})\) as monads.
\end{corollary}
\begin{proof}
  We consider the identity \(\Id_{\cat C} \from \cat C \to \cat C\) as a Hopf monad.
  Its Drinfeld  and anti-Drinfeld double are \(D(\Id_{\cat C}) = \mathfrak D(\cat C) \rtimes \Id_{\cat C}\) and \(Q(\Id_{\cat C}) = \mathfrak Q(\cat C) \rtimes \Id_{\cat C}\).
  From here it follows that \(D(\Id_{\cat C}) = \mathfrak D(\cat C)\) and similarly \(Q(\Id_{\cat C}) = \mathfrak Q(\cat C)\).
  The proof is concluded by Theorem~\ref{thm: equivalence_in_the_monad_setting}.
\end{proof}

%%% Local Variables:
%%% mode: latex
%%% TeX-master: "../main"
%%% End:

\appendix
\section{Additional proofs}
\setcounter{section}{1}

\subsection{Proof of Lemma~\ref{lem: pivotal_structure_from_entwinement}} \label{sec:proof:pivotal_structure_from_entwinement}

\begin{proof}
  Fix an object \(A=(\alpha, \sigma_{\alpha, \blank}) \in \ACat C\) such that \(\alpha\) is invertible in \(\cat C\), and write \(\Omega = (\omega, \sigma_{\omega, \blank}) \in \QCat C\) for its left dual.
  Furthermore, we assume \(X \in \ZCat C\) to be any object in the Drinfeld centre of \(\cat C\).
  We note that for any \(Y \in \cat C\) a variant of the \YangBaxter identity holds:
  \begin{equation}\label{eq: auxiliary_YB_identity}
  \tikzsetnextfilename{auxiliary_YB_identity}%
  \input{sections/tikzfigures/auxiliary_YB_identity.tikz}%

  \end{equation}

  The above identity combined with those  displayed in Diagram~\eqref{eq: unit_braiding_coev} proves that \(\rho_{A,X} \from X \to \lbiDual X\) is a morphism in the Drinfeld centre of \(\cat C\):
  \begin{equation}
  \tikzsetnextfilename{pivotality_morphism_well_defined}%
  \input{sections/tikzfigures/pivotality_morphism_well_defined.tikz}%

  \end{equation}
  Since the forgetful functor \(U^{(Z)} \from \ZCat C \to \cat C\) is conservative and \(\rho_{A, X}\) is a composite of isomorphisms in \(\cat C\), it is an isomorphism in the centre \(\ZCat C\).

  The naturality of the half-braidings implies that \(\rho_{A}\) is natural as well.
  \begin{center}
  \tikzsetnextfilename{pivotality_morphisms_naturality}%
  \input{sections/tikzfigures/pivotality_morphisms_naturality.tikz}%

  \end{center}
  Lastly, the natural isomorphism \(\rho_A \from \Id_{\ZCat C} \to \lbiDual{(\blank)}\) being monoidal is established by the hexagon identities, as is made evident by the next diagram.
  \begin{center}
  \tikzsetnextfilename{pivotality_morphisms_monoidal}%
  \input{sections/tikzfigures/pivotality_morphisms_monoidal.tikz}%

  \end{center}
\end{proof}

\subsection{Proof of Lemma~\ref{lem: right_action_through_centralisers}} \label{sec:proof:right_action_through_centralisers}

\begin{proof}
  We proceed as in \cite[Proposition~5.9]{Bruguieres2012} and fix objects \((M, \sigma_{M, \blank} ) \in \ZCat[L]{C}\) and \((X, \sigma_{X,\blank}) \in \ZCat[B]{C}\).
  The compatibility of the half-braiding of \(M \otimes X\) with the unit of \(\cat C\) is a short computation:
  \begin{equation*}
  \tikzsetnextfilename{unit_identity_from_bimonad}%
  \input{sections/tikzfigures/unit_identity_from_bimonad.tikz}%

  \end{equation*}

  Similarly, we verify the hexagon axiom:
  \begin{equation*}
  \tikzsetnextfilename{hexagon_identity_from_bimonad}%
  \input{sections/tikzfigures/hexagon_identity_from_bimonad.tikz}%

  \end{equation*}

  The compatibility of the action \(\alpha \from LB \to L\) with the multiplication and unit of \(B\) asserts that \(\ZCat[L]{C}\) is a right module of the generalised Drinfeld centre.

  By construction, we have for all \((M, \sigma_{M, \blank} ) \in \ZCat[L]{C}\) and \((X, \sigma_{X,\blank}) \in \ZCat[B]{C}\)
  \begin{equation*}
    U^{(L)}((M, \sigma_{M,\blank}) \ract (X, \sigma_{X,\blank}))
    = M \otimes X
    = U^{(L)}(M, \sigma_{M,\blank}) \otimes U^{(B)} (X, \sigma_{X,\blank}).
  \end{equation*}
  Thus, \(U^{(L)}\) is a strict comodule functor over \(U^{(B)}\).
\end{proof}

\subsection{Proof of Proposition~\ref{prop: construction_anti_central_comonad}} \label{sec:proof:construction_anti_central_comonad}

\begin{proof}
  By Remark~\ref{rem: adjoint_from_inverse_of_the_comparison_functor} and Proposition~\ref{prop: inverse_comparison_functor} we have monadic adjunctions
  \begin{equation*}
    \adj{F^{(B)}}{U^{(B)}}{\cat C}{\ZCat[B]{C}} \qquad \text{ and } \qquad \adj{F^{(L)}}{U^{(L)}}{\cat C}{\ZCat[L]{C}}
  \end{equation*}
  which, due to \cite[Remark~5.13]{Bruguieres2012}, give rise to the bimonad \(Z\) and monad \(Q\), respectively.
  Lemma~\ref{lem: right_action_through_centralisers} shows that \(U^{(L)}\) is a strict comodule functor over \(U^{(B)}\).
  A straightforward generalisation of \cite[Remark~3.3]{Bruguieres2012} shows that \(Q\) is a comodule monad over \(B\);
  the coaction \(\lambda \from Q(\blank \otimes \blank) \to Q(\blank) \otimes Z(\blank)\) implementing the action of \(\cat{C}^{Z}\) on \(\cat{C}^{Q}\) is for all \(X, Y \in \cat C\) given by
  \begin{equation}
    \lambda_{X, Y} = \vartheta_{Q(X) \otimes Z(Y)} Q(\eta^{(Q)}_{X} \otimes \eta^{(Z)}_{Y}).
  \end{equation}
  For a proof of this fact, we refer the reader to \cite{halbig24}.

  By using the relation between universal coactions and half-braidings, explained in Equation~\eqref{eq: free_half_braiding}, and  applying the hexagon identity we compute:
  \begin{equation}
  \tikzsetnextfilename{uniqueness_of_coaction}%
  \input{sections/tikzfigures/uniqueness_of_coaction.tikz}%

  \end{equation}
  The uniqueness property of universal coactions implies \(\lambda= \delta^{(Q)}\).

  It remains to show that \(\cat C^{Q}\) and \(\ZCat[L]{C}\) are isomorphic as modules over \(\cat C^{Z}\).
  In \cite{halbig24} it is shown that the comparison functor
  \(\Sigma^{(Z)} \from \ZCat[B]{C} \to \cat C^{Z}\) is strong monoidal,
  and that \(\Sigma^{(Q)} \from \ZCat[L]{C} \to \cat C^{Q}\) is a strong comodule functor over it.
  Furthermore, due to Proposition~\ref{prop: inverse_comparison_functor}, both \(\Sigma^{(Z)}\) and \(\Sigma^{(Q)}\)  admit inverses
  \begin{equation*}
    E^{(Z)} \from \cat C^{Z} \to \ZCat[B]{C} \qquad \text{ and } \qquad
    E^{(Q)} \from \cat C^{Q} \to \ZCat[L]{C}.
  \end{equation*}
  Using that \(E^{(Z)}\) is monoidal as well, we identify the right action of \(\ZCat[B]{C}\) on \(\ZCat[L]{C}\) with a right action \(\bract\from\ZCat[L]{C} \times \cat C^Z \to \ZCat[L]{C} \) of \(\cat C^Z\) by setting
  \begin{equation*}
    \ZCat[L]{C} \times \cat C^Z \trightarrow{\Id \times E^{(Z)}}
    \ZCat[L]{C} \times \ZCat[B]{C} \trightarrow{(\blank)\ract(\blank)}
    \ZCat[L]{C}.
  \end{equation*}
  For any \(M \in \ZCat[L]{C}\) and \(X \in\ZCat[L]{C}\)  we have
  \begin{equation*}
    \Sigma^{(Q)} (M\bract X)
    =
    \Sigma^{(Q)} (M \ract E^{(Z)}(X))
    \trightarrow{\delta^{\Sigma^{(Q)}}} \Sigma^{(Q)} (M) \ract \Sigma^{(Z)}E^{(Z)}(X) = \Sigma^{(Q)} (M) \ract X
  \end{equation*}
  and therefore \(\Sigma^{(Q)} \from \ZCat[L]{C} \to \cat C^{Q}\) is an isomorphism of module categories.
\end{proof}

\bibliographystyle{alpha}
\footnotesize\bibliography{main}\normalsize

\end{document}